\documentclass[preprint]{elsarticle}

\usepackage{lineno,hyperref}
\modulolinenumbers[5]

\journal{Arxiv}
\usepackage{multirow}
\usepackage{etex}
\usepackage{amsfonts}
\usepackage{mathrsfs}
\usepackage{comment}
\usepackage{tablefootnote}
\usepackage{float}
\usepackage{algorithm}
\usepackage{algpseudocode}
\usepackage{todonotes}
\usepackage{amsmath}
\usepackage{longtable}
\usepackage{threeparttable}
\usepackage{caption}
\usepackage{pdflscape}
\usepackage{subfigure}
\usepackage{ulem}

\usepackage{xspace}
\newcommand\BRFrag{battery-restricted fragment\xspace}

\usepackage{microtype}
\usepackage{amsthm}
\usepackage{amsmath}
\DeclareMathOperator*{\minimize}{minimize}
\usepackage{amssymb}
\usepackage{graphics}
\usepackage{stmaryrd}
\usepackage{color,xcolor}
\usepackage{soul}
\soulregister{\cite}7
\soulregister{\citep}7
\soulregister{\ref}7

\usepackage{graphicx}
\usepackage{geometry}
\usepackage{footnote}
\makesavenoteenv{table}
\newcommand{\beq}{\begin{equation}}
\newcommand{\eeq}{\end{equation}}

\theoremstyle{definition}
\newtheorem{exmp}{Example}
\theoremstyle{remark}

\theoremstyle{TH}
\newtheorem{theorem}{Theorem}
\newtheorem{definition}{Definition}
\newtheorem{proposition}{Proposition}
\newtheorem{lemma}{Lemma}

\makeatletter 
\def\@makecaption#1#2{%
  \vskip\abovecaptionskip
  \sbox\@tempboxa{#1 #2}%
    {\bfseries #1} #2\par
  \vskip\belowcaptionskip}
\makeatother
\usepackage{booktabs}
\usepackage[justification=centering]{caption}

\usepackage{hyperref}
\hypersetup{
    colorlinks=true,
    linkcolor=blue,
    filecolor=gray,      
    urlcolor=blue,
    citecolor=blue,
}
\biboptions{authoryear}
\begin{document}
\begin{frontmatter}

\title{A Branch-and-Price Algorithm for the Electric Autonomous Dial-A-Ride Problem}

\author[address1,address5]{Yue Su \corref{correspondingauthor}}
\ead{yue.su@enpc.fr}
\author[address2]{Nicolas Dupin}
\ead{nicolas.dupin@univ-angers.fr}
\author[address3]{Sophie N. Parragh}
\ead{sophie.parragh@jku.at}
\author[address4,address5]{Jakob Puchinger}
\ead{jpuchinger@em-normandie.fr}

\address[address1]{CERMICS, Ecole des Ponts ParisTech, 77420 Champs-sur-Marne, France.}
\address[address2]{Univ Angers, LERIA, SFR MATHSTIC, F-49000 Angers, France.}
\address[address3]{Institute of Production and Logistics Management, Johannes Kepler University Linz, 4040, Linz, Austria.}
\address[address4]{EM Normandie Business School, Métis Lab, 92110, Clichy, France.}
\address[address5]{Université Paris-Saclay, CentraleSupélec, Laboratoire Génie Industriel, 91190, Gif-sur-Yvette, France.}

\cortext[correspondingauthor]{Corresponding author: yue.su@enpc.fr}

\begin{abstract} 
The Electric Autonomous Dial-A-Ride Problem (E-ADARP) consists in scheduling a fleet of electric autonomous vehicles to provide ride-sharing services for customers that specify their origins and destinations. The E-ADARP differs from the classical DARP in two aspects: (i) a weighted-sum objective that minimizes both total travel time and total excess user ride time; (ii) the employment of electric autonomous vehicles and a partial recharging policy. 
This paper presents a highly-efficient labeling algorithm, which is integrated into Branch-and-Price (B\&P) algorithms to solve the E-ADARP. To handle (i), we introduce a fragment-based representation of paths. A novel approach is invoked to abstract fragments to arcs while ensuring excess-user-ride-time optimality. 
We then construct a new graph that preserves all feasible routes of the original graph by enumerating all feasible fragments, abstracting them to arcs, and connecting them with each other, depots, and recharging stations in a feasible way. 
On the new graph, partial recharging (ii) is tackled exactly by tailored Resource Extension Functions (REFs). We apply strong dominance rules and constant-time feasibility checks to compute the shortest paths efficiently. These methods construct the first labeling algorithm that can deal with minimizing (excess) user ride time. 
In the computational experiments, the B\&P algorithm achieves optimality in 71 out of 84 instances. Remarkably, among these instances, 50 were solved optimally at the root node without branching. We identify 26 new best solutions, improve 30 previously reported lower bounds, and provide 17 new lower bounds for large-scale instances with up to 8 vehicles and 96 requests. In total 42 new best solutions are generated on previously solved and unsolved instances. In addition, we analyze the impact of incorporating the total excess user ride time within the objectives and allowing unlimited visits to recharging stations. The following managerial insights are provided: (1) solving a weighted-sum objective function can significantly enhance the service quality, while still maintaining operational costs at nearly optimal levels, (2) the relaxation on charging visits allows us to solve all instances feasibly and reduces the average solution cost by 0.65\%. We also derive good upper bounds on the maximum recharging visits per station for considered instance sets. 
\end{abstract}



\end{frontmatter}

\section{Introduction}\label{intro}
The Dial-A-Ride Problem (DARP) consists in designing minimum-cost routes by scheduling a fleet of vehicles to serve a set of customers who specify their origins and destinations \citep{cordeau2007dial}. For each customer request, a time window is defined on either the origin or the destination. The DARP was first introduced in the context of providing door-to-door service for handicapped individuals, e.g., \cite{madsen1995heuristic, toth1996fast}. In recent years, the concept of the DARP has been extended to adopt requests from normal individuals and provide them with ride-sharing services. Many demand-responsive systems have been constructed, such as the mobile-based app of BlaBlaCar in France and Didi Hitch in China \citep{jin2018ridesourcing}. With the blooming of on-demand ride-sharing services, considerable attention has arisen to solving the DARP and its variants  \citep{cordeau2007dial,ho2018survey}. The DARP is a generalization of several NP-hard problems such as the Pickup and Delivery Vehicle Routing Problem (PDPVRP) and the Vehicle Routing Problem with Time Windows (VRPTW) and is therefore very difficult to solve to optimality using exact methods. Consequently, only a few studies have proposed exact methods, e.g., \cite{ cordeau2006branch,braekers2014exact,gschwind2015effective}. The DARP is even more challenging than the PDPVRP and the VRPTW, as user inconvenience needs to be considered while minimizing the operational cost \citep{cordeau2003tabu}. The classical DARP model imposes a maximum user ride time constraint on every user request to maintain a certain level of service quality. Due to this constraint in combination with time windows, scheduling service start times as early as possible does not necessarily result in a feasible schedule for a given sequence of pickup and delivery locations, given that one exists. On the contrary, allowing delays in the service start time may help to eliminate unnecessary waiting time for succeeding nodes and, as such, reduce the user ride time.
Heuristic solution methods for the DARP usually invoke the ``eight-step" procedure of \cite{cordeau2003tabu}, which composes a feasible schedule with the latest possible service start time at the origin depot and subsequently delays service start times at pickup nodes using the notion of forward time slack \citep{savelsbergh1992vehicle} while respecting maximum user ride time constraints. As a result, the complexity of the route evaluation in the DARP rises substantially. 

This article addresses the Electric Autonomous DARP (E-ADARP), which was first formulated by \cite{bongiovanni2019electric}. 
Although the E-ADARP shares some constraints with the classical DARP \citep{cordeau2007dial}, the E-ADARP contains other problem-specific features that derive from the electric and autonomous nature of vehicles. First, the use of an electric vehicle fleet necessitates the calculation of State of Charge (SoC) when a vehicle arrives at a node in order to prevent vehicles from running out of charge. Vehicles make detours to recharging stations when necessary, and partial recharging is allowed at recharging stations. 
Second, considering the autonomy of vehicles removes the maximum route duration constraints and makes the E-ADARP more difficult to solve. Also, the autonomy of vehicles offers the possibility to operate vehicles in a non-stop manner. As autonomous vehicles need to continuously relocate during their non-stop service, the destination depot is no longer predefined. Another important feature is that the E-ADARP employs a weighted-sum objective, including total travel time and total excess user ride time. Incorporating the total excess user ride time into the objective function offers tangible advantages for ride-sharing service providers. It enables them to greatly improve service quality (i.e., to reduce total excess user ride time) while yielding very small increases in operational costs (i.e., total travel time). However, solving such a weighted-sum objective function is challenging, as it requires obtaining the excess-user-ride-time optimal schedule for an E-ADARP route. 
This task is even more complicated when designing a labeling algorithm to solve column-generation subproblems of the E-ADARP, as each E-ADARP partial path is not fixed until it is extended to the sink node (i.e., destination depot). In this case, one must decide all excess-user-ride-time optimal schedules from battery-feasible schedules along the extension.
In this paper, we first propose a Column Generation (CG) algorithm relying on a highly-efficient labeling algorithm to solve the E-ADARP. Additionally, we consider adding valid cuts to strengthen the quality of lower bounds. Then, we integrate the developed CG algorithm into the B\&P scheme to solve the E-ADARP exactly.

The contributions of our work are summarized as follows: 
\begin{itemize}
    \item \textbf{From a theoretical perspective:} We propose 
    the first labeling algorithm for a path-based formulation of the DARP/E-ADARP that can deal with minimizing the total (excess) user ride time. 
    We first present a fragment-based representation of the resource-constrained paths. On each fragment, we apply a new approach to determine the minimum excess user ride time and abstract the fragment to an arc that captures all excess-user-ride-time optimal schedules. We then construct a new graph that preserves all feasible routes of the original one, and we ensure excess-user-ride-time optimality on each arc of the new graph. We develop an efficient labeling algorithm with strong dominance rules on the new graph to allow fast shortest-path computations in solving the CG subproblems of the E-ADARP.
    \item \textbf{From an experimental perspective:} The numerical results demonstrate the superiority of our CG and B\&P algorithms over the state-of-the-art methods. Our CG algorithm solves 50 out of 84 instances optimally at the root node and obtains 40 equal lower bounds and 24 better lower bounds than those reported in \cite{bongiovanni2019electric}.
    Twenty-two new best solutions and 15 new optimal solutions are identified with our CG algorithm on previously solved and unsolved instances. On larger-scale instances, the CG algorithm yields 14 better solutions as well as 17 new lower bounds, compared to the best-reported heuristic results of \cite{su2023deterministic}. By integrating the CG algorithm into the B\&P framework, we finally solve 71 out of 84 instances optimally within the two-hour time limit, while the B\&C algorithm can only solve 49 instances optimally. In addition, our B\&P algorithm obtains 26 new best solutions and 30 improved lower bounds. 
    On larger-scale instances, we obtain 16 new best solutions, compared with the existing results of \cite{su2023deterministic}.
     \item \textbf{From a model perspective:} In the initial version of the E-ADARP in \cite{bongiovanni2019electric}, at-most-one charging visit is allowed on each recharging station. In our work, we propose a more general version of the E-ADARP, which allows unlimited visits to each recharging station.
    \item \textbf{Managerial insights and other interesting findings:} We investigate the impact of incorporating the total excess user ride time within the objectives and allowing unlimited charging visits to each recharging station. The following managerial insights are provided: (1) compared with the optimal results of solely minimizing operational costs (i.e., total travel times), we achieve substantial improvements in service quality (i.e., total excess user ride times) with only a slight increase in the other objective. This provides practical benefits for service providers, as they can significantly improve service quality while maintaining nearly optimal operational costs. (2) By allowing unlimited charging visits per recharging station, all instances can be solved feasibly, and their solution costs are reduced by 0.65\% on average. We also analyze the impact of applying different column initialization strategies. We find that employing an efficient algorithm to initialize the column pool may sometimes accelerate convergence.
\end{itemize}

This paper is organized into six sections. Following the introduction, Section \ref{review} presents an overview of related literature in the DARP and Electric VRPs (E-VRPs). 
In Section \ref{description}, we present the problem definition and the notations of sets, parameters, and variables used throughout the paper. An extended formulation of the E-ADARP is given at the end of this section. Section \ref{methodology} introduces the proposed CG and B\&P algorithms. The results of the proposed CG and B\&P algorithms are presented in Section \ref{experiments}. Finally, Section \ref{conclusion} discusses conclusions and extensions.

\section{Literature Review}\label{review}
The E-ADARP is different from classical DARPs in the following aspects: (1) the application of Electric Autonomous Vehicles (EAVs) in the vehicle fleet and partial recharging performed at recharging stations; (2) a weighted sum objective function minimizing total travel time and total excess user ride time.
Following these features, this section briefly reviews the literature that considers partial recharging in E-VRPs and emphasizes literature that applies CG and B\&P. Then, we review articles related to DARPs with Electric Vehicles (EVs) and those that specifically minimize the excess user ride time.

\subsection{Related Literature of E-VRPs}
The E-VRP was first introduced by \cite{conrad2011recharging} and has been extensively studied in recent years. \cite{schneider2014electric} proposed the Electric Vehicle Routing Problem with Time Windows and Recharging Stations (E-VRPTW). Their work considered customer time windows and limited vehicle capacities, and introduced new E-VRP instances. 
Since then, several problem variants of the E-VRPTW have been proposed considering different problem features. 

Many works relax the full recharging assumption to a partial recharging policy, e.g.,\citep{bruglieri2015variable,keskin2016partial,desaulniers2016exact,duman2021branch,ceselli2021branch}. The majority of these articles develop meta-heuristic algorithms to solve the problem, e.g., \cite{felipe2014heuristic,hiermann2019routing,montoya2017electric,froger2017matheuristic}, while some of them propose exact methods. The representative work is \cite{desaulniers2016exact}, where the authors investigate four E-VRP variants: E-VRP with full recharging plus single/multiple visits to recharging stations and E-VRP with partial recharging plus single/multiple visits to recharging stations. For each variant, customized mono- and bi-directional labeling algorithms are proposed to solve the CG subproblems within a B\&P framework. 
 \cite{desaulniers2020variable} further accelerate their previous labeling algorithms with a variable-fixing-based acceleration strategy, which yields significant speedup. Recently, \cite{duman2021branch} develop both an exact (i.e., Branch-and-Cut-and-Price, hereafter B\&C\&P) and a CG-based heuristic algorithm for solving E-VRPTW. Six acceleration techniques are applied, and both the proposed heuristic and exact algorithm obtain a number of new best and optimal solutions.
 \cite{lam2022branch} investigate a more practical case of EVRPTW in which the availability of chargers at the recharging stations and a piecewise-linear recharging time is taken into account. A B\&C\&P is proposed and is capable of solving instances with 100 customers.

\subsection{Related Literature of DARPs with EVs}
Several articles have investigated the impact of EVs on the DARP. \cite{masmoudi2018dial} is the first work to investigate a version of the DARP with EVs. Their model considers a battery swapping technique to recharge electric vehicles and applies a realistic energy consumption rate as in \cite{genikomsakis2017computationally}; the battery swapping time is a constant while the energy consumption on each arc varies, influencing the routing results. \cite{masmoudi2018dial} propose Evolutionary Variable Neighborhood Search (EVO-VNS) algorithms that can solve instances with up to three vehicles and 18 requests. \cite{bongiovanni2019electric} investigate EAVs in the context of the DARP and propose a new problem variant, namely the E-ADARP. The authors impose a minimum battery level (defined by $\gamma Q$, where $\gamma$ is the minimum battery level ratio and $Q$ is the total battery capacity) at the end of the route to restrict the vehicle's SoC at the destination depot. They analyze three different $\gamma$ values, i.e., $\gamma \in \{0.1, 0.4, 0.7\}$, representing a minimum SoC of 10\%, 40\%, and 70\% of the total battery capacity must be kept when the vehicle returns to the destination depot. As the $\gamma$ values rise, the problem becomes more constrained and more challenging to solve. The authors propose a Branch-and-Cut algorithm (B\&C) that solves 42 out of 56 instances optimally in less-constrained cases (i.e., $\gamma = 0.1, 0.4$). However, the proposed B\&C algorithm cannot obtain feasible solutions for 9 out of 28 instances within two hours of run time in the highly-constrained case (i.e., $\gamma = 0.7$). Moreover, for the instances that are not solved optimally, the reported lower bounds still have important gaps (with up to 9\%) to the best objective values. The largest instance that can be solved optimally by the B\&C algorithm contains 5 vehicles and 40 requests. More recently, \cite{su2023deterministic} propose a Deterministic Annealing (DA) local search, which provides near-optimal solutions within reasonable computational times. The authors construct a large-sized benchmark instance set and provide results for instances with up to 8 vehicles and 96 requests. \cite{bongiovanni2022machine} extend the static E-ADARP to the dynamic E-ADARP and propose a Machine Learning-based Large Neighborhood Search (MLNS), where random forest classification is applied in the selection phase of destroy and repair operators. The proposed MLNS algorithm outperforms the Adaptive Large Neighborhood Search (ALNS) metaheuristic of \cite{ropke2006adaptive} by nine percent on average in terms of solution quality, at the cost of additional computational time generated from solving the prediction problem at every iteration. \cite{kullman2022dynamic} propose an Electric Ridehail Problem with Centralized Control (E-RPC). They employ reinforcement learning to the E-RPC, and develop model-free policies which anticipate demand without any prior knowledge.


\subsection{Related Literature of Excess User Ride Time Minimization}
Several works have specifically minimized excess user ride time in the DARP. \cite{parragh2009heuristic} observe the ``eight-step" procedure in \cite{cordeau2003tabu} does not necessarily minimize excess user ride time. Therefore, the authors adapt the forward time slack calculation to avoid increasing the excess user ride time for a route. However, the adapted procedure leads to a more restrictive feasibility check and may result in incorrect infeasibility declarations. 
This drawback is fixed by the scheduling heuristic proposed by \cite{molenbruch2017multi}. The heuristic first sets the excess ride time of each request to its lower bound to construct a potentially infeasible schedule and then shifts service start times at some nodes to recover the infeasibility while minimizing excess user ride time. 
However, the developed scheduling procedures in \cite{parragh2009heuristic} and \cite{molenbruch2017multi} cannot ensure excess-user-ride-time optimality for a given route. Recently, \cite{bongiovanni2022ride} have proposed the first exact scheduling procedure for the DARP, which is then extended to the E-ADARP by integrating a battery management heuristic. However, the extended scheduling procedure is no longer exact as the excess-user-ride-time optimal schedules for a sequence may not be battery-feasible schedules. \cite{su2023deterministic} have introduced the first exact method to calculate the value of the minimal total excess user ride time for a given E-ADARP route. However, their method does not encompass the determination of schedules optimized for minimal excess user ride time and cannot be applied in a labeling algorithm as they only minimize excess ride time for a fixed route.

\subsection{Conclusion and Motivation}
Based on our review, \cite{bongiovanni2019electric} is the only work that proposes an exact method (i.e., the B\&C algorithm) to solve the E-ADARP. However, their proposed B\&C algorithm has difficulties obtaining optimal/feasible solutions within the given time limit for medium-to-large-sized instances. As a result, significant gaps between their best-found upper and lower bounds are observed. These limitations motivate us to propose an efficient CG approach, including a labeling algorithm for solving the pricing problem to provide tighter lower bounds and improved upper bounds. The CG performance largely depends on the efficiency of the labeling algorithm to solve the E-ADARP subproblems, where one must always determine the excess-user-ride-time optimal schedule from battery-feasible schedules during label extension. 
This issue complicates solving the CG subproblems of the E-ADARP and cannot be handled exactly by existing DARP feasibility check methods \citep{gschwind2015effective,gschwind2019adaptive} and scheduling procedures \citep{parragh2009heuristic,molenbruch2017multi,bongiovanni2022ride}. In this work, we address this issue by constructing a new sparser graph, where excess-user-ride-time optimality is guaranteed on each arc. We define a labeling algorithm on this new graph and handle battery feasibility by tailored REFs. Its efficiency is ensured by strong dominance rules and constant-time feasibility checks.


\section{The E-ADARP Description and Extended Formulation} \label{description}
This section presents the mathematical notation used throughout the paper and sketches the structure of the compact three-index E-ADARP formulation proposed in \cite{bongiovanni2019electric}. We extend the initial E-ADARP model to a more general version by allowing at most $N_s^{max}$ visits per recharging station, where $N_s^{max}$ can be any non-negative integer. When we set $N_s^{max} = \infty$, it implies that unlimited visits are allowed to recharging stations. In the second part, we briefly introduce the weighted-sum objective function of minimizing total travel time and total excess user ride time and the E-ADARP constraints. The final part introduces a path-based formulation for the E-ADARP. 

\subsection{Notation and Problem Statement}
Let $n$ denote the number of requests to be served by vehicles in the vehicle set $K$. The E-ADARP can be defined on a complete directed graph, denoted as $G=(V,A)$, where $V$ is the set of vertices and $A=\{(i,j):i,j \in V, i \neq j\}$ is the set of arcs. $V$ can be partitioned into several subsets, namely, $V= N \cup S \cup O \cup F$. Subset $N$ is the set of all customer nodes and is the union of customer pickup set (denoted as $P$) and drop-off set (denoted as $D$), namely, $N = P \cup D$. Subset $S$ is the set of recharging stations. Sets $O$ and $F$ denote the set of origin depots and destination depots, respectively. Particularly, the cardinality of $F$ might be higher than the total number of vehicles and therefore allows vehicles to select destination depots. As mentioned, EAVs are driverless and can be operated in a non-stop manner, which removes maximum route duration constraints and eliminates the need to predefine depots for the drivers to start and end their shifts. 
These characteristics differentiate the E-ADARP from most of the DARP literature (e.g., \cite{braekers2014exact}; \cite{parragh2011introducing}), where maximum route duration constraints are considered, and the depots are predefined. 

Each user request is a pair $(i,n+i)$, where $i=1,\cdots,n$ denotes the pickup and $n+i$ the drop-off node. A maximum user ride time $m_i$ is imposed for each request. Each customer node $i \in N$ is associated with a load $q_i$, and a service duration $s_i$ ($s_i > 0$), such that $q_i = -q_{n+i}$, $i = 1,\cdots,n$. For all other nodes $j \in O \cup F \cup S$, $q_j$ and $s_j$ are equal to zero. 
A time window on each node is also given and is denoted by $[e_i,l_i]$ for $i \in V$. We tackle the static E-ADARP, i.e., all user requests are known at the beginning of the planning horizon and have to be served.

Each electric vehicle $k \in K$ is characterized by a maximum vehicle capacity and a maximum battery capacity. Following \cite{bongiovanni2019electric}, we assume that vehicles have heterogeneous vehicle capacity $C_k$ and homogeneous battery capacity $Q$. The initial battery charge of each vehicle is denoted as $Q_0^k$. 
The travel time and battery consumption of each arc $(i,j) \in A$ are denoted as $t_{i,j}$ and  $b_{i,j}$, respectively. We assume that $b_{i,j}$ is constant on arc $(i,j)$ and is independent of load, speed, and SoC, as in \cite{desaulniers2016exact}. This assumption is suitable when the vehicle load is significantly smaller than the curb weight of the vehicle and velocity remains constant or has infrequent variations (\cite{pelletier2017battery}). 
Therefore, it is reasonable to keep this assumption in our work, as EAVs are characterized by limited vehicle capacity and constant cruise speed. The triangle inequality is assumed to hold for both $t_{i,j}$ and $b_{i,j}$. At recharging stations, partial recharging is performed with a constant battery recharge rate $\alpha$. This assumption remains reasonable as we allow partial recharging, and the charging process can be approximated by piece-wise linear functions, where the first phase ended at recharging 80\% of $Q$ (\cite{montoya2017electric}). 
To avoid numerical problems that may appear when computing time and battery consumption, all the battery consumption is converted to the time required to recharge the amount of consumption. That is, we define $h_{i,j} =  b_{i,j}/\alpha$ to convert the battery consumption on arc $(i,j)$ to the time required to recharge $b_{i,j}$. Similarly, the current battery charging level can also be converted to the time required to recharge to this battery charging level, and we define $H$ to denote the time required to recharge from zero to full battery capacity $Q$. Moreover, each vehicle has to maintain a certain level of battery when returning to the destination depot. This restriction is introduced to better evaluate the algorithm's performance in solving the E-ADARP. Without this restriction, even a large-sized instance with 5 vehicles and 50 requests can be solved without visiting the recharging station. It is also interesting to analyze the results under different levels of battery restrictions. The minimum battery level that needs to be reached at destination depots is $\gamma Q$, where $\gamma$ is the minimum battery level ratio. Finally, each recharging station can be visited at most $N_s^{max}$ times, where $N_s^{max}$ can be any non-negative integer.


\subsection{Objectives of the E-ADARP}
Different from most works in the context of the DARP, the objective of the E-ADARP is expressed as a weighted sum of the total travel time of all the vehicles and the excess user ride time of all the users. Taking excess user ride time into the objective allows minimizing the user inconvenience directly. 
Related works considering user ride time in the objective function are \cite{parragh2009heuristic} and \cite{molenbruch2017multi}. The objective function is formulated as:
\begin{equation} \label{objective}
     \min w_1\sum\limits_{k \in K}\sum\limits_{i,j \in V}t_{i,j}x_{i,j}^k + w_2\sum\limits_{i \in P}R_i
\end{equation}
where $x_{i,j}^k$ is a binary decision variable which denotes whether vehicle $k$ travels from node $i$ to $j$, $w_1$ and $w_2$ are weight factors. $R_i$ is the excess user ride time of request $i \in P$ and is formulated as the difference between the actual ride time $T_{n+i}^k-(T_{i}^k+s_i)$ and direct travel time $t_{i,n+i}$ and $T_{i}^k$ is the service start time at node $i$. 

Table \ref{tab::notation} summarizes all the notations and definitions for sets and parameters.

\begin{table}[ht]
\renewcommand\arraystretch{0.8}
    \caption{Notations of the E-ADARP}
    \label{tab::notation}
    \begin{center}
    \begin{tabular}{c c}
    \toprule
    Sets& Definitions\\
    \hline
    $N = \{1,\cdots,n,n+1,\cdots,2n\}$& Set of pickup and drop-off nodes \\
    $P=\{1,\cdots,i,\cdots,n\}$& Set of pickup nodes \\
    $D=\{n+1,\cdots,n+i,\cdots,2n\}$& Set of drop-off nodes \\
    $K=\{1,\cdots,k\}$& Set of available vehicles \\
    $O=\{o_1,o_2,\cdots,o_k\}$& Set of origin depots \\
    $F=\{f_1,f_2,\cdots,f_h\}$& Set of all available destination depots (supposing the total number is $h$)\\
    $S=\{s_1,s_2,\cdots,s_g\}$& Set of recharging stations (supposing the total number is $g$) \\
    $V= N \cup S \cup O \cup F$& Set of all nodes \\
    \midrule
    Parameters& Definitions\\
    \hline
    $t_{i,j}$& Travel time from location $i\in V$ to location $j\in V$ \\
    $b_{i,j}$& Battery consumption from location $i\in V$ to location $j\in V$ \\
    $h_{i,j}$& The charging time required to recharge $b_{i,j}, i,j\in V$ \\
    $e_i$& Earliest time at which service starts at $i\in V$ \\
    $l_i$& Latest time at which service starts at $i\in V$ \\
    $s_i$& Service duration at $i\in V$ \\
    $q_i$& Change in load at $i\in N$ \\
    $m_i$& Maximum user ride time of request $i\in P$ \\
    $C_k$& The vehicle capacity  \\
    $Q$& The battery capacity  \\
    $H$& Recharging time required to recharge from zero to $Q$ \\
    $Q_0^k$& The initial battery charging level of vehicle $k$\\
    $\alpha$& The recharged battery per time unit \\
    $\gamma$& The minimum battery level ratio \\
    $w_1, w_2$&The weight factors for total travel time and total excess user ride time \\
    \bottomrule
    \end{tabular}
    \end{center}
\end{table}

\subsection{Constraints of the E-ADARP}
We briefly list the series of constraints that need to be satisfied in the E-ADARP:
\begin{itemize}
\item [1)] Every route starts at an origin depot and ends at a destination depot and has no cycle;
\item [2)] Customer nodes are visited exactly once by all vehicles;
\item [3)] For each request, its pickup node $i \in P$ and drop-off node $n+i$ are served by the same vehicle, and the pickup node is visited before the drop-off node;
\item [4)] The maximum user ride time must be respected for each request;
\item [5)] Each destination depot $f \in F$ can be visited at most once by all vehicles; 
\item [6)] Each recharging station $s \in S$ can be visited at most $N_s^{max}$ times by all vehicles;
\item [7)] The vehicle load cannot exceed the maximum vehicle capacity at each node $i \in V$;
\item [8)] Each node $i \in V$ must be visited within its time window $[e_i,l_i]$. Waiting time occurs when the vehicle arrives earlier than the earliest time window;
\item [9)] The battery charging level at the destination depot must be at least equal to the minimum battery level $\gamma Q$;
\item [10)] The battery charging level at any node along a route must be positive and not exceed the battery capacity $Q$;
\item [11)] The recharging station can only be visited when there is no passenger on board;
\end{itemize}



\subsection{Extended Formulation of the E-ADARP} \label{MP}
The E-ADARP consists of finding a set of E-ADARP routes for EAVs to transport users with specific origin-destination pairs such that each customer $i \in N$ is visited exactly once, and the weighted sum of total travel time and total excess user ride time is minimized. The definition of an E-ADARP route is as follows:
\begin{definition}[E-ADARP route] \label{E-ADARP route}
    An E-ADARP route is a path that starts at an origin depot and ends at a destination depot such that Constraints (1), (3), (4), (6) to (11) are satisfied.
\end{definition}

We obtain the extended formulation (also called ``Master Problem", abbreviated as MP) of the E-ADARP via Dantzig-Wolfe decomposition. The MP is formulated as a set covering problem, where $\Omega$ denotes the set of all E-ADARP routes. 
For each route $\omega \in \Omega$, we define $c_{\omega}$ as the cost of route $\omega$. Note that the route cost accounts for the weighted sum of total travel time and total excess user ride time, as presented in Equation (\ref{objective}). To consider heterogeneous vehicles, we define $\mathcal{T}$ as the set of available vehicle types and classify vehicles with the same capacity into one vehicle type $t \in \mathcal{T}$ (as in \cite{parragh2012models}). Let $\Omega^t$ denote the set of feasible routes for vehicles of type $t \in \mathcal{T}$ and $\Omega = \bigcup_{t \in \mathcal{T}}\Omega^t$ is the set of all feasible routes. We denote $M_t$ as the maximum number of vehicles of type $t$ that can be included in the solution.

We define $\theta_{i\omega}$ as a binary coefficient that equals one if the request $i$ is visited by route $\omega$ (zero otherwise). Let $y_{\omega}$ denote a binary variable that equals one if and only if route $\omega \in \Omega$ is included in the solution (0 otherwise). The number of requests to be served is $n$. To restrict the visits to each destination depot and recharging station, we define binary coefficients $\phi_{f\omega}$ and $\rho_{s\omega}$, determining whether the destination depot $f$ or recharging station $s$ is visited in route $\omega$. As we have multiple origin depots that may be located at different places and cannot be visited repeatedly by different vehicles, we define  $\epsilon_{o \omega}$ to denote whether origin depot $o$ is visited in $\omega$. The objective function of MP is formulated as the total routing cost and we improve the formulation of MP by adding a high penalty $P_i$ for request $i$ not served in the objective function. The penalties for unserved requests are only used to warm-start the problem. For instance, we can start from a heuristic solution of the E-ADARP (e.g., obtained from the DA algorithm of \cite{su2023deterministic}) that does not include all the requests. Also, we introduce a binary variable, denoted as $a_i$, to represent whether request $i$ is visited or not. If $a_i = 1$, request $i$ is omitted, otherwise, request $i$ is visited. The set covering problem is formulated as:
\newline
\begin{equation} \label{1.1}
\min \sum\limits_{\omega \in \Omega}c_{\omega} y_{\omega} + \sum\limits_{i \in P}P_i a_i
\end{equation}

subject to:
\begin{equation}\label{1.2}
 \sum\limits_{\omega \in \Omega}\theta_{i \omega} y_{\omega} \geqslant 1 - a_i, \quad \forall i \in P
\end{equation}
 
\begin{equation}\label{1.3}
 \sum\limits_{\omega \in \Omega}\phi_{f \omega} y_{\omega} \leqslant 1, \quad \forall f \in F
\end{equation}

\begin{equation}\label{1.3.2}
 \sum\limits_{\omega \in \Omega}\rho_{s \omega} y_{\omega} \leqslant N_s^{max}, \quad \forall s \in S
\end{equation}

\begin{equation}\label{type-u specific constraint}
     \sum\limits_{\omega \in \Omega}\epsilon_{o \omega} y_{\omega} \leqslant 1, \quad \forall o \in O
\end{equation}

\begin{equation}\label{1.4}
\sum\limits_{\omega \in \Omega} y_{\omega} \leqslant M_t, \quad \forall t \in \mathcal{T}
\end{equation}

\begin{equation}\label{1.5}
     y_{\omega} \in \{0,1\}, \quad \forall \omega \in \Omega
\end{equation}

\begin{equation}\label{introduced variable}
     a_i \in \{0,1\}, \quad \forall i \in P
\end{equation}

Constraints (\ref{1.2}), (\ref{1.3}), and (\ref{1.3.2}) restrict the visit to each request, destination depot, and recharging station. Constraints (\ref{type-u specific constraint}) and (\ref{1.4}) guarantee that each origin depot appears at most once in the solution and at most $M_t$ of vehicles of type $t$ can be used. Due to the large size of $\Omega$, we cannot solve the MP directly. Instead, we solve the linear relaxation of the MP on a subset of set $\Omega$ (denoted as $\Omega'$), which we call the continuous Restricted Master Problem (hereafter continuous RMP). The subset $\Omega'$ can be generated by CG.

In CG, the continuous RMP and the pricing subproblems are solved iteratively. The subproblems are solved to generate E-ADARP routes with negative reduced costs. These routes are added to $\Omega'$, and the continuous RMP will be solved to update the dual variable values. With the renewed dual variable values, the subproblems will be solved again to find E-ADARP routes with negative reduced costs. The iterative solving of the continuous RMP and subproblems ends when no more negative-reduced-cost columns can be found. In this case, the optimal solution of the continuous MP is found. The integer RMP is solved at the end to obtain integer solutions by using all columns generated. 


\section{Branch-and-Price Algorithm}\label{methodology}
In this section, we present the CG subproblem in the first part. Then, we focus on designing the labeling algorithm to solve the subproblem in Section \ref{labeling}. 
The cutting planes to strengthen the continuous RMP are elaborated in Section \ref{cuts}. The considered branching strategies in the B\&P framework are presented in Section \ref{branches}. 

\subsection{Column Generation Subproblem} \label{CG}
As mentioned, we solve the pricing subproblem in order to identify E-ADARP routes with negative reduced cost $\Bar{c}_{\omega}, \omega \in \Omega$. The reduced cost for an E-ADARP route $\omega$ is formulated as:



\begin{equation}\label{type-u subproblem}
     \bar{c}_{\omega} = c_{\omega} - \sum\limits_{i \in P}\theta_{i\omega}\lambda_i -\sum\limits_{f \in F}\phi_{f \omega}\tau_f - \sum\limits_{s \in S \cup F}\xi_{s \omega}\rho_s - \sum\limits_{o \in O}\epsilon_{o\omega}\zeta_o - \kappa_t
\end{equation}
where $\lambda_i, i \in P$, $\tau_f, f\in F$, $\xi_s, s \in S$, and $\zeta_o, o \in O$ are the dual variable values of constraints (\ref{1.2}) to (\ref{type-u specific constraint}), respectively. The dual variable values associated with constraint (\ref{1.4}) is $\kappa_t$. 

The objective function of the subproblem is:

\begin{equation}
    \displaystyle{\minimize_{\omega \in \Omega} \Bar{c}_{\omega} }
\end{equation}




\subsection{Forward Labeling Algorithm for ESPPRC-MERT} \label{labeling}
We design a customized forward labeling algorithm to solve the pricing sub-problems, which are formulated as Elementary Shortest Path Problems with Resource Constraints and Minimizing Excess Ride Time (hereafter ESPPRC-MERT). This labeling algorithm extends the one tackling E-VRP subproblems in \cite{desaulniers2016exact} by considering the following aspects: 

\begin{itemize}
    \item [1)]
    The characteristics of the DARP are taken into account (i.e., pickup and delivery);
    \item [2)]
    Problem-specific constraints (i.e., minimum-battery-level constraint, maximum user ride time constraint, limited visits to each recharging station) are considered;
    \item [3)]
    Minimizing the total excess user ride time for a partial path.
\end{itemize}

The last point is the most challenging one of solving the ESPPRC-MERT, as the minimum excess user ride time is particularly difficult to be calculated in the extension of labels. This difficulty manifests in two aspects: (1) the minimum excess user ride time can only be determined at nodes where the vehicle has no passenger onboard at arrival/departure; (2) an excess-user-ride-time optimal schedule for a partial path may conflict with time window constraints on succeeding nodes. 


To handle this issue, we construct a new sparser graph $G_{sp}$, where each arc is ensured to be excess-user-ride-time optimal. We propose an efficient labeling algorithm over $G_{sp}$ to compute routes with negative reduced costs. The construction of $G_{sp}$ takes three steps:(1) we generate all battery-restricted fragments (defined in Definition \ref{fragment}); (2) we abstract each fragment to an arc that ensures excess-user-ride-time optimality; (3) we construct $G_{sp}$ by connecting each transformed arc with depots, recharging stations, and other transformed arcs in a feasible way.

This section is organized as follows: in Section \ref{representation}, we introduce a fragment-based representation of paths, which regards fragments as basic components. Then, each partial path is a concatenation of fragments, over which the minimum excess user ride time is determined. In Section \ref{scheduling}, we explain how fragments are abstracted to arcs. 
This abstraction allows us to design a single REF that represents the extension from the start node to the end node of a fragment. In Section \ref{graph construction}, we enumerate all feasible fragments, abstract fragments to arcs, and connect transformed arcs with each other, depots, and recharging stations in a feasible way to construct $G_{sp}$. We show in Theorem \ref{thm:optimization}, $G_{sp}$ preserves all feasible routes of the original one and therefore preserves all negative-reduced-cost routes of the original graph. Finally, we define  labels for nodes on $G_{sp}$ and present their notations and the definitions of their associated resources in Section \ref{labeling algorithm}. The following part includes the label feasibility check, REFs, and dominance rules.

\subsubsection{Representation of Partial Paths.} \label{representation}
One important characteristic of the ESPPRC-MERT is that the reduced cost incorporates the total excess user ride time, which needs to be minimized along the extension. In the classical representation of a partial path, the path is extended in a node-by-node fashion. As mentioned, in the case of open requests existing on the partial path, we cannot calculate the excess user ride time for these open requests. Hence, for the label associated with this partial path, its reduced cost cannot be determined. In order to calculate the minimum excess user ride time in the extension of a partial path, we extend the partial path with battery-restricted fragments, as in \cite{su2023deterministic}. This representation generalizes the notion of fragments as proposed in \cite{rist2021new} by adding battery constraints in the feasibility check. The definition for a battery-restricted fragment is as follows:

\begin{definition}[Battery-restricted fragment, \cite{su2023deterministic}] \label{fragment}
Assuming that $\mathcal{F} = (i_1,i_2, \cdots,i_k)$ is a sequence of pickup and drop-off nodes, where the vehicle arrives empty at $i_1$ and leaves empty at $i_k$ and has passenger(s) on board at other nodes. Then, we call $\mathcal{F}$ a \BRFrag if there exists a feasible route of the form:
$$(o,s_{i_1},\cdots,s_{i_v},\overbrace{i_1,i_2, \cdots,i_k}^{\mathcal{F}},s_{i_{v+1}},\cdots,s_{i_m},f),$$ where $s_{i_1},\cdots,s_{i_v},s_{i_{v+1}},\cdots,s_{i_m} (v,m \geqslant 0)$ are recharging stations, and $o \in O$, $f \in F$.
\end{definition}

It should be noted that if the route is feasible without visiting a recharging station (i.e., $v=m=0$), the battery-restricted fragment is equivalent to the one defined in \cite{rist2021new}. Based on Definition \ref{fragment}, each E-ADARP route can be regarded as the concatenation of an origin depot, battery-restricted fragments (hereinafter referred to as ``fragments''), recharging stations (if required), and a destination depot. Clearly, we can exactly minimize the excess user ride time on each fragment. 

\subsubsection{Abstracting Fragments to Arcs.}
\label{scheduling}
In this section, we present the general method to abstract each fragment $\mathcal{F}$ to an arc that captures all excess-user-ride-time optimal schedules over $\mathcal{F}$. Assuming that a fragment $\mathcal{F}$ is $\{1, 2,\cdots, m\}$, we denote any excess-user-ride-time optimal schedule $\mathcal{A}$ for $\mathcal{F}$ as a set of service start times, namely, $\mathcal{A} = (\mathcal{A}_1, \mathcal{A}_2, \cdots, \mathcal{A}_m)$. Then, to abstract $\mathcal{F}$ to an arc, it is enough to determine all possible values of $\mathcal{A}_1$ and $\mathcal{A}_m$ for any excess-user-ride-time optimal schedule $\mathcal{A}$ over $\mathcal{F}$.  To calculate all possible values of $\mathcal{A}_1$ and $\mathcal{A}_m$, we introduce \emph{vehicle-waiting-time optimal schedules}:

\begin{definition}
A vehicle-waiting-time optimal schedule $\mathcal{B}$ for a fragment $\mathcal{F}$ is defined as a set of service start times $\mathcal{B}_i$, $i \in \mathcal{F}$ that minimize the sum of vehicle waiting times at each node along $\mathcal{F}$ (i.e., $\sum_{i = 2}^m [\mathcal{B}_{i}-(\mathcal{B}_{i-1}+ t_{i-1,i}+s_{i-1})]$). 
\end{definition}
Note that a vehicle-waiting-time optimal schedule is not necessarily an excess-user-ride-time optimal schedule, as the latter one minimizes a weighted sum of waiting times along $\mathcal{F}$, which weight factors are equal to vehicle loads at nodes with waiting time. 
For a given fragment $\mathcal{F}$, we determine two vehicle-waiting-time optimal schedules:
\begin{enumerate}
    \item the ``latest" vehicle-waiting-time optimal schedule $\mathcal{B}^l$;
    \item the ``earliest" vehicle-waiting-time optimal schedule $\mathcal{B}^e$.
\end{enumerate}
We show later in Theorem \ref{thm:segment scheduling} that these schedules $\mathcal{B}^l, \mathcal{B}^e$  determine all possible values of $\mathcal{A}_1$ and $\mathcal{A}_m$ in any excess-user-ride-time optimal schedule $\mathcal{A}$ by the following means:
For any excess-user-ride-time optimal schedule $\mathcal{A}$, there exists $ \delta^l,\delta^e \geqslant 0 $ such that:
    \begin{align*}
    \mathcal{A}_1 = \mathcal{B}_1^l - \delta^l, \ \mathcal{A}_m = \mathcal{B}_m^l - \delta^l;\\
    \mathcal{A}_1 = \mathcal{B}_1^e + \delta^e, \ \mathcal{A}_m = \mathcal{B}_m^e + \delta^e;\\
    \end{align*}

To better illustrate our idea, we take an example as follows:

\begin{exmp} \label{example2}
Given a fragment $\mathcal{F} = \{1+,2+,1-,2-\}$, the time window on each node and the travel time for each arc are shown in Figure \ref{all optimal schedules on fragment}. The dashed lines present the direct travel times from pickup nodes to the corresponding drop-off nodes. Assuming that the service time at each node is equal to zero and each request includes one passenger to be transported.

\begin{figure}[!htp]
\centering
\includegraphics[width=11cm]{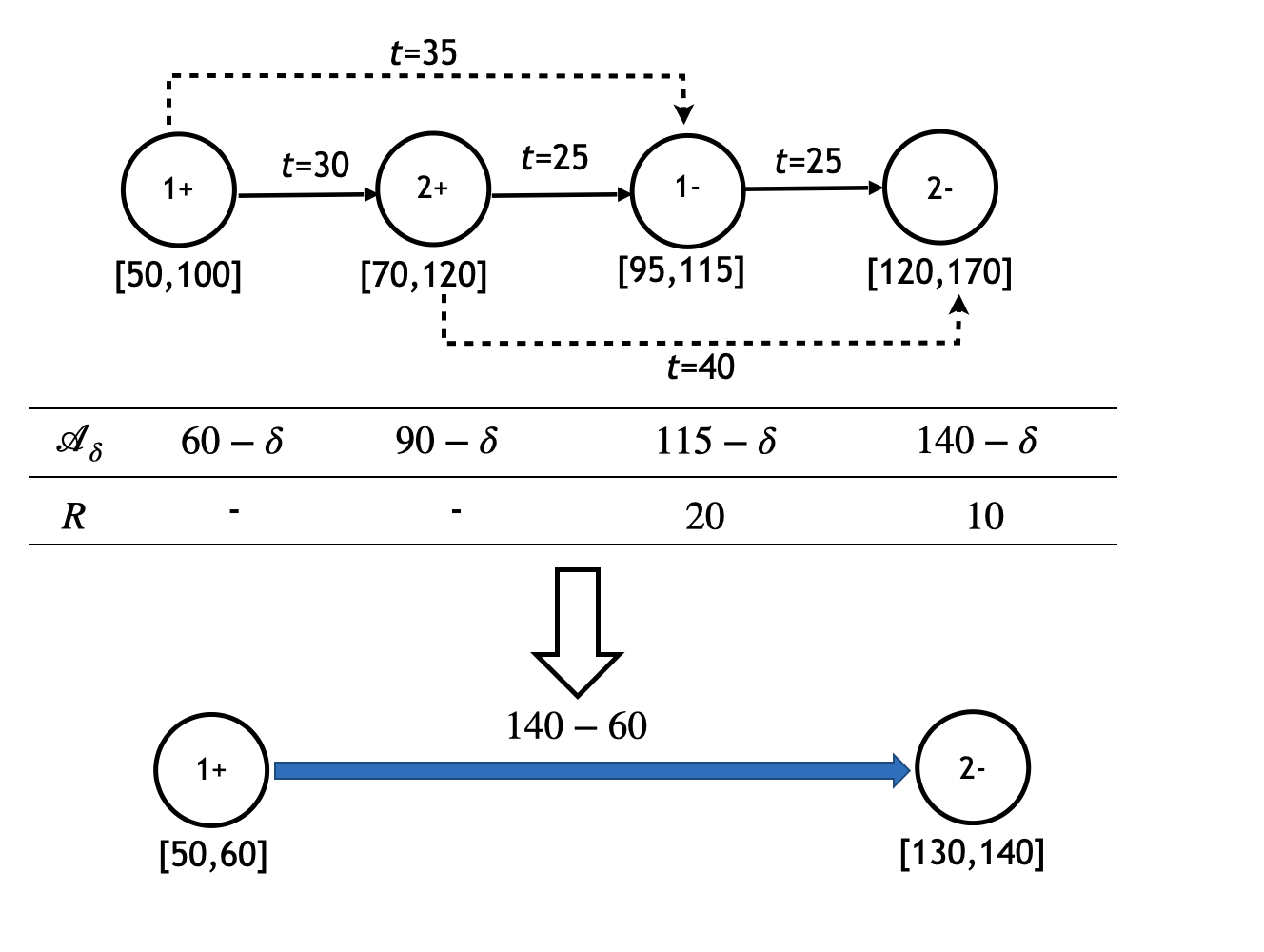}
\caption{\centering Example of abstracting a fragment to an arc}
\label{all optimal schedules on fragment}
\end{figure}


Clearly, any excess-user-ride-time optimal schedules on $\mathcal{F}$ can be represented as $\mathcal{A}_{\delta}$, where $0 \leqslant \delta \leqslant 10$. These schedules have an identical minimum excess user ride time of 30 minutes. Hence, we can abstract $\mathcal{F}$ to an arc from node 1+ and node 2- with time windows being [50,60] and [130,140]. In this example, we have $B_{1}^e = 50, B_{1}^l = 60, B_{m}^e = 130, B_{m}^l = 140$, where node $1$ and $m$ corresponds to node $1+$ and $2-$. This arc captures all excess-user-ride-time optimal schedules $\mathcal{A}_{\delta}$ on $\mathcal{F}$. The travel time from node 1+ to node 2- is presented on the converted arc.

\end{exmp}



Next, for a given fragment $\mathcal{F} = \{1,2,\cdots,m\}$, we present the construction scheme for $\mathcal{B}^l$ and $\mathcal{B}^e$. 

\paragraph{\textbf{Construction of $\mathcal{B}^l$ and $\mathcal{B}^e$}}

The latest vehicle-waiting-time optimal schedule $\mathcal{B}^l$ must obey the following two rules:
\begin{enumerate}
    \item Starting service as late as possible at the first node in $\mathcal{F}$;
    \item Starting service as early as possible at all other nodes in $\mathcal{F}$;
\end{enumerate}

For the first rule, as the vehicle arrives at the first node of $\mathcal{F}$ with no passenger, the delay of service start time at the first node will always help to eliminate unnecessary vehicle waiting time at succeeding nodes. 

As for the second rule, when there is/are passenger(s) on board, it is straightforward to start service as early as possible as in this case to reduce vehicle waiting time. In the following part, we will first construct $\mathcal{B}^l$ and then construct $\mathcal{B}^e$.
\begin{itemize}

    \item \textbf{Construct $\mathcal{B}^l$:} Assuming that a fragment $\mathcal{F}=\{1,2,\cdots,m\}$. Let $\mathcal{B}^l_i$ be the service start time for schedule $\mathcal{B}^l$ at node $i$. Then the arrival time at each node $i$ is $Arr_i = \mathcal{B}^l_{i-1}+t_{i-1,i}+s_{i-1}, 2 \leqslant i \leqslant m$. The waiting time $\Delta_i$ at node $i$ is calculated as $\Delta_i = \mathcal{B}^l_i-Arr_i$.
Based on the proposed rules, we define $\mathcal{B}^l_i$ inductively as follows:
\begin{enumerate}
    \item $\mathcal{B}^l_1 = l_1$;
    \item assuming $\mathcal{B}^l_i$ has been defined for $i< v$, we define $\mathcal{B}^l_v$ by:
    \begin{enumerate}
        \item if the extension from node $(v-1)$ respects the time window constraint at node $v$ (i.e., $Arr_v \leq l_v $), then we define $\mathcal{B}^l_v = \max \{ e_v, Arr_v\}$;
        \item Otherwise, 
        \begin{itemize}
            \item if $\min \limits_{i<v}\{\mathcal{B}^l_{i}-e_i\}\geq Arr_v-l_v$,  we can update the schedule at nodes $1,\cdots, v-1$ by moving forward $Arr_v-l_v$. Then the extension from node $(v-1)$ to node $v$ will not violate the time window constraint. I.e.,  we update $\mathcal{B}^l_i$ as $\mathcal{B}^l_i- (Arr_v-l_v)$ for $i = 1,\cdots , v-1 $ and define $\mathcal{B}^l_v = l_v$;
            \item Otherwise, there is no feasible schedule for $\{1,2,\cdots,v\}$.
        \end{itemize}
    \end{enumerate}
\end{enumerate}

For the last point, the maximum amount of time that can be moved forward from node $1$ to node $(v-1)$ is $\min \limits_{i<v}\{\mathcal{B}^l_{i}-e_i\}$. To satisfy the time window constraint at node $v$, the minimum amount of time that is required to be moved forward is $Arr_v-l_v$. In the case of $\min \limits_{i<v}\{\mathcal{B}^l_{i}-e_i\} < Arr_v-l_v$, the time window constraint at node $v$ is never fulfilled.

\item \textbf{Construct $\mathcal{B}^e$:} After determining the latest vehicle-waiting-time optimal schedule $\mathcal{B}^l$ on $\mathcal{F}$, the earliest vehicle-waiting-time optimal schedule $\mathcal{B}^e$ is determined by moving forward $\mathcal{B}^l$ on $\mathcal{F}$ by the maximum amount of time (denoted as $\mathop{\Delta}\limits ^{\leftarrow}$) that will not change the minimum vehicle waiting time. 
$\mathop{\Delta}\limits ^{\leftarrow}$ can be calculated by taking the minimum value among all the $\mathcal{B}_i^l-e_i$, for $i \in \mathcal{F}$. That is: $\mathop{\Delta}\limits ^{\leftarrow} = \min \limits_{i \in \mathcal{F}}\{\mathcal{B}_i^l-e_i\}$, and $\mathcal{B}_i^e=\mathcal{B}_i^l-\mathop{\Delta}\limits^{\leftarrow}$, $ i \in \mathcal{F}$.

\end{itemize}

We recall the previous example to present the construction of $\mathcal{B}^l$ and $\mathcal{B}^e$:
\begin{exmp}[Example \ref{example2} continued] \label{example3}
For the given fragment $\mathcal{F}$, $\mathcal{B}^l$ and $\mathcal{B}^e$ for fragment $\mathcal{F}$ are directly shown in Figure \ref{example_schedule}. In this example, $\mathop{\Delta}\limits ^{\leftarrow} = 10$ and $\mathcal{B}_i^e=\mathcal{B}_i^l-10$, $ i \in \mathcal{F}$. Interested readers are invited to derive these results by themselves. 
\begin{figure}[htp]
\centering
\includegraphics[width=11cm]{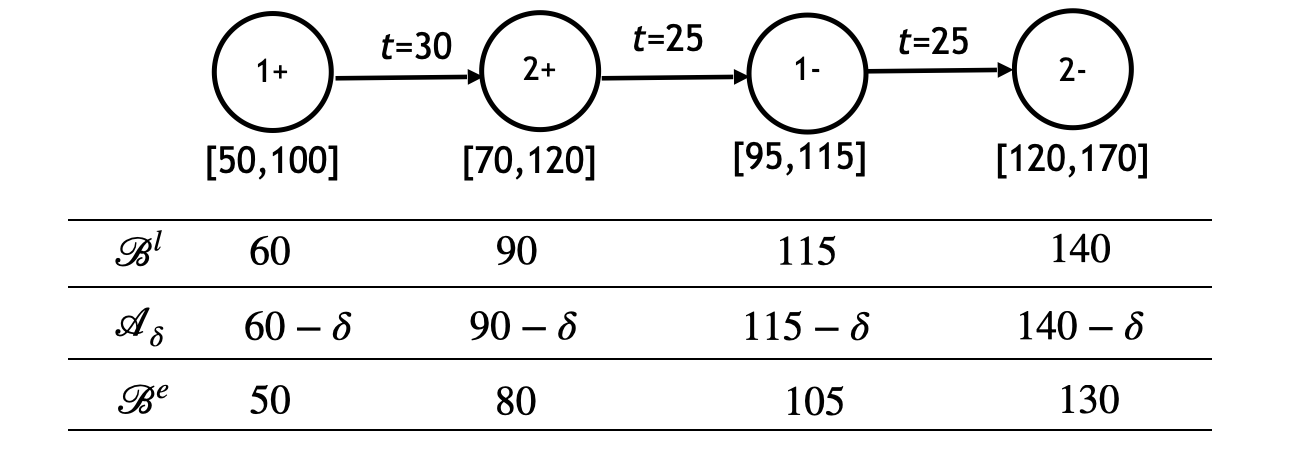}
\caption{\centering Example of any excess-user-ride-time optimal schedule on a fragment}
\label{example_schedule}
\end{figure}


\end{exmp}

Now, we prove that $\mathcal{B}^l, \mathcal{B}^e$ can determine all possible values of $\mathcal{A}_1$ and $\mathcal{A}_m$ of any excess-user-ride-time optimal schedule $\mathcal{A}$. 


\begin{theorem} \label{thm:segment scheduling}
Assuming that fragment $\mathcal{F} = \{1, 2,\cdots, m\}$, $\mathcal{A}$ is an \textbf{excess-user-ride-time optimal} schedule, $\mathcal{B}^l$ is the constructed latest vehicle-waiting-time optimal schedule over $\mathcal{F}$, and $\mathcal{B}^e$ is the earliest vehicle-waiting-time optimal schedule. Then there exists $\delta^e, \delta^l \geq 0$ such that:
\begin{align*}
    \mathcal{A}_1 = \mathcal{B}_1^e + \delta^e, \ \mathcal{A}_m = \mathcal{B}_m^e + \delta^e;\\
       \mathcal{A}_1 = \mathcal{B}_1^l - \delta^l, \ \mathcal{A}_m = \mathcal{B}_m^l - \delta^l.\\
\end{align*}

\end{theorem}

\begin{proof}[Proof of Theorem \ref{thm:segment scheduling}]
  In the case of $\Delta_u(\mathcal{B}^l)=0$ for all $1\leq u\leq m$ ($\Delta_i(\cdot)$ is the waiting time at node $i$ according to a given schedule), then the theorem clearly holds. Next, we assume that there exists $1\leq v\leq m$ such that $\Delta_v(\mathcal{B}^l) \neq 0$. We will show that  $\mathcal{A}_1 = \mathcal{B}^l_1, \mathcal{A}_m = \mathcal{B}^l_m$.

The proof contains two parts.

\begin{itemize}
    \item $\mathcal{A}_1 = \mathcal{B}^l_1$. According to the first construction rule of $\mathcal{B}^l$, we have $\mathcal{B}^l_1 \geq \mathcal{A}_1$. Next, we prove $\mathcal{B}^l_1 \leq \mathcal{A}_1$ by contradiction.
    
    Assuming that $\mathcal{B}^l_1 > \mathcal{A}_1$, then if $\Delta_{2}(\mathcal{A})=0$, we must have $\mathcal{B}^l_2 > \mathcal{A}_2\geq e_2$, therefore $\Delta_{2}(\mathcal{B}^l)=0$ by our construction. Moreover, if $\Delta_{3}(\mathcal{A})=0$, we must have $\mathcal{B}^l_3 > \mathcal{A}_3\geq e_3$, therefore $\Delta_{3}(\mathcal{B}^l)=0$. 
    
    Repeat the above process, since $\Delta_v(\mathcal{B}^l) \neq 0$, there exists $v_0\leq v$ such that: $\mathcal{B}^l_i > \mathcal{A}_i, \Delta_{i}(\mathcal{A})=0$ for $1\leq i< v_0$, $\Delta_{v_0}(\mathcal{A})>0$.
    Then the total excess user ride time of $\mathcal{A}$ can be further reduced by \textbf{delaying} the service start time by:
    $$min\{min_{1\leq i<v_0}\{\mathcal{B}^l_i- \mathcal{A}_i\}, \Delta_{v_0}(\mathcal{A})\}$$
    in $1, \cdots, v_0-1$. $\mathcal{A}$ is not an optimal plan, which is a contradiction!
    
    \item $\mathcal{B}^l_m = \mathcal{A}_m$. Since there exists $1\leq v\leq m$ such that $\Delta_v(\mathcal{B}^l) \neq 0$, we have $\mathcal{B}^l_v = e_v$ at node $v$. According to the second construction rules of $\mathcal{B}^l$, we also have $\mathcal{B}^l_m \leq \mathcal{A}_m$. Next, we prove $\mathcal{B}^l_m \geq \mathcal{A}_m$ by contradiction.
    
    Assuming that $\mathcal{B}^l_m < \mathcal{A}_m$, if $\Delta_{m}(\mathcal{A})=0$, we must have $ \mathcal{B}^l_{m-1} < \mathcal{A}_{m-1}$. 
    Moreover, if $\Delta_{m-1}(\mathcal{A})=0$, we must have $\mathcal{B}^l_{m-2} < \mathcal{A}_{m-2}$. 
    Since $\mathcal{B}^l_1 = \mathcal{A}_1$ as we proved above, there must exists $2\leq v_1 \leq m$ such that: $\mathcal{B}^l_i < \mathcal{A}_i,  \Delta_{i}(\mathcal{A})=0$ for $v_1< i\leq m$ and $e_{v_1}\leq \mathcal{B}^l_{v_1} < \mathcal{A}_{v_1}, \Delta_{v_1}(\mathcal{A})>0$.
  Then the excess user ride time of $\mathcal{A}$ can be further reduced by \textbf{moving forward} the service start time by:
    $$min\{min_{v_1\leq i\leq m}\{\mathcal{A}_i- \mathcal{B}^l_i\}, \Delta_{v_1}(\mathcal{A})\}$$
    in $v_1, \cdots,m$. $\mathcal{A}$ is not an optimal plan, which is a contradiction!
    
\end{itemize}
\end{proof}

Theorem \ref{thm:segment scheduling} also implies that as soon as the excess-user-ride-time optimal schedule for fragment $\mathcal{F}$ contains waiting time, we have $\mathcal{B}^l = \mathcal{B}^e$ and $\delta^l = \delta^e = 0$. In this case, any excess-user-ride-time optimal schedule (denoted as $\mathcal{A}$) must satisfy $\mathcal{A}_1 = \mathcal{B}^l_1 = \mathcal{B}^e_1$ and $\mathcal{A}_m = \mathcal{B}^l_m = \mathcal{B}^e_m$, where node 1 and node $m$ are the first and the last node of $\mathcal{F}$. 
In the other case, an excess-user-ride-time optimal schedule can be obtained by moving forward(backward) the vehicle-waiting-time optimal schedule $\mathcal{B}^l$ ($\mathcal{B}^e$) by $\delta$, such that $0 \leqslant \delta \leqslant \mathcal{B}^l_1-\mathcal{B}^e_1$.


\paragraph{\textbf{Abstracting a Fragment to an Arc.}}
For each fragment $\mathcal{F}=\{1, \cdots, m\}$, assuming $\mathcal{B}^e,\mathcal{B}^l$ are the corresponding earliest and latest vehicle-waiting-time optimal schedules. Based on Theorem \ref{thm:segment scheduling}, restricting time windows at node $1$ and node $m$ to $[\mathcal{B}^e_1, \mathcal{B}^l_1]$ and $[\mathcal{B}^e_m, \mathcal{B}^l_m]$ will include all excess-user-ride-time optimal schedules on fragment $\mathcal{F}=\{1, \cdots, m\}$. Then we can abstract $\mathcal{F}$ to an arc $(1, m)$ such that: 
\begin{enumerate}
    \item the total travel time from $1$ to $m$ (denoted as $t_{1,m}'$) is $\mathcal{B}^l_m-\mathcal{B}^l_1$;
    \item the original time windows of node 1 and node $m$ are restricted to  $[\mathcal{B}^e_1, \mathcal{B}^l_1]$ and $[\mathcal{B}^e_m, \mathcal{B}^l_m]$;
    \item the battery consumption from $1$ to $m$ is $\sum \limits_{i = 1}^{m-1}h_{i,i+1}$;
\end{enumerate}

If no waiting time is generated on $\mathcal{F}$, we can calculate the minimum excess user ride time directly. It is also straightforward to compute the minimum excess user ride time for $\mathcal{F}$ that contains one request with waiting time generated. When waiting time is generated on a fragment containing more than two requests, calculating the value of minimum excess user ride time becomes difficult. In this case, we improve the calculation process in \cite{su2023deterministic} by setting the service start times at node 1 and node $m$ to $\mathcal{B}_1^l$ and $\mathcal{B}_m^l$, respectively. Then, we solve a Linear Program (LP) to determine the minimum excess user ride time, as presented in \ref{case 2 solving}.

\subsubsection{Constructing a New Sparser Graph.} \label{graph construction}
 In this section, we construct a new sparser graph $G_{sp}$ by two steps: (1) we enumerate all feasible fragments and abstract them to arcs; (2) we connect each transformed arc with depots, recharging stations, and other transformed arcs in a feasible way.
 
The fragment enumeration is conducted with depth-first search as in \cite{alyasiry2019exact}. For each feasible fragment, the corresponding restricted time windows and minimum excess user ride times are recorded. To generate all feasible fragments, we assume that the vehicle departs from each pickup node with a full battery level and must respect constraints of maximum user ride time, battery capacity, time window, pairing, precedence, and vehicle capacity. We start from each pickup node and extend it node by node until no more feasible fragment that starts from this pickup node can be generated. By enumerating all feasible fragments before computation, we largely accelerate the labeling algorithm as we only need to query information instead of recalculating. To provide more details, we conduct a preliminary test for fragment enumeration on each instance in \ref{preliminary for frag enumeration}. For all the instances, the fragment enumeration can be fulfilled in a matter of seconds. In the computational experiments, we report the CPU time, which includes the computational time for fragment enumeration. With the information of all feasible fragments, we abstract fragments to arcs as presented in Section \ref{scheduling}.

Then, we construct $G_{sp}$ by connecting depots, recharging stations, and the start nodes and end nodes of transformed arcs in a feasible way. Details for the connection between nodes are as follows:
\begin{enumerate}
    \item Each origin depot connects with all start nodes of arcs, recharging stations, and destination depots;
    \item Each recharging station connects with start nodes of arcs and destination depots in a feasible way;
    \item Each end node of an arc connects with destination depots, recharging stations, and all the start nodes of arcs in a feasible way;
\end{enumerate}


Figure \ref{new graph2} shows an example of constructing new arcs in $G_{sp}$. It should be noted that for two different fragments, even though they have the same start node $i+$ and end node $j-$, we need to treat them as two different arcs in $G_{sp}$ as they represent different fragments consisting of different sequences of nodes, which lead to different restricted time windows. To distinguish, we make copies (e.g., $i'+$, $j'+$) for node $i+$ and node $j-$ and we generate two arcs $(i+,j-)$ and $(i'+,j'-)$.

\begin{figure}[!htp]
\centering
\includegraphics[width=14cm]{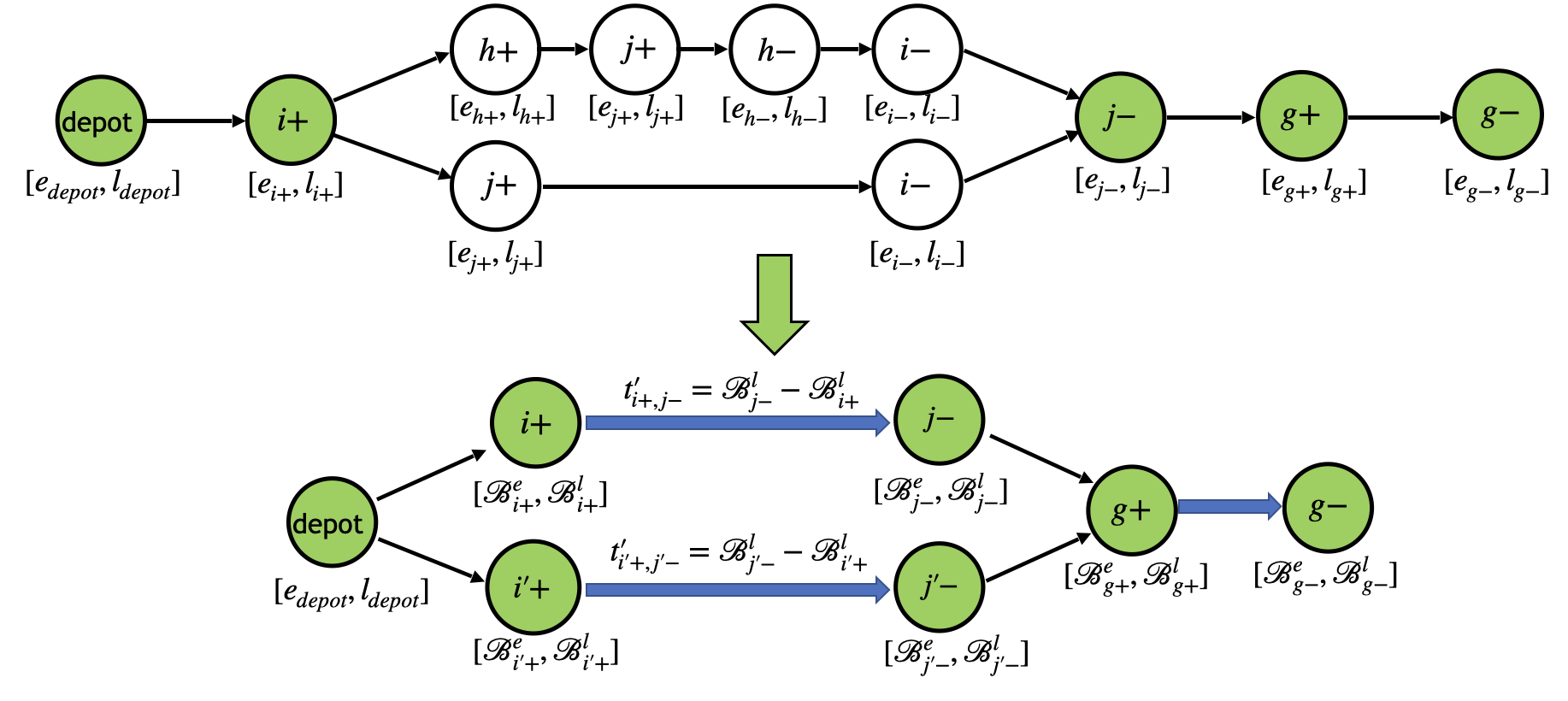}
\caption{\centering Example of constructing arcs in the new graph $G_{sp}$}
\label{new graph2}
\end{figure}

Note that our way of constructing the network is completely different from \cite{alyasiry2019exact}, who build a network by discretizing the time windows of all pickup and delivery nodes with a fixed length of time $\eta$. 
However, we make copies $(i+',j+')$ to distinguish different fragments with the same start/end nodes as they have different restricted time windows.

Next, we work on the new sparser graph $G_{sp}$ instead of $G$, as we show in Theorem \ref{thm:optimization} that $G_{sp}$ preserves all feasible routes over $G$.
\begin{theorem}\label{thm:optimization}
Let $\mathcal{R}$ be a route over graph $G$, and $\mathcal{R}_{sp}$ be the corresponding route over $G_{sp}$. Then $\mathcal{R}$ is feasible if and only if $\mathcal{R}_{sp}$  is feasible.

\end{theorem}
\begin{proof}[Proof of Theorem \ref{thm:optimization}]
Clearly, if $\mathcal{R}_{sp}$ is feasible then $\mathcal{R}$ is feasible. Therefore, we only need to prove the other direction.

Now, we assume that $\mathcal{R}$ is feasible and $\mathcal{A}$ is a \textbf{feasible} schedule over  $\mathcal{R}$. It is enough to show that there is also a \textbf{feasible} schedule over  $\mathcal{R}_{sp}$. This is implied by the following lemma:

\begin{lemma}\label{claim}
For any fragment $\mathcal{F} = \{1, 2,\cdots, m\}$ in $\mathcal{R}$. Let $\mathcal{B}^e, \mathcal{B}^l$ be the constructed vehicle-waiting-time optimal schedule over $\mathcal{F}$. Then there exists $0\leq \delta \leq \mathcal{B}^l_1 - \mathcal{B}^e_1$ such that $\mathcal{A}_1\leq \mathcal{B}^{l}_1 - \delta , \mathcal{B}^{l}_m  - \delta \leq \mathcal{A}_m$.
\end{lemma}

Then we can obtain a feasible schedule over $\mathcal{R}_{sp}$ by replacing the schedule of every fragment $\mathcal{F} = \{1, 2,\cdots, m\}$ to the schedule $\{\mathcal{B}^{l}_1 - \delta , \mathcal{B}^{l}_m  - \delta\}$ over the arc generated from $\mathcal{F}$.

\end{proof}
\begin{proof}[Proof of Lemma \ref{claim}]
According to the construction rules of $\mathcal{B}^l$, we always have $\mathcal{A}_1\leq \mathcal{B}^l_1$. Let $\mathcal{B}^{\delta}$ be the vehicle-waiting-time optimal schedule obtained by moving forward the service begin time by $\delta$ from $\mathcal{B}^l$ and $\Delta_v(\mathcal{B}^l)$ denote the waiting time at node $v$ according to $\mathcal{B}^l$. 

In the proof, we show that we can find a $\delta_0$ satisfying $0\leq \delta_0 \leq \mathcal{B}^l_1 - \mathcal{B}^e_1$ such that $\mathcal{A}_1\leq \mathcal{B}^{\delta_0}_1,  \mathcal{B}^{\delta_0}_m \leq \mathcal{A}_m$. There are two cases:
\begin{enumerate}
    \item If $\Delta_v(\mathcal{B}^l) \neq 0$ for some $1\leq v\leq m$, then it is enough to take $\delta_0 = 0$ as shown in Theorem \ref{thm:segment scheduling}.
    \item If $\Delta_v(\mathcal{B}^l)=0$ for all nodes on $\mathcal{F}$, let $\delta_0$ be the \textbf{maximal} value that satisfies (i) $\mathcal{B}^{\delta_0}$ is feasible and (ii) $ \mathcal{A}_1\leq \mathcal{B}^{\delta_0}_1$.
    There are two cases:
    \begin{enumerate}
        \item If $\mathcal{A}_1 \geqslant \mathcal{B}^{e}_1$, then we have $\mathcal{A}_1 = \mathcal{B}^{\delta_0}_1$. Since $\Delta_v(\mathcal{B}^l)=0$ for all nodes, we derive $\mathcal{B}^{\delta_0}_m \leq \mathcal{A}_m$.
        \item If $\mathcal{A}_1 < \mathcal{B}^{e}_1$, then we have $\mathcal{A}_1 < \mathcal{B}^{\delta_0}_1$ and $\mathcal{B}^{\delta_0} = \mathcal{B}^{e}$. By definition of $\mathcal{B}^{e}$, there must exist a node $u \in \mathcal{F}$ such that $\mathcal{B}^e_{u} = e_{u}$. Therefore, we derive $\mathcal{B}^{\delta_0}_m \leq \mathcal{A}_m$ as we have $\Delta_u(\mathcal{B}^l)=0$ for all nodes.
    \end{enumerate}
    Summing up these cases, we can always find $\delta_0$ such that $\mathcal{A}_1\leq \mathcal{B}^{\delta_0}_1,  \mathcal{B}^{\delta_0}_m \leq \mathcal{A}_m$.
\end{enumerate}
\end{proof}

Then, we design our labeling algorithm on the newly constructed graph. 

\subsubsection{Labeling Algorithm.} \label{labeling algorithm}
We design a labeling algorithm on the new sparser graph $G_{sp}$, where excess-user-ride-time optimality is ensured on each arc. 
The proposed labeling algorithm extends the label at the end of the partial path $\mathcal{P}$. We denote $L_i$ as the label associated with a partial path ends with node $i$.
The forward labeling algorithm extends labels from a source node $o_k \in O$ to a non-predefined sink node $f \in F$. Let a label associated with a partial path $\mathcal{P}$ from $o_k$ to current vertex $i$ be $L_i=\{R_i^{cost},(R_i^{rch_s})_{s \in S},R_i^{tMin},R_i^{tMax},R_i^{rtMax},R_i^{req}\}$, the definition of each resource is described as follows:
\begin{enumerate}
\item $R_i^{cost}$: The reduced cost of the partial route until $i$;
\item $R_i^{rch_s}$: The number of times recharging station $s \in S$ is visited along partial path $\mathcal{P}$;
\item $R_i^{tMin}$: The earliest service start time at vertex $i$ that considers a minimum recharging time (ensuring the battery feasibility up to vertex $i$) at the recharging station if a recharging station is visited along the partial path before reaching $i$;
\item $R_i^{tMax}$: The earliest service start time at vertex $i$ that considers a maximum recharge time (ensuring the time-window feasibility up to vertex $i$) at the recharging station if a recharging station is visited along the partial path before reaching $i$;
\item $R_i^{rtMax}$: 
The maximum recharging time required to fully recharge at vertex $i$. In the case that a recharging station is visited prior to $i$ along $\mathcal{P}$, the vehicle performs a minimum recharge that ensures the battery feasibility up to vertex $i$;
\item $R_i^{req}$: The set of unreachable requests until $i$ along partial path $\mathcal{P}$. A request is said to be ``unreachable" if time window constraints are violated or this request has been visited.
\end{enumerate}

In case no recharging station is visited on partial path $\mathcal{P}$, the value of $R_i^{tMin}$ is equal to $R_i^{tMax}$, indicating the earliest service start time at vertex $i$. $R_i^{rtMax}$ represents the accumulated amount of needed recharging time until $i$. In the initial label at vertex $o_k$, $R_{o_k}^{req}$ is an empty set, $R_{o_k}^{tMin}$ and $R_{o_k}^{tMax}$ are equal to $e_{o_k}$, while all other components are set to zero. 
We extend a label $L_i = \{R_i^{cost},(R_j^{rch_s})_{s \in S},R_i^{tMin},R_i^{tMax},R_i^{rtMax},R_i^{req}\}$ along arc $(i,j)\in A'$ using the following REFs:

\begin{equation}
    R_j^{cost}=R_i^{cost}+ \bar{c}_{i,j}
\end{equation}


\begin{equation}
R_j^{rch_s}= R_i^{rch_s} +
\begin{cases}
    1,\quad & \text{if $j = s$} \\
    0, \quad & \text{otherwise}
\end{cases}
\end{equation}

\begin{equation}
R_j^{tMin}=
\begin{cases}
   \max \{\mathcal{B}_j^e,R_i^{tMin}+t_{i,j}'\} , \quad &\text{if $R_i^{rch}= \emptyset$}  \\
   \max \{\mathcal{B}_j^e,R_i^{tMin}+t_{i,j}'\}+Z_{i,j}, \quad &\text{otherwise} 
\end{cases}
\end{equation}

\begin{equation}
R_j^{tMax}=
\begin{cases}
  \min \{\mathcal{B}_j^l,\max\{\mathcal{B}_j^e,R_i^{tMin}+R_i^{rtMax}+t_{i,j}'\}\}, \quad & \text{if $i \in S$} \\
   \min \{\mathcal{B}_j^l,\max\{\mathcal{B}_j^e,R_i^{tMax}+t_{i,j}'\}\}, \quad & \text{otherwise} 
\end{cases}
\end{equation}

\begin{equation}
R_j^{rtMax}=
\begin{cases}
  R_i^{rtMax}+h_{i,j}', \quad & \text{if $R_i^{rch} = \emptyset$} \\
  \min\{H,\max\{0,R_i^{rtMax}-S_{i,j}\}+h_{i,j}'\}, \quad & \text{otherwise}
\end{cases}
\end{equation}

\begin{equation}
   R_j^{req} = R_i^{req} \cup U_n(R_j^{tMin})
\end{equation}

where in these functions:

\begin{equation}
S_{i,j}(R_i^{tMin},R_i^{tMax},R_i^{rtMax})=
\begin{cases}
  \max\{0, \min\{\mathcal{B}_j^e-R_i^{tMin}-t_{i,j}',R_i^{rtMax}\}\},\quad & \text{if $i\in S$ }\\
  \max\{0, \min\{\mathcal{B}_j^e-R_i^{tMin}-t_{i,j}',R_i^{tMax}-R_i^{tMin}\}\},\quad & \text{otherwise} 
\end{cases}
\end{equation}

\begin{equation}
    Z_{i,j}(R_i^{tMin},R_i^{tMax},R_i^{rtMax}) = \max\{0, \max\{0, R_i^{rtMax}-S_{i,j}(R_i^{tMin},R_i^{tMax},R_i^{rtMax})\}+h_{i,j}'-H\}
\end{equation}

The $S_{i,j}$ is the slack time between the earliest vehicle-waiting-time optimal service start time $\mathcal{B}_j^e$ at $j$ and the earliest arrival time to $j$. If $i$ is a recharging station, $S_{i,j}$ may be equal to the maximum possible recharging time  $R_i^{rtMax}$ at vertex $i$, while at other nodes, it may be equal to $R_i^{tMax}-R_i^{tMin}$. $Z_{i,j}$ is the minimum recharging time accounting for the available slack that the previous recharging station must perform to maintain battery feasibility. $U_n(R_j^{tMin})$ is the function to determine the unreachable nodes from $j$. 

An extension feasibility check is performed while extending label $L_i$ to label $L_j$ via arc $(i,j) \in A'$. The feasibility check rules are presented in the following proposition:

\begin{proposition}
The extension of label $L_i$ to label $L_j$ is feasible if and only of label $L_j$ satisfies:

\begin{equation} \label{tw at j}
    R_j^{tMin} \leqslant \mathcal{B}_j^l,\quad R_j^{tMin} \leqslant R_j^{tMax},\quad  R_j^{req_p} \leqslant 1, \forall p \in P,\quad R_j^{rtMax} \leqslant
    \begin{cases}
      (1-\gamma)H, \quad & j \in F\\
      H, \quad & \text{otherwise}\nonumber
    \end{cases}
\end{equation}



where $R_j^{req_p}$ is the number of times request $p$ is visited along the partial path.
\end{proposition}

If $j$ is a recharging station, then constraint (\ref{visits to rc}) must be considered:


\begin{equation} \label{visits to rc}
    R_j^{rch_j} \leqslant N_s^{max}
\end{equation}

In case of a feasibility violation, the corresponding label will be discarded. Also, it should be mentioned that each time when a fragment is added, the visited customers on this fragment need to be checked. If the fragment contains a visited customer, it should be discarded. The maximum user ride time constraints and capacity constraints are checked when generating fragments.

\begin{definition}
    Suppose that $L^k = \{R_k^{cost},(R_k^{rch_s})_{s \in S},R_k^{tMin},R_k^{tMax},R_k^{rtMax},R_k^{req}\}, k \in {1,2}$ are two labels and the partial path associated to $L^1$ and $L^2$ are $\mathcal{P}_1$ and $\mathcal{P}_2$, respectively. Assuming that $\mathcal{P}_1$, $\mathcal{P}_2$ end at the same node, $L^1$ dominates $L^2$ if and only if:
    \begin{equation}
    R_1^{r} \leqslant R_2^{r}, \forall r \in \{cost, rch, tMin\}
    \end{equation}

    \begin{equation}
    R_1^{req} \subseteq R_2^{req}
    \end{equation}
    
    \begin{equation}
   R_1^{rtMax}-(R_1^{tMax}-R_1^{tMin}) \leqslant R_2^{rtMax} - (R_2^{tMax}-R_2^{tMin}) 
    \end{equation}

    \begin{equation}
    R_1^{rtMax}-(R_2^{tMax}-R_1^{tMin}) \leqslant R_2^{rtMax} 
    \end{equation}
    
\end{definition}

The last two conditions are equivalent to the requirement that: for every service start time $T_2 \in [R_2^{tMin},R_2^{tMax}]$, there exists a service start time $T_1 \in [R_1^{tMin},T_2]$ such that $R_1^{rtMax}-(T_1-R_1^{tMin}) \leqslant R_2^{rtMax}-(T_2-R_2^{tMin})$. In other words, we can always find a service start time $T_1\leqslant T_2$ that does not consume more energy.




\subsection{Cutting Planes} \label{cuts}
To strengthen the continuous MP formulation, we apply two types of cutting planes for instances that are not solved optimally by CG. The first type of cutting plane is the two-path cut, which was initially proposed by \cite{kohl19992} for solving VRPTW and is defined as follows. For a subset $W \subseteq N \cup S$, we define the sets of predecessors of $W$ as $\pi(W) = \{i \in P: i+n \in W, i \notin W\}$, the sets of successors of $W$ as $\sigma(W) = \{i + n\in D: i \in W, i + n \notin W\}$, and the flow enter subset $W \subseteq N \cup S$ is: 
\begin{equation}
    x(W) = \sum \limits_{i \notin W} \sum \limits_{j \in W}x_{i,j}
\end{equation}
where $x_{i,j} = \sum \limits_{k \in K}x_{i,j}^k$ is calculated with the current solution of continuous MP.

We aim at finding the subset $W$ such that $x(W) < 2$ and $k(W) > 1$, where $k(W)$ is the smallest number of vehicles needed to serve all nodes in $W$. The corresponding two-path inequality for such a subset $W$ is: 
\begin{equation} \label{2-path cuts}
   \sum \limits_{\omega \in \Omega} n_{\omega}^W y_{\omega} \geqslant 2 
\end{equation}
where $n_{\omega}^W$ is the number of times column $\omega$ enters $W$, i.e., the number of arcs $(i,j) \in \omega$ such that $i \notin W$ and $j \in W$.  

To separate two-path inequalities, we adapt the greedy heuristic proposed in \cite{kohl19992}. After identifying set $W$ satisfying $x(W) < 2$, we determine whether there exists an elementary path that serves all nodes in $\pi(W) \cup W \cup \delta(W)$ in a feasible way. If no such path can be found, then the subset $W$ defines a valid two-path inequality (\ref{2-path cuts}), which is added to the continuous MP formulation. The dual variable associated with inequality (\ref{2-path cuts}) must be subtracted from $\Bar{c}_{\omega}$ for all arcs $(i,j)$ with $i \notin W$ and $j \in W$.

The second type of cutting planes are subset row inequalities which were first introduced by \cite{jepsen2008subset} to solve the VRPTW. For a subset $W \subseteq N \cup S$ of three elements, the corresponding subset row inequality is:
\begin{equation}
    \sum \limits_{\omega \in \Omega'} m_{\omega}^W y_{\omega} \leqslant 1
\end{equation}
where $m_{\omega}^W = \lfloor{\beta_{\omega}^W}/2 \rfloor$ and $\beta_{\omega}^W$ is the number of visits to a customer in $W$ along route $\omega$. For an elementary route $\omega$, we have $m_{\omega}^W = 1$ if $\omega$ visits two or three customers in $W$ and $m_{\omega}^W = 0$ otherwise. 

Subset row inequalities are separated if we find no two-path cut from the current solution of continuous MP. We separate subset row inequalities by enumerating all subsets of three customers and checking for each subset if the corresponding inequality is violated. The valid inequalities are added to the continuous MP formulation. Note that the dual variables associated with the cuts cannot be integrated into the reduced cost of the labels. Instead, one additional resource attribute is created to record the number of visits to each subset of customers. Also, the dominance rules are modified, as in \cite{jepsen2008subset,desaulniers2011cutting}. Each time, we identify such cuts and select at most $n_{max}^{SRC}$ cuts according to the rules proposed in \cite{desaulniers2008tabu}. In our case, $n_{max}^{SRC} = 10$. 

\subsection{Branching Strategies} \label{branches}
In case the CG algorithm does not obtain integer solutions, we impose branching strategies on fractional solutions to derive integer solutions. We consider two branching strategies in the branch-and-bound search tree: the first branching strategy branches on the total number of routes (as in \cite{desrochers1992new}), and the second branching strategy branches on the total flow of an arc (as in \cite{desaulniers2016exact}). For a fractional solution, we evaluate the branching strategies with the mentioned order and select the first type that can be imposed. When imposing the selected branching strategy, the first type of branching strategy is imposed by adding the respective inequality to the continuous MP. As for the second type of branching strategy, we remove columns that contain incompatible arcs from the column pool and prevent the generation of columns that include these arcs in the labeling algorithm. If we obtain fractional arc flow for several arcs, the arc with a fractional flow that is closest to 0.5 is selected. Each time when a branching strategy is imposed, we create two branches and explore the branch-and-bound tree in a depth-first fashion. 

\section{Computational Experiments} \label{experiments}
In this section, we first describe the benchmark instance sets used to examine the performance of the proposed CG and B\&P algorithms. In order to validate our algorithms, we take the identical problem settings ($N_s^{max} = 1$) as in the literature (i.e., \cite{bongiovanni2019electric, su2023deterministic}).
With this setting, we present our B\&P algorithm results under different minimum battery level restrictions (i.e., $\gamma = 0.1, 0.4, 0.7$). In Section \ref{sec::CG results}, we compare the root node results (i.e., Lagrangian dual bounds and CG primal bounds solving integer RMP on the previously generated columns) to the best-reported B\&C results of \cite{bongiovanni2019electric}. We also test our algorithm on larger-scale instances with up to 8 vehicles and 96 requests, which are first introduced in \cite{su2023deterministic}. We compare the root node results with their best heuristic results. In Section \ref{sec:: BP experimental results}, we report the B\&P algorithm results on all existing instances. To facilitate reading and comparison, Tables \ref{Cordeau instances}, \ref{Uber instances}, and \ref{Ropke instances} consolidate all our algorithm results (including results from CG, CG with cutting planes, and B\&P algorithms), as well as best-known literature results. 
Following the validation of our algorithm, we delve into some intriguing perspectives in Section \ref{interesting perspectives}, where we address several noteworthy aspects: (1) the impacts of considering total excess user ride time in the objectives and its practical advantages, (2) the impacts of having a good initialization of the column pool, (3) the impacts of allowing unlimited visits to recharging stations ($N_s^{max} = \infty$).

All algorithms are coded in Julia 1.7.2 and are run on a standard PC with an Intel(R) Core(TM) i7-8700 CPU at 3.20 GHz and with 32 Gb of RAM using a single thread only. Other packages used in the program are JuMP 1.0.0 and Gurobi, 0.11.1. It should be noted that the results reported in \cite{bongiovanni2019electric} are obtained from a computer with an Intel(R) Core(TM) i7-4790 CPU at 3.60 GHz and with 16 Gb of RAM. 

\subsection{Benchmark Instances}
Three benchmark instance sets are used in the computational experiments. Two are existing benchmark instance sets from \cite{bongiovanni2019electric}. 
The third set corresponds to the large-scale E-ADARP instances introduced in \cite{su2023deterministic}. All the instances are labeled in the form xk-n, where $x \in \{a,u,r\}$, $k$ is the number of vehicles and $n$ is the number of requests. The characteristics of the three benchmark instance sets are as follows:

\begin{itemize}
 \item “a” denotes the E-ADARP instances proposed by \cite{bongiovanni2019electric} that are extended from the standard DARP benchmark instance set of \cite{cordeau2006branch} by supplementing features of electric vehicles and recharging stations. 
 These instances are called type-a instances. For these instances, the number of vehicles is in the range $2\leqslant k \leqslant 5$, and
the number of requests is in the range $16\leqslant n \leqslant 50$.
\item “u” denotes the E-ADARP instances generated by \cite{bongiovanni2019electric} that are based on the ride-sharing data from Uber Technologies. These instances are called type-u instances. For these instances, the number of vehicles is in the range $2\leqslant k \leqslant 5$ and the number of requests is in the range $16\leqslant n \leqslant 50$, as in the type-a instances.
\item “r” denotes the larger E-ADARP instances adopted from \cite{su2023deterministic}. These instances are based on the large-scale instances of \cite{ropke2007models} for the standard DARP and are adapted to the E-ADARP following the same logic as the type-a instances. These instances are called type-r instances.
For these instances, the number of vehicles is in the range $5\leqslant k \leqslant 8$ and
the number of requests is in the range $60\leqslant n \leqslant 96$.
\end{itemize}

The supplemented information on both type-a and type-r instances includes recharging station ID, vehicle capacity, battery capacity, the final state of charge requirement, recharging rates, and discharging rates. For type-a and type-r instances, the vehicle capacity is set to three and the maximum user ride time is 30 minutes. The recharging rate is equal to the discharging rate, and is set to 0.055KWh per minute, according to the design parameter of EAVs given in \url{https://www.hevs.ch/media/document/1/fiche-technique-navettes-autonomes.pdf}. The amount of battery consumption for a sequence of nodes is proportional to the travel time on the sequence. With full effective battery capacity (it is set to 14.85 KWh), approximately 20 nodes can be visited without recharging.

The ride-sharing data of Uber Technologies can be obtained via \url{https://github.com/dima42/uber-gps-analysis/tree/master/gpsdata}. To create type-u instances, the origin/destination locations are extracted from the GPS logs of the city of San Francisco (CA, USA). To calculate the travel time matrix, Dijkstra's shortest path algorithm is applied, assuming a constant speed of 35 km/h. The recharging station locations are obtained through the Alternative Fueling Station Locator from Alternative Fuels Data Center (AFDC). The datasets are available at: \url{https://luts.epfl.ch/wpcontent/uploads/2019/03/e_ADARP_archive.zip}

Following \cite{bongiovanni2019electric}, we set $w_1 = 0.75$, $w_2 = 0.25$. We consider three different $\gamma$ values, namely, $\gamma \in \{0.1,0.4,0.7\}$, representing the low-,medium-, and high-battery-level restriction case, respectively.
Higher values of $\gamma$ result in more tightly constrained instances of the E-ADARP, allowing us to analyze the algorithm's performance from the loosely- to the highly-constrained case.

\subsection{Computational Results} \label{all results}
We first solve the continuous RMP with the proposed CG algorithm. To initialize the column pool, we iterate the DA algorithm 500 times and store all the generated columns. To obtain integer solutions, all the feasible columns generated are used to compute the integer RMP at the end of the CG algorithm. In Tables \ref{Cordeau instances}, \ref{Uber instances}, and \ref{Ropke instances}, we provide an overview of our CG results for type-a, -u, and -r instances, organized into four columns labeled ``$Obj$", ``$LB$", ``$LB\%$", and ``$CPU$". These columns correspond to the upper bound derived from solving the integer RMP, the lower bound obtained from solving continuous RMP, the percentage gaps between best integer solution and lower bound, and the computation time in seconds. To close the integrality gaps in the case it exists, we tried the following ideas: (1) separating two-path cuts and subset row inequalities and adding them to the continuous MP formulation so as to strengthen lower bounds, and (2) embedding the CG approach into a branch-and-bound algorithm, resulting in a B\&P algorithm. 
The respective results are reported in Tables \ref{Cordeau instances}, \ref{Uber instances}, and \ref{Ropke instances}.
It should be noted that one may consider combining idea (1) with (2), resulting in a branch-and-price-and-cut (B\&P\&C) scheme to solve the E-ADARP. Our computational results (Section \ref{sec:: BP experimental results}) indicate that branching seems to be more computationally attractive than adding cuts. Hence, our work does not consider implementing the B\&P\&C algorithm.

In all experiments, the maximum time limit for executing the considered algorithm (e.g., CG algorithm) is set to two hours. For type-r instances, we extend the maximum time limit to 5 hours. The benchmark results from \cite{bongiovanni2019electric} are given in the last three columns in Table \ref{Cordeau instances}, Table \ref{Uber instances}. The meanings of the abbreviations are summarized as follows:
\begin{itemize}
    \item $Obj$ and $Obj_1$: The objective values of integer RMP solutions obtained by the CG algorithm without and with cutting planes being added, respectively;
    \item $LB$ and $LB_1$: Lagrangian dual bound (hereafter LB) values obtained by solving the continuous MP without and with cutting planes being added, respectively;
    \item $LB\%$ and $LB_1\%$: the gaps of the $LB$ and $LB_1$ to the best-obtained integer solution values among all considered algorithms (denoted as $Best$), respectively;
    \item $CPU$ and $CPU_1$: the CPU time in seconds of solving the continuous MP without and with cutting planes being added, including the time of fragment enumeration, all preprocessing works, and column pool initialization;
    \item $Obj'$: the optimal objective values from B\&P;
    \item $LB'$: the lower bounds obtained from B\&P;
    \item $LB'\%$: the gaps of $LB'$ to $Best$;
    \item $N_{node}$: the number of nodes explored in the B\&P tree;
    \item $CPU'$: the CPU time in seconds of applying B\&P algorithm;
    \item $Obj_2$: the best solution found by \cite{bongiovanni2019electric};
    \item $LB_2$: the best lower bounds reported in \cite{bongiovanni2019electric};
    \item \textit{$LB_2\%$}: the gaps between the $LB_2$ and $Best$;
    \item \textit{BKS}: the best known solutions of the DA algorithm of \cite{su2023deterministic};
    \item $\overline{CPU}$: average CPU time (in seconds) per instance;
    \item $\overline{LB}\%$ and $\overline{LB_1}\%$: average value of $LB\%$ and $LB_1\%$ per instance, respectively;
    \item $\overline{LB_2}\%$: average $LB_2\%$ per instance;
    \item NA (Not Available): the considered algorithm does not obtain a feasible integer solution or solve the continuous MP within the given time limit for the respective instance;
    \item $\#opt$: number of optimal solutions obtained by the respective algorithm;
    \item $\#bestlb$: number of times the respective algorithm provides the best lower bound of all considered algorithms;
    \item $\#bestub$: number of times the respective algorithm provides the best integer solution of all considered algorithms;
\end{itemize}

And:
$${LB\% = \dfrac {Best-LB} {Best}}\times 100\%$$ 

$LB_1\%$ and $LB_2\%$ are calculated similarly.


In Tables \ref{Cordeau instances}, \ref{Uber instances}, and \ref{Ropke instances}, we mark the obtained LB from the CG algorithm in italics if the CG algorithm does not converge within the time limit. If the CG algorithm converges, then the obtained LB values are equal to the optimal solution values to continuous RMP. If the corresponding solution is an integer solution, then we obtain the optimal solution of MP, and we indicate it by an asterisk in column ``$Obj$". Otherwise, we try to close the integrality gap by (1) enhancing the continuous MP formulation with cuts or (2) branching directly. 
For (1), we perform the CG algorithm to solve the continuous MP, considering additional dual variables introduced by cuts. If the remaining integrality gap is closed, we mark the obtained optimal solution with an asterisk in column ``$Obj_1$". For (2), we perform the CG algorithm within the branch-and-bound tree, considering the branching strategies of Section \ref{branches}. If the B\&P algorithm terminates within the time limit, the obtained solutions are optimal and are marked with an asterisk in column ``$Obj'$".
Equal/improved integer solutions and LB values compared to those reported in \cite{bongiovanni2019electric} are marked in bold.


In the scenario of $\gamma = 0.4$ and $\gamma = 0.7$, we find strictly better integer solutions than the reportedly optimal solution in \cite{bongiovanni2019electric}. After analyzing their model and parameter settings, we find that in the model of \cite{bongiovanni2019electric}, the employed ``big M" values were not correctly computed. Readers can refer to \ref{supplementary document} for a more in-depth analysis and how the ``big M" values should be set correctly. These problematic results are marked in italics in column ``$Obj_2$", and we add two asterisks to our values in columns ``$Obj$", ``$Obj_1$", and ``$Obj'$" for the concerned instances. These associated instances are a2-24-0.4, a3-30-0.4, a3-36-0.4, a2-24-0.7, a3-24-0.7, a4-24-0.7, u2-24-0.1, u2-24-0.4, u4-40-0.4, u3-30-0.7. The corresponding $LB_2\%$ values are therefore negative.

\subsubsection{CG results on type-a, -u, and -r instances under different minimum battery level restrictions.} \label{sec::CG results}
In this part, we conduct experiments on type-a, -u, and -r instances considering different minimum battery level restrictions $\gamma = 0.1, 0.4, 0.7$. With rising values of $\gamma$, vehicles must have a higher minimum battery level when returning to the depot. Recall that we set $N_s^{max} = 1$ for being comparable to existing benchmark results in these experiments. We set the time limit for performing our algorithms on type-a and -u instances to 2 hours (the same setting as in \cite{bongiovanni2019electric}), while a maximum of 5 hours' execution time is set for each type-r instance.

From our results on type-a and -u instances, our proposed CG algorithm obtains optimal solutions for 50 out of 84 instances before adding cuts to further enhance lower bounds. In addition, we obtain 48 equal integer solutions and 23 new best solutions on previously solved and unsolved instances. As for the quality of obtained lower bounds, we obtain 41 equal lower bounds while enhancing 23 lower bounds (with an improvement of up to 11.88\%). On average, we reduce the overall $LB\%$ by 1.24\%, compared to the best-reported B\&C results in \cite{bongiovanni2019electric}.
Then, the remaining integrality gaps are closed to a large extent by adding cuts to the continuous MP formulation. The proposed CG algorithm solves 16 more instances optimally and further improves 14 lower bounds and 4 integer solutions. In total 66 out of 84 instances are solved optimally at the root node, and we improve the reported lower bounds of \cite{bongiovanni2019electric} by 1.35\% on average.



From our results on large-scale type-r instances (Table \ref{Ropke instances}), we obtain 15 better solutions than the heuristic solutions reported in \cite{su2023deterministic}, and 10 of these solutions are proven optimal. We report the computed lower bound values in the third column, which are the first lower bound values for the type-r instances. For setting $\gamma = 0.7$, neither the CG algorithm nor the DA of \cite{su2023deterministic} is able to provide feasible solutions. We improve a majority of the integer solutions, while four DA heuristic solutions are still better than CG-obtained solutions. 

\begin{table}[!hp]
\renewcommand\arraystretch{0.8}
 \setlength\tabcolsep{0.5pt}
    \caption{Overview of E-ADARP solution approaches on type-a instances under $\gamma=0.1, 0.4, 0.7$}
    \label{Cordeau instances}
    \vspace{-4mm}
    \begin{center}
    \footnotesize
    \begin{tabular}{c| c c c c| c c c c|c c c c c|c c c c}
    \hline
        \textbf{$\gamma = 0.1$}&\multicolumn{4}{c}{CG results} &\multicolumn{4}{c}{CG with cutting planes} &\multicolumn{5}{c}{B\&P results} & \multicolumn{4}{c}{B\&C (Bongiovanni et al., 2019)$^a$}\\
        \hline
        Instance & $Obj$ & $LB$ & $LB\%$ &CPU(s) &$Obj_1$ &$LB_1$ &$LB_1\%$ &CPU(s) & $Obj'$ & $LB'$ & $LB'\%$  &$N_{node}$ & CPU$'$ 
        &$Obj_2$ & $LB_2$ & $LB_2\%$ &CPU(s) \\
        \hline
        a2-16 &\textbf{237.38$^*$}  &\textbf{237.38} &0 &11.1 
        &\textbf{237.38$^*$} &\textbf{237.38} &0 &11.1 
        &\textbf{237.38$^*$} &\textbf{237.38} &0 &1 &11.1 
        & \textbf{237.38$^*$} & \textbf{237.38} &0   & 1.2 \\
        a2-20 &\textbf{279.08$^*$}  & \textbf{279.08} &0 &70.9
        &\textbf{279.08$^*$} &\textbf{279.08} &0 &70.9
        &\textbf{279.08$^*$} &\textbf{279.08} &0 & 1&70.9
        &\textbf{279.08$^*$} & \textbf{279.08} & 0&4.2 \\
        a2-24  &\textbf{346.21$^*$} & \textbf{346.21} &0 & 243.3  
        &\textbf{346.21$^*$} &\textbf{346.21} &0 & 243.3
        &\textbf{346.21$^*$} &\textbf{346.21} &0 &1  & 243.3
        &\textbf{346.21$^*$} & \textbf{346.21} &0 &9.0   \\
        a3-18 &\textbf{236.82$^*$}  & \textbf{236.82} &0 & 5.0  
        &\textbf{236.82$^*$} &\textbf{236.82} &0 &5.0
        &\textbf{236.82$^*$} &\textbf{236.82} &0  &1 &5.0 
        &\textbf{236.82$^*$} & \textbf{236.82}  &0 &4.8 \\
        a3-24 &\textbf{274.80$^*$}  & \textbf{274.80} &0 & 81.2  
        &\textbf{274.80$^*$} &\textbf{274.80} &0 &81.2
        &\textbf{274.80$^*$} &\textbf{274.80} &0  &1 &81.2
        &\textbf{274.80$^*$} & \textbf{274.80} &0  &13.8  \\
        a3-30 &\textbf{413.27$^*$}  & \textbf{413.27} &0 & 221.7  
        &\textbf{413.27$^*$} &\textbf{413.27} &0 &221.7
        &\textbf{413.27$^*$} &\textbf{413.27} &0 & 1 &221.7
        & \textbf{413.27$^*$} &\textbf{413.27}  &0 &102.0  \\
        a3-36 &\textbf{481.17$^*$}  &\textbf{481.17} &0  &730.8   
        &\textbf{481.17$^*$} &\textbf{481.17} &0 &730.8
        &\textbf{481.17$^*$} &\textbf{481.17} &0 &1 &730.8 
        & \textbf{481.17$^*$} &\textbf{481.17}  &0 &106.8 \\
        a4-16 &\textbf{222.49}  & 222.38 &0.05 & 7.5  
        &\textbf{222.49$^*$} &\textbf{222.49}  &0 & 11.4
        &\textbf{222.49$^*$} &\textbf{222.49} &0 &5 &5.5
        &\textbf{222.49$^*$}  &\textbf{222.49}  &0 & 3.6 \\
        a4-24 &\textbf{310.84$^*$}  & \textbf{310.84} &0 & 25.0  
        &\textbf{310.84$^*$} &\textbf{310.84} &0 & 25.0
        &\textbf{310.84$^*$} &\textbf{310.84} &0 &1 & 25.0
        &\textbf{310.84$^*$} &\textbf{310.84} &0  &31.2 \\
        a4-32 &\textbf{393.96$^*$}  &\textbf{393.96} &0 &143.2  
        &\textbf{393.96$^*$} &\textbf{393.96} &0 & 143.2
        &\textbf{393.96$^*$} &\textbf{393.96} &0 & 1 & 143.2
        &\textbf{393.96$^*$} &\textbf{393.96} &0 &612.0\\
        a4-40 &\textbf{453.84$^*$}  &\textbf{453.84} &0  &653.0   
        &\textbf{453.84$^*$} &\textbf{453.84} &0 &653.0
        &\textbf{453.84$^*$} &\textbf{453.84} &0 &1 &653.0
        &\textbf{453.84$^*$} & \textbf{453.84} &0 &517.2 \\
        a4-48 &\textbf{554.54$^*$}  &\textbf{554.54}  &0 &1334.4   
        &\textbf{554.54$^*$} &\textbf{554.54} &0 &1334.4
        &\textbf{554.54$^*$} &\textbf{554.54} &0 &1 &1334.4
        &\textbf{554.54}  &526.96 &5.04 & 7200.0\\
        a5-40 &416.79  &413.48 &0.25 &318.6   
        &\textbf{414.51$^*$} &\textbf{414.51} &0 & 805.2
        &\textbf{414.51$^*$} &\textbf{414.51} &0 &21 &448.3
        &\textbf{414.51$^*$} & \textbf{414.51} &0 &1141.8\\
        a5-50 &\textbf{559.17$^*$} &\textbf{559.17}  &0 &496.2  
        &\textbf{559.17$^*$} &\textbf{559.17} &0 &496.2
        &\textbf{559.17$^*$} &\textbf{559.17} &0 &1  &496.2
        &\textbf{559.17} & 531.15  &5.01 &7200.0\\
        \hline
        Avg & & &\textbf{0.02} & 310.1
        & & &\textbf{0} &345.1 
        & & &\textbf{0} &2.7 &319.3
        & & &0.72 &1210.5\\
        \hline
        \textbf{$\gamma = 0.4$} & $Obj$ & $LB$ & $LB\%$ &CPU(s) 
        &$Obj_1$ &$LB_1$ &$LB_1\%$ &CPU(s) 
        & $Obj'$ & $LB'$ & $LB'\%$  &$N_{node}$ & CPU$'$ 
        &$Obj_2$ & $LB_2$ & $LB_2\%$ &CPU(s) \\
        \hline
        a2-16 & \textbf{237.38$^*$}  & \textbf{237.38} &0 & 12.7 
        &\textbf{237.38$^*$} &\textbf{237.38} &0 &12.7
        &\textbf{237.38$^*$} &\textbf{237.38} &0 &1  &12.7
        &\textbf{237.38$^*$} &\textbf{237.38} &0 &1.8 \\
        a2-20 &\textbf{280.70$^*$}   &\textbf{280.70} &0 &93.8  
        &\textbf{280.70$^*$} &\textbf{280.70} &0 &93.8
        &\textbf{280.70$^*$} &\textbf{280.70} &0 &1  &93.8
        &\textbf{280.70$^*$}  &\textbf{280.70} &0 &49.8 \\
        a2-24  &\textbf{347.04$^{**}$}  &\textbf{347.04}  &0 &267.8  &\textbf{347.04$^{**}$} &\textbf{347.04} &0 &267.8
        &\textbf{347.04$^{**}$} &\textbf{347.04} &0 &1 &267.8
        &\textit{348.04$^*$} &\textit{348.04} &-0.29 &25.2 \\
        a3-18 &\textbf{236.82$^*$} & \textbf{236.82} &0 &5.3   
        &\textbf{236.82$^*$} &\textbf{236.82} &0 &5.3
        &\textbf{236.82$^*$} &\textbf{236.82} &0 &1 &5.3
        &\textbf{236.82$^*$} & \textbf{236.82} &0 &4.2 \\
        a3-24 &\textbf{274.80$^*$}  & \textbf{274.80} &0 &69.7  
        &\textbf{274.80$^*$} &\textbf{274.80} &0 &69.7
        &\textbf{274.80$^*$} &\textbf{274.80} &0 &1 &69.7 
        &\textbf{274.80$^*$} &\textbf{274.80}  &0 &16.8 \\
        a3-30 &\textbf{413.34$^{**}$}  &\textbf{413.34} &0  &306.9  &\textbf{413.34$^{**}$} &\textbf{413.34} &0 &306.9
        &\textbf{413.34$^{**}$} &\textbf{413.34} &0 &1 & 306.9
        &\textit{413.37$^*$} &\textit{413.37} &-0.01 &99.0 \\
        a3-36 &483.83$^{**}$  &481.85 &0.41 &1273.5  
        &\textbf{482.75$^{**}$} &\textbf{482.64} &0.02 & 7200.0
        &\textbf{482.75$^{**}$} &\textbf{482.75} &0 &13 &5154.1
        &\textit{484.14$^*$} &\textit{484.14} &-0.06 &306.6 \\
        a4-16&\textbf{222.49} & 222.38 &0.05 & 6.5   
        &\textbf{222.49$^*$} &\textbf{222.49} &0 & 10.1
        &\textbf{222.49$^*$} &\textbf{222.49} &0 &5 &3.3
        &\textbf{222.49$^*$} &\textbf{222.49} &0 &5.4 \\
        a4-24 &\textbf{311.03$^*$}  &\textbf{311.03} &0  &18.4  
        &\textbf{311.03$^*$} &\textbf{311.03}  &0 &18.4 
        &\textbf{311.03$^*$} &\textbf{311.03}  &0 &1  &18.4 
        &\textbf{311.03$^*$} &\textbf{311.03} &0 &39.6 \\
        a4-32 &\textbf{394.26$^*$} &\textbf{394.26} &0 & 199.3  
        &\textbf{394.26$^*$} &\textbf{394.26} &0 &199.3
        &\textbf{394.26$^*$} &\textbf{394.26} &0 &1 &199.3
        &\textbf{394.26$^*$} &\textbf{394.26} &0 &681.6 \\
        a4-40 &\textbf{453.84$^*$} & \textbf{453.84} &0 & 792.0  
        &\textbf{453.84$^*$} &\textbf{453.84} &0 & 792.0
        &\textbf{453.84$^*$} &\textbf{453.84} &0 &1 & 792.0
        &\textbf{453.84$^*$} &\textbf{453.84} &0 &417.6\\
        a4-48  &\textbf{554.60$^*$} &\textbf{554.60} &0 &2292.0  
        &\textbf{554.60$^*$} &\textbf{554.60} &0 & 2292.0
        &\textbf{554.60$^*$} &\textbf{554.60} &0 &1 & 2292.0
        &\textbf{554.60} &529.22 &4.58 &7200.0 \\
        a5-40 &415.25  &413.48  &0.25 &307.8  
        &\textbf{414.51$^*$} &\textbf{414.51} &0 & 481.6
        &\textbf{414.51$^*$} &\textbf{414.51} &0 &3 &323.6
        &\textbf{414.51$^*$} &\textbf{414.51} &0 &1221.0 \\
        a5-50 &\textbf{559.51$^*$}  &\textbf{559.51}  &0 &1957.1  
        &\textbf{559.51$^*$} &\textbf{559.51} &0 &1957.1 
         &\textbf{559.51$^*$} &\textbf{559.51} &0 &1 &1957.1 
        &560.50 &528.91 &5.47 &7200.0 \\
        \hline
        Avg & & &\textbf{0.05} &543.1
        & & &\textbf{0} &979.1
        & & &\textbf{0} &2.3 &821.1
        & & &0.74 &1233.5\\
        \hline
        \textbf{$\gamma = 0.7$} & $Obj$ & $LB$ & $LB\%$ &CPU(s) 
        &$Obj_1$ &$LB_1$ &$LB_1\%$ &CPU(s) 
        &$Obj'$ & $LB'$ & $LB'\%$  &$N_{node}$ & CPU$'$ 
        &$Obj_2$ & $LB_2$ & $LB_2\%$ &CPU(s) \\
        \hline
        a2-16 &\textbf{240.66$^*$} &\textbf{240.66} &0 &20.8 
        &\textbf{240.66$^*$} &\textbf{240.66} &0 &20.8
        &\textbf{240.66$^*$} &\textbf{240.66} &0 &1 &20.8
        &\textbf{240.66$^*$} &\textbf{240.66} &0 &5.4 \\
        a2-20 &\textbf{293.27}  &\textbf{292.73} &0.18 &1108.8 
        &\textbf{293.27$^*$} &\textbf{293.27} &0 & 1733.8
        &\textbf{293.27$^*$} &\textbf{293.27} &0  &5 &1758.6
        &NA &287.17 &2.08 &7200.0 \\
        a2-24 &\textbf{353.18$^{**}$}  &\textbf{353.18} &0 &1241.9 &\textbf{353.18$^{**}$} &\textbf{353.18} &0 &1241.9 
        &\textbf{353.18$^{**}$} &\textbf{353.18} &0 &1 &1241.9
        &\textit{358.21$^*$} &\textit{358.21}  &-1.42 &961.2 \\
        a3-18 &\textbf{240.58$^*$} &\textbf{240.58} &0 &14.1  
        &\textbf{240.58$^*$} &\textbf{240.58} &0 & 14.1
        &\textbf{240.58$^*$} &\textbf{240.58} &0 &1 & 14.1
        &\textbf{240.58$^*$} &\textbf{240.58}  &0 &48.0 \\
        a3-24  &\textbf{275.97$^{**}$}  &\textbf{275.97} &0 &100.8 &\textbf{275.97$^{**}$} &\textbf{275.97} &0 & 100.8
        &\textbf{275.97$^{**}$} &\textbf{275.97} &0 &1 & 100.8
        &\textit{277.72$^*$} &\textit{277.72} &-0.63  &152.4 \\
        a3-30 &\textbf{426.40} &\textbf{422.57} &0.56 &1171.6 
        &\textbf{424.93$^*$} &\textbf{424.93} &0 &3401.9 
        &\textbf{424.93$^*$} &\textbf{424.93} &0 &3 &1893.3
        &NA &417.06 &1.85  &7200.0 \\
        a3-36 &500.57  &\textbf{491.80} &0.45 &2191.6 
        &\textbf{494.04} &\textbf{493.71} &0.07 &7200.0
        &\textbf{494.04} &\textbf{494.01} &0.01 &20 &7200.0
        &494.04 &485.91 &1.65 &7200.0  \\
        a4-16 &\textbf{223.13$^*$}  &\textbf{223.13} &0 &8.4 
        &\textbf{223.13$^*$} &\textbf{223.13} &0 &8.4
        &\textbf{223.13$^*$} &\textbf{223.13} &0 &1 &8.4
        &\textbf{223.13$^*$} &\textbf{223.13} &0 &67.2\\
        a4-24 &318.24 &316.24 &0.13 &84.3 
        &\textbf{316.65$^{**}$} &\textbf{316.65} &0 &129.7
        &\textbf{316.65$^{*}$} &\textbf{316.65} &0 &7 &164.4
        &\textit{318.21$^*$} &\textit{318.19} &-0.49 &1834.8 \\
        a4-32 &\textbf{397.87}  &\textbf{396.84} &0.26 &734.7  
        &\textbf{397.87} &\textbf{397.78} &0.02 & 7200.0
        &\textbf{397.87$^{*}$} &\textbf{397.87} &0 &9 &1638.6
        &430.07 &387.99  &2.48 &7200.0 \\
        a4-40 &\textbf{467.72}  &\textbf{466.96} &0.16 &3252.4  
        &\textbf{467.72$^*$} &\textbf{467.72} &0 &5238.4
        &\textbf{467.72$^*$} &\textbf{467.72} &0 &9 &3987.4
        &NA &443.62  &5.15 &7200.0 \\
        a4-48  &NA  &\textit{476.54} &NA &7200.0  
        &NA  &\textit{476.54} &NA &7200.0
        &NA  &\textit{476.54} &NA &1 &7200.0
        &NA &524.92 &NA &7200.0 \\
        a5-40 &\textbf{418.75} &\textbf{417.25} &0.36 &2182.2 
        &\textbf{418.75$^*$} &\textbf{418.75} & 0 &6808.1
        &\textbf{418.75$^*$} &\textbf{418.75} & 0  &15 &1572.0
        &447.63 &405.99  &4.78 &7200.0 \\
        a5-50 &NA  &\textit{176.58} &NA &7200.0  
        &NA  &\textit{176.58} &NA &7200.0 
        &\textbf{690.79}  &\textit{327.84} &NA  & 1 &7200.0  
        &NA &522.37 &NA &7200.0 \\
        \hline
        Avg & & &\textbf{0.32} & 1893.7
        & & &\textbf{0.01} &3394.5
        & & &\textbf{0} &5.4 &2428.6
        & & &1.70 &4333.5\\
        \hline
        Summary &$\#opt$ & $\#best lb$ &$\overline{LB}\%$ &$\#best ub$   
        &$\#opt$ & $\#best lb$ &$\overline{LB_1}\%$ & $\#best ub$
        &$\#opt$ & $\#best lb$ &$\overline{LB'}\%$ &$\overline{N_{node}}$ &$\#best ub$  
        &$\#opt$ & $\#best lb$ & $\overline{LB_2}\%$  & $\#best ub$\\
        & 28 &34 &0.13 &36 
        &37 &40 &0.003 &40
        &39 &40 &0 &3.5 &40
        &24 &26 &1.06 &28\\
        \hline
    \end{tabular}
    \begin{tablenotes}
    \footnotesize
     \item a: Due to incorrect big M values, some of the reported results of \cite{bongiovanni2019electric} are higher than the actual optimal values. Those results are highlighted in italics;
   \end{tablenotes}
    \end{center}
\end{table}

\begin{table}[!hp]
\renewcommand\arraystretch{0.8}
\setlength\tabcolsep{0.5pt}
    \caption{Overview of E-ADARP solution approaches on type-u instances under $\gamma=0.1, 0.4, 0.7$}
    \vspace{-4mm}
    \label{Uber instances}
    \begin{center}
    \footnotesize
    \begin{tabular}{c| c c c c| c c c c| c c c c c|c c c c}
    \hline
        \textbf{$\gamma = 0.1$}&\multicolumn{4}{c}{CG results} &\multicolumn{4}{c}{CG with cutting planes} &\multicolumn{5}{c}{B\&P results} & \multicolumn{4}{c}{B\&C (Bongiovanni et al., 2019)$^a$}\\
        \hline
        Instance & $Obj$ & $LB$ & $LB\%$ &CPU(s) 
        &$Obj_1$ &$LB_1$ &$LB_1\%$ &CPU(s)
        &$Obj'$ & $LB'$ & $LB'\%$  &$N_{node}$ & CPU$'$
        &$Obj_2$ & $LB_2$ & $LB_2\%$ &CPU(s) \\
        \hline
        u2-16 &\textbf{57.61}  &57.08 &0.92  &43.6 
        &\textbf{57.61$^*$} &\textbf{57.61} &0 &72.1
        &\textbf{57.61$^*$} &\textbf{57.61} &0 &3 &22.2
        &\textbf{57.61$^*$} &\textbf{57.61} &0 &21.0\\
        u2-20 &\textbf{55.59$^*$}  &\textbf{55.59} &0 &69.3  
        &\textbf{55.59$^*$} &\textbf{55.59} &0 &69.3
        &\textbf{55.59$^*$} &\textbf{55.59} &0 &1 &69.3
        &\textbf{55.59$^*$} &\textbf{55.59} &0 &9.6\\
        u2-24 &91.36  &90.55 &0.12 &851.7 
        &\textbf{90.66$^{**}$} &\textbf{90.66} &0 &7001.9
        &\textbf{90.66$^{**}$} &\textbf{90.66} &0 &7 &1649.0
        &\textit{91.27$^*$} & \textit{91.27} & -0.67 &432.0  \\
        u3-18 &\textbf{50.74$^*$} &\textbf{50.74} &0 &14.0 
        &\textbf{50.74$^*$} &\textbf{50.74} &0 &14.0
        &\textbf{50.74$^*$} &\textbf{50.74} &0 &1 &14.0
        &\textbf{50.74$^*$} &\textbf{50.74} &0 &10.8 \\
        u3-24 &\textbf{67.56$^*$} &\textbf{67.56} &0 &40.3  
        &\textbf{67.56$^*$} &\textbf{67.56} &0 &40.3
        &\textbf{67.56$^*$} &\textbf{67.56} &0 &1 &40.3
        &\textbf{67.56$^*$} &\textbf{67.56} &0 &130.2 \\
        u3-30 &\textbf{76.75$^*$} &\textbf{76.75} &0 &570.8  
        &\textbf{76.75$^*$} &\textbf{76.75} &0 &570.8
        &\textbf{76.75$^*$} &\textbf{76.75} &0 &1 &570.8
        &\textbf{76.75$^*$} &\textbf{76.75} &0 &438.0 \\
        u3-36 &104.39 &103.94 &0.10 &1894.6  
        &\textbf{104.04$^*$} &\textbf{104.04} &0 &3293.4
        &\textbf{104.04$^*$} &\textbf{104.04} &0 &11 &4204.6
        &\textbf{104.04$^*$} &\textbf{104.04} &0 &1084.8 \\
        u4-16 &\textbf{53.58$^*$} &\textbf{53.58} &0 &1.9  
        &\textbf{53.58$^*$} &\textbf{53.58} &0 &1.9
        &\textbf{53.58$^*$} &\textbf{53.58} &0 &1 &1.9
        &\textbf{53.58$^*$} &\textbf{53.58} &0 &48.0 \\
        u4-24 &\textbf{89.83$^*$}  &\textbf{89.83} &0 &25.5  
        &\textbf{89.83$^*$} &\textbf{89.83} &0 &25.5
        &\textbf{89.83$^*$} &\textbf{89.83} &0 &1 &25.5
        &\textbf{89.83$^*$} &\textbf{89.83} &0 &13.2 \\
        u4-32 &\textbf{99.29$^*$}  &\textbf{99.29} &0 &366.2  
        &\textbf{99.29$^*$} &\textbf{99.29} &0 &366.2
        &\textbf{99.29$^*$} &\textbf{99.29} &0 &1 &366.2
        &\textbf{99.29$^*$} &\textbf{99.29} &0 &1158.6 \\
        u4-40 &\textbf{133.11$^*$}  &\textbf{133.11} &0 &1376.6  
        &\textbf{133.11$^*$} &\textbf{133.11} &0 &1376.6
        &\textbf{133.11$^*$} &\textbf{133.11} &0 &1 &1376.6
        &\textbf{133.11$^*$} &\textbf{133.11} &0 &185.4 \\
        u4-48 &\textbf{148.08} &\textbf{147.02} &0.72 &7200.0  
        &\textbf{148.08} &\textbf{147.02} &0.72 &7200.
        &\textbf{148.08} &\textbf{147.02} &0.72 &1 &7200.0
        &148.30 &134.48 &9.18 &7200.0 \\
        u5-40 &\textbf{121.86$^*$} &\textbf{121.86} &0 &496.9  
        &\textbf{121.86$^*$} &\textbf{121.86} &0 &496.9
        &\textbf{121.86$^*$} &\textbf{121.86} &0 &1 &496.9
        &\textbf{121.86} &114.12 &6.35 &7200.0 \\
        u5-50 &\textbf{142.82}  &\textbf{142.75} &0.05 &7100.9 
        &\textbf{142.82} &\textbf{142.75} &0.05 &7200.0
        &\textbf{142.82} &\textbf{142.77} &0.04 &7 &7200.0
        &143.10 &132.69 &7.09 &7200.0 \\
        \hline
        Avg & & &\textbf{0.14} &1432.3
        & & &\textbf{0.06} &1659.8
        & & &\textbf{0.06} &9.5 &1659.8
        & & &1.66 &1795.1\\
        \hline
        \textbf{$\gamma = 0.4$} & $Obj$ & $LB$ & $LB\%$ &CPU(s) 
        &$Obj_1$ &$LB_1$ &$LB_1\%$ &CPU(s)
        &$Obj'$ & $LB'$ & $LB'\%$  &$N_{node}$ & CPU$'$
        &$Obj_2$ & $LB_2$ & $LB_2\%$ &CPU(s) \\
        \hline
        u2-16 &\textbf{57.65$^*$} &\textbf{57.65} &0 &53.2  
        &\textbf{57.65$^*$} &\textbf{57.65} &0 &53.2
        &\textbf{57.65$^*$} &\textbf{57.65} &0 &1 &53.2
        &\textbf{57.65$^*$} &\textbf{57.65} &0 &25.8\\
        u2-20 &56.61  &56.06 &0.50 &407.3  
        &\textbf{56.34$^*$} &\textbf{56.34} &0 &665.2
         &\textbf{56.34$^*$} &\textbf{56.34} &0 &3 &281.3
        &\textbf{56.34$^*$} &\textbf{56.34} &0 &12.0 \\
        u2-24 &91.62  &90.80 &0.51 &1140.2  
        &\textbf{91.27$^{**}$} &90.95 &0.33 &4266.0
        &\textbf{91.06$^{**}$} &\textbf{91.06} &0 &25 &6917.0
        &\textit{91.63$^*$} &\textit{91.63} &-0.39 &757.2 \\
        u3-18 &\textbf{50.74$^*$}  &\textbf{50.74} &0 &15.0  
        &\textbf{50.74$^*$} &\textbf{50.74} &0 &15.0
        &\textbf{50.74$^*$} &\textbf{50.74} &0 &1 &15.0
        &\textbf{50.74$^*$} &\textbf{50.74} &0 &13.8 \\
        u3-24 &\textbf{67.56$^*$}  &\textbf{67.56} &0 &58.5  
        &\textbf{67.56$^*$} &\textbf{67.56} &0 &58.5
        &\textbf{67.56$^*$} &\textbf{67.56} &0 &1 &58.5
        &\textbf{67.56$^*$} &\textbf{67.56} &0 &220.8 \\
        u3-30 &\textbf{76.75$^*$}  &\textbf{76.75} &0 &500.2 
        &\textbf{76.75$^*$} &\textbf{76.75} &0 &500.2
        &\textbf{76.75$^*$} &\textbf{76.75} &0 &1 &500.2
        &\textbf{76.75$^*$} &\textbf{76.75} &0 &336.6 \\
        u3-36 &\textbf{104.06$^*$} &\textbf{104.06} &0 &1769.0  
        &\textbf{104.06$^*$} &\textbf{104.06} &0 &1769.0
        &\textbf{104.06$^*$} &\textbf{104.06} &0 &1 &1769.0
        &\textbf{104.06$^*$} &\textbf{104.06} &0 &2010.0 \\
        u4-16 &\textbf{53.58$^*$}  &\textbf{53.58} &0 &2.2   
        &\textbf{53.58$^*$} &\textbf{53.58} &0 &2.2
        &\textbf{53.58$^*$} &\textbf{53.58} &0 &1 &2.2
        &\textbf{53.58$^*$} &\textbf{53.58} &0 &44.4 \\
        u4-24 &\textbf{89.83$^*$}  &\textbf{89.83} &0 &30.8   
        &\textbf{89.83$^*$} &\textbf{89.83} &0 &30.8
        &\textbf{89.83$^*$} &\textbf{89.83} &0 &1 &30.8
        &\textbf{89.83$^*$} &\textbf{89.83} &0 &28.2 \\
        u4-32 &\textbf{99.29$^*$}  &\textbf{99.29} &0 &393.4  
        &\textbf{99.29$^*$} &\textbf{99.29} &0 &393.4
        &\textbf{99.29$^*$} &\textbf{99.29} &0 &1 &393.4
        &\textbf{99.29$^*$} &\textbf{99.29} &0 &2667.6 \\
        u4-40 &\textbf{133.78$^{**}$}  &133.37 &0.31 &2186.1  
        &\textbf{133.78$^{**}$} &\textbf{133.46} &0.24 &3958.2
        &\textbf{133.78$^{**}$} &\textbf{133.70} &0.06 &32 &7200.0
        &\textit{133.91$^*$} &\textit{133.91} &-0.10 &2653.2 \\
        u4-48 &\textbf{147.63}  &\textbf{\textit{146.37}} &0.85 &7200.0  
        &\textbf{147.63} &\textbf{\textit{146.37}} &0.85 &7200.0
        &\textbf{147.63} &\textbf{\textit{146.37}} &0.85 & 1 &7200.0
        &NA &133.86 &9.33 &7200.0 \\
        u5-40 &\textbf{121.96$^{*}$}  &\textbf{121.96} &0 &723.7  
        &\textbf{121.96$^*$} &\textbf{121.96} &0 &723.7
        &\textbf{121.96$^*$} &\textbf{121.96} &0 &1 &723.7
        &122.23 &112.58 &7.69 &7200.0 \\
        u5-50 &\textbf{142.84} &\textbf{142.75} &0.06 &7200.0  
        &\textbf{142.84} &\textbf{142.75} &0.06 &7200.0
        &\textbf{142.84} &\textbf{142.75} &0.06 &1 &7200.0
        &143.14 &134.09 &6.13 &7200.0 \\
        \hline
        Avg & & &\textbf{0.16} & 1548.5
        & & &\textbf{0.11} &1916.8
        & & &\textbf{0.07} &5.1 &2310.3
        & & &1.69 &2169.3\\
        \hline
        \textbf{$\gamma = 0.7$} & $Obj$ & $LB$ & $LB\%$ &CPU(s) 
        &$Obj_1$ &$LB_1$ &$LB_1\%$ &CPU(s) 
        &$Obj'$ & $LB'$ & $LB'\%$  &$N_{node}$ & CPU$'$
        &$Obj_2$ & $LB_2$ & $LB_2\%$ &CPU(s) \\
        \hline
        u2-16 &60.01 &58.48 &1.20 &89.6  
        & \textbf{59.19$^*$}&\textbf{59.19} &0 &2597.2
        &\textbf{59.19$^*$}&\textbf{59.19} &0 &39 &679.8
        &\textbf{59.19$^*$} &\textbf{59.19} &0 &338.4 \\
        u2-20 &\textbf{56.86$^*$} &\textbf{56.86} &0 &337.0  
        &\textbf{56.86$^*$} &\textbf{56.86} &0 &337.0
        &\textbf{56.86$^*$} &\textbf{56.86} &0 &1 &337.0
        &\textbf{56.86$^*$} &\textbf{56.86} &0 &72.0 \\
        u2-24 &\textbf{92.17} &\textbf{91.48} &0.75 &2056.3  
        &\textbf{92.17} &\textbf{91.71} &0.50 &7200.0
        &\textbf{92.01} &91.60 &0.52 &30 &7200.0
        &NA &90.83 &1.45 &7200.0 \\
        u3-18 &\textbf{50.99}  &50.94 &0.10 &65.4  
        &\textbf{50.99$^*$} &\textbf{50.99} &0 &116.4
        &\textbf{50.99$^*$} &\textbf{50.99} &0  &3 &45.2
        &\textbf{50.99$^*$} &\textbf{50.99} &0 &24.0 \\
        u3-24 &68.44  &68.00 &0.57 &119.9  
        &68.44 &68.12 &0.47 &176.2
        &\textbf{68.34$^{**}$} &\textbf{68.34} &0 &27 &541.9
        &\textbf{68.39$^*$} &\textbf{68.39} &0 &400.2 \\
        u3-30 &\textbf{77.41$^{**}$} &77.33 &0.10 &955.3  
        &\textbf{77.41$^{**}$} &\textbf{77.41} &0 &1778.3
        &\textbf{77.41$^{**}$} &\textbf{77.41} &0  &5 &1725.5
        &\textit{78.14$^*$} &\textit{78.14} &-0.94 &3401.4 \\
        u3-36 &106.45  &\textbf{105.46} &0.31 &7200.0  
        &106.45 &\textbf{105.46} &0.31 &7200.0
        &\textbf{105.79} &\textbf{105.78} &0.01 & 15 &7200.0
        &\textbf{105.79} &104.37 &1.34 & 7200.0 \\
        u4-16 &\textbf{53.87$^*$}  &\textbf{53.87} &0 &4.0 
        &\textbf{53.87$^*$} &\textbf{53.87} &0 &4.0
        &\textbf{53.87$^*$} &\textbf{53.87} &0 &1 &4.0 
        &\textbf{53.87$^*$} &\textbf{53.87} &0 &88.8 \\
        u4-24 &\textbf{89.96$^*$} &\textbf{89.96} &0 &82.3  
        &\textbf{89.96$^*$} &\textbf{89.96} &0 &82.3
        &\textbf{89.96$^*$} &\textbf{89.96} &0 &1 &82.3
        &\textbf{89.96$^*$} &\textbf{89.96} &0 &22.8 \\
        u4-32 &\textbf{99.50$^*$}  &\textbf{99.50} &0 &1028.6  
        &\textbf{99.50$^*$} &\textbf{99.50} &0 &1028.6
        &\textbf{99.50$^*$} &\textbf{99.50} &0 &1 &1028.6
        &\textbf{99.50$^*$} &\textbf{99.50} &0 &2827.2 \\
        u4-40 &135.33 &\textbf{134.56} &0.54 &1959.5  
        &\textbf{135.29} &\textbf{134.65} &0.47 &5555.5
        &135.76 &\textbf{134.96} &0.59 &43 &7200.0
        &NA &133.01 &2.46 &7200.0 \\
        u4-48 &\textbf{185.16} &NA &NA &7200.0 
        &\textbf{185.16} &NA &NA &7200.0
        &\textbf{185.16} &NA &NA &1 &7200.0
        &NA &\textbf{132.49} &NA &7200.0 \\
        u5-40 &124.01 &\textbf{122.82} &0.81 &4774.6  
        &\textbf{123.82} &\textbf{122.87} &0.77 &7200.0
        &\textbf{122.93$^{*}$} &\textbf{122.93} &0 &7 &6513.9
        &NA &109.28 &11.88 &7200.0 \\
        u5-50 &195.72  &\textit{132.79} &NA &7200.0  
        &195.72 &\textit{132.79} &NA &7200.0
        &195.72 &\textit{132.79} &NA &1 &7200.0
        &\textbf{144.36} &\textbf{133.33} &7.64 &7200.0 \\
        \hline
        Avg & & &\textbf{0.36} & 2362.3 
        & & &\textbf{0.21} &3405.4
        & & &\textbf{0.09} &12.4 &3354.1
        & & &1.98 &3598.2\\
        \hline
        Summary&$\#opt$ & $\#best lb$ &$\overline{LB}\%$ & $\#best ub$ &$\#opt$ & $\#best lb$ &$\overline{LB}\%$ & $\#best ub$
        &$\#opt$ & $\#best lb$ &$\overline{LB}\%$ &$\overline{N_{node}}$ & $\#best ub$
        &$\#opt$ & $\#best lb$ & $\overline{LB'}\%$  & $\#best ub$\\
        &22 &30 &0.22 &35 
        &29 &38 &0.13 &39
        &32 &39 &0.07 &9 &39
        &26 &28 &1.78 &29\\
        \hline
    \end{tabular}
    \begin{tablenotes}
    \footnotesize
     \item a: Due to incorrect big M values, some of the reported results of \cite{bongiovanni2019electric} are higher than the actual optimal values. Those results are highlighted in italics;
   \end{tablenotes}
    \end{center}
\end{table}

\begin{table}[h!]
    \centering
    \begin{threeparttable}
    \caption{Overview of E-ADARP solution approaches on type-r instances under $\gamma=0.1,0.4$}
    \label{Ropke instances}
    \setlength{\belowcaptionskip}{0.5cm}
    \setlength\tabcolsep{4.5pt}
    \footnotesize
    \begin{tabular}{c |c c c |c c c|c c c c c| c }
    \toprule
        \textbf{$\gamma = 0.1$}&\multicolumn{3}{c}{CG results}&\multicolumn{3}{c}{CG with cutting planes}&\multicolumn{5}{c}{B\&P results}&DA \\
        \hline
        Instance & $Obj$ & $LB$ &CPU(s) 
        & $Obj_1$ & $LB_1$ &CPU(s) 
        &$Obj'$ & $LB'$ & $LB'\%$  &$N_{node}$ & CPU$'$
        &BKS \\
        \hline
        r5-60 &\textbf{687.80} &\textbf{682.27} &11263.2  
        &\textbf{683.39$^*$} &\textbf{683.39} &13983.7 
        &\textbf{683.39$^*$} &\textbf{683.39} &0  &7 &7817.0
        &691.83 \\
        r6-48 &\textbf{506.45$^*$} & \textbf{506.45} &1109.1  
        &\textbf{506.45$^*$} & \textbf{506.45} &1109.1
        &\textbf{506.45$^*$} & \textbf{506.45} &0 &1 &1109.1
        &506.72 \\
        r6-60  & \textbf{692.25} & \textbf{689.31} & 2389.2  
        &\textbf{689.45$^*$} &\textbf{689.45} &16387.4
        &\textbf{689.45$^*$} &\textbf{689.45} &0 &15 &3695.7
        &692.00 \\
        r6-72  & \textbf{761.34} & \textbf{\textit{748.90}} & 18000.0 
        & \textbf{761.34} & \textbf{\textit{748.90}} & 18000.0
        & \textbf{761.34} & \textbf{\textit{748.90}} &NA &1 &18000.0
        &777.44  \\
        r7-56 & \textbf{613.19} & \textbf{611.97} & 1402.5 
        &\textbf{612.02$^*$} &\textbf{612.02} &4975.3
        &\textbf{612.02$^*$} &\textbf{612.02} &0 &3 &810.8
        &613.10  \\
        r7-70 & \textbf{760.23}  & \textbf{753.56} & 7275.6  
        &\textbf{754.28} &\textbf{753.56} &18000.0
        &\textbf{754.28$^*$} &\textbf{754.28} & 0 &31 &18000.0
        &760.90 \\
        r7-84 &975.26   &\textbf{\textit{638.59}} &18000.0  
        &975.26   &\textbf{\textit{638.59}} &18000.0
        &975.26   &\textbf{\textit{638.59}} &NA &1 &18000.0
        &\textbf{889.38} \\
        r8-64& \textbf{632.22$^*$}  & \textbf{632.22}  & 3246.5 
        & \textbf{632.22$^*$}  & \textbf{632.22}  & 3246.5 
        & \textbf{632.22$^*$}  & \textbf{632.22} &0 &1 &3246.5
        &641.99 \\
        r8-80 & \textbf{788.99$^*$} & \textbf{788.99} & 14563.8 
        &\textbf{788.99$^*$} & \textbf{788.99} & 14563.8
        &\textbf{788.99$^*$} & \textbf{788.99} &0 &1 &14563.8
        &803.52 \\
        r8-96 &1329.75  &\textbf{\textit{651.32}}  &18000.0 
        &1329.75  &\textbf{\textit{651.32}}  &18000.0
        &1329.75  &\textbf{\textit{651.32}} &NA &1 &18000.0
        &\textbf{1053.11} \\
        \hline
        Avg & &  &9525.0   & & &12626.6 & & &0 &6.2 &10324.3 & \\
        \hline
        $\gamma = 0.4$ & $Obj$ & $LB$ &CPU(s) 
        & $Obj_1$ & $LB_1$ &CPU(s) 
        &$Obj'$ & $LB'$ & $LB'\%$  &$N_{node}$ & CPU$'$
        &BKS \\
        \hline
        r5-60 &\textbf{686.85}  &\textbf{682.76} &15555.3  &\textbf{686.85} &\textbf{683.56} &18000.0
        &\textbf{684.49$^*$} &\textbf{684.49} &0 &9 &17281.3
        &697.97 \\
        r6-48 &\textbf{506.45$^*$}  &\textbf{506.45}  &1195.6  &\textbf{506.45$^*$}  &\textbf{506.45}  &1195.6
        &\textbf{506.45$^*$}  &\textbf{506.45} &0 &1 &1195.6 
        &506.91  \\
        r6-60  & 695.69 &\textbf{688.81} &4109.1 
        &\textbf{689.48} &\textbf{689.33} &18000.0
        &\textbf{689.46$^*$} &\textbf{689.46} &0 &3 &2521.1
        &694.78 \\
        r6-72  & NA  & NA &18000.0  & NA  & NA &18000.0
        & NA  & NA &NA &1 &18000.0
        &\textbf{799.60} \\
        r7-56 & \textbf{612.17}  &\textbf{611.97}  & 1496.3  &\textbf{612.02$^*$} &\textbf{612.02} &3991.7
        &\textbf{612.02$^*$} &\textbf{612.02} &0 &3 &803.0
        &613.66 \\
        r7-70 &\textbf{759.27}   &\textbf{753.56}  &10380.5  &\textbf{754.28$^*$} &\textbf{754.28} &16540.0
        &\textbf{754.28$^*$} &\textbf{754.28} &0 &11 &13664.6
        &766.05 \\
        r7-84 &1081.76  &NA  &18000.0  &1081.76  &NA  &18000.0
        &1081.76  &NA &NA &1 &18000.0 
        &\textbf{932.12} \\
        r8-64 &\textbf{632.22$^*$}   &\textbf{632.22}  &3030.0 &\textbf{632.22$^*$}   &\textbf{632.22}  &3030.0
        &\textbf{632.22$^*$}   &\textbf{632.22} &0 &1 &3030.0
        &638.36 \\
        r8-80 &\textbf{788.99} &\textbf{\textit{780.81}}  &18000.0 &\textbf{788.99} &\textbf{\textit{780.81}}  &18000.0
        &\textbf{788.99} &\textbf{\textit{780.81}} &NA &1 &18000.0
        &811.19 \\
        \hline
        Avg & &  &9974.1 & & &12750.8 
        & & & &0 &10277.3 &\\
        \hline
        Summary & $\#opt$ &$\#best ub$ & $\#best lb$ 
        & $\#opt$ &$\#best ub$ & $\#best lb$ 
        & $\#opt$ &$\#best lb$ & $\overline{LB}\%$ &$\overline{N_{node}}$ &$\#best ub$ 
        & $\#best ub$  \\
        &5 &14 &17 &10 &15 &17 
        &13 &17 &0 &3.4 &16
        &4 \\
        \hline
    \end{tabular}
   \end{threeparttable}
\end{table}

\subsubsection{B\&P results on type-a, -u, and -r instances under different minimum battery level restrictions.}
\label{sec:: BP experimental results}
From previous experiments, adding cuts is shown to be efficient in enhancing the quality of lower bounds and can close the integrality gaps for some instances. In this part, we explore the performance of the B\&P algorithm by applying it to solve instances where the CG algorithm obtains fractional solutions. The obtained B\&P results are compared to state-of-the-art results in \cite{bongiovanni2019electric} and \cite{su2023deterministic}. Then, we compare the obtained B\&P results with those of the CG algorithm with adding cuts. This comparison aims to analyze the trade-off between the time saved from adding cuts to enhance the lower bound, which may possibly close the integrality gaps without branching, and the time lost due to the increasing complexity of solving the current node.

Compared with the best-reported B\&C results in \cite{bongiovanni2019electric}, we finally solve 71 out of 84 instances optimally within the two-hour time limit, while the B\&C algorithm can only solve 49 instances optimally. In terms of computational efficiency, we have observed a significant reduction of 24\% with our proposed B\&P algorithm when solving both type-a and type-u instances, as compared to the reported CPU times of the B\&C algorithm. 
Then, we compared our B\&P algorithm results with our results of the CG algorithm with cuts. On type-a and -u instances, the proposed B\&P algorithm solves 5 additional instances optimally, obtains 6 tighter lower bounds, and generates 3 additional new best solutions. In addition, benefiting from the good quality of lower bounds obtained at the root node, only a few nodes of the search tree suffice to close the gap when applying the B\&P. 
As a result, branching seems to be more computationally attractive than adding valid cuts, as subset-row inequalities introduce additional complexity to solve the subproblems while branching will not. We have the same conclusion from the obtained B\&P results on type-r instances. Compared with the results of the CG algorithm with cutting planes, the proposed B\&P algorithm solves 3 more instances into optimality, further tightens 4 previously reported lower bounds, and improves 3 previously generated integer solutions.

\subsubsection{The impact of solving weighted-sum objective function, column pool initialization, and allowing unlimited charging visits.}
\label{interesting perspectives}
\paragraph{\textbf{The impact of considering total excess user ride time in the objective function}} To illustrate the tangible impact of incorporating the total excess user ride time within the objective function, we conduct experiments wherein we solely focused on minimizing the total travel time. For the obtained solutions, we report their corresponding total excess user ride time and compare them to the results we obtained from solving the weighted-sum objective problem. The experimental results can be found in \ref{effect of weighted sum objective}.

From Table \ref{BP results type-a with total travel time only} and Table \ref{BP results type-u with total travel time only}, comparing to the solutions achieved when solely minimizing the total travel time, our analysis reveals the following insights: (1) Across all cases, we have observed substantial reductions in total excess user ride time upon integrating it into the objective function. These improvements come at the cost of only slight increases in total travel time. (2) The above observation holds significant managerial implications, particularly for ride-sharing service providers like Uber. By considering a weighted-sum objective function, they can significantly enhance the quality of their service, ensuring a more favorable user experience while still maintaining operational costs at nearly optimal levels. An interesting avenue for further exploration involves addressing a bi-objective version of the E-ADARP. By identifying all Pareto optimal solutions, we can better understand the trade-off between total travel time and total excess user ride time.

\paragraph{\textbf{The impact of performing the DA algorithm to initialize the column pool}} 
In the previous section, we initialize the column pool by iterating the DA algorithm 500 times and storing all generated columns. To analyze the effect of this initialization strategy on the computational performance, we perform our algorithms with only 1 iteration of the DA algorithm to initialize the column pool. The computational results are presented in \ref{experiments 1 iter}.  

From Table \ref{BP results type-a with 1 iteration} and Table \ref{BP results type-u with 1 iteration}, we have the following findings: (1) Even when initiated with a randomly constructed solution, the B\&P algorithm yields outcomes of similar quality to those obtained using a higher-quality initial column pool. 
(2) The incorporation of the DA algorithm into the B\&P framework seems to help enhance the computational efficiency. By seeding the column pool with high-quality columns, the B\&P algorithm accelerates its convergence towards optimality, resulting in faster solution times. There is a trade-off between the time invested in the initialization phase and the time saved as a result of this high-quality initialization, which accelerates convergence. Based on preliminary experiments, iterating the DA algorithm 500 times is a good choice for initialization. 

\paragraph{\textbf{The impact of allowing unlimited visits to recharging stations}}
In \cite{bongiovanni2019electric}, each recharging station $s \in S$ can be visited at most once. The authors allow multiple visits to each recharging station by replicating the set of recharging stations $S$. Therefore, the maximum number of visits per recharging station (denoted as $N_{max}^s, s \in S$) must be predefined. Only three different values of $N_{max}^s$, $N_{max}^s = \{1,2,3\}$ are tested for type-u instances. From their results, the computational time increases substantially with a higher value of $N_{max}^s$, which prevents them from applying the B\&C algorithm to allow unlimited visits to each recharging station ($N_{max}^s = \infty$). 

In this part, we investigate the case of $N_{max}^s = \infty$. To allow unlimited visits per recharging station, the constraints (\ref{1.3.2}) in MP are removed. In the labeling algorithm, constraints (\ref{visits to rc}) are deleted to allow visits to the same recharging station. The B\&P algorithm is executed on type-a, -u, and -r instances. If the adapted B\&P algorithm solves instances optimally within the maximum time limit, the corresponding results are marked with an asterisk in the column named ``$Obj$". 

In \ref{experiments multiple visits}, Table \ref{multiple type-a}, Table \ref{multiple type-u}, and Table \ref{multiple type-r} show the results of allowing unlimited visits to each recharging station under different settings of $\gamma$ for type-a, -u, and -r instances, respectively. For $\gamma = 0.1$, as the optimal solutions are obtained without visiting recharging stations multiple times for type-a instances, we only conduct experiments for $\gamma = 0.4, 0.7$ for type-a instances.  For type-a and -u instances, our B\&P algorithm obtains optimal solutions for 62 out of 70 instances, among which 60 are solved optimally at the root node without branching. For type-r instances, the proposed B\&P algorithm solves 18 instances optimally (12 instances are solved optimally at the root node). Feasible solutions are yielded for all instances. Hence, we demonstrate in all cases the effectiveness of the B\&P algorithm to solve a more general version of the E-ADARP.

To analyze the effect of allowing unlimited visits, we compare our obtained results with the best-reported results of allowing at most one recharging visit among the B\&P and B\&C algorithms (presented in the column named ``BKS$'$"). The column of ``BKS$'$\%" represents the deviation between BKS and the best-obtained results allowing unlimited charging visits per station. If a solution obtained with $N_{max}^s = \infty$ has a lower cost than with $N_{max}^s = 1$, the corresponding BKS$'$\% is positive. We also report the maximum number of recharging visits to a recharging station for the obtained solutions in the column named ``$N_{max}^s$". 
The major observations include: (1) we find 52 out of 100 obtained solutions contain multiple visits to the same recharging station, especially under $\gamma = 0.7$, and we observe considerable improvements in solution feasibility and quality with setting $N_{max}^s = \infty$; (2) there are several solutions for type-a instances that contain multiple recharging visits but have the same cost as the obtained solutions with $N_{max}^s = 1$ (e.g., a3-24-0.4). This situation is possible as all recharging stations for a type-a instance are at the same location and one can have an equal-cost solution by replacing one recharging station with another; (3) setting $N_{max}^s = 3$ seems enough for solving type-u instances allowing unlimited recharging visits, while one needs to set $N_{max}^s$ to 4 and 8 for type-a and -r instances, respectively.

\section{Conclusion and Extensions} \label{conclusion}
This paper has presented an effective B\&P algorithm relying on a problem-tailored labeling algorithm to solve the E-ADARP. The E-ADARP shares some characteristics with the classical DARP and E-VRP but is mainly differentiated from these problems regarding the following features: a weighted-sum objective function that minimizes total travel time and excess user ride time, EAVs, and partial recharging at recharging stations. 
In this work, we propose a new paradigm to handle the excess-user-ride-time optimization in the labeling algorithm. This paradigm includes (1) a fragment-based representation of paths that replaces a sequence of REFs with a single one, (2) a novel approach that abstracts fragments to arcs with excess user ride time being minimized to build a new sparser graph that preserves all feasible routes of the original graph, and (3) strong dominance rules and constant time feasibility checks on the new graph.

In the computational experiments, the performance of the B\&P algorithm is rigorously evaluated across an extensive range of benchmark instances from \cite{bongiovanni2019electric} and \cite{su2023deterministic}. We first compare the algorithm performance between CG with cutting planes and the B\&P algorithm. With our best-performing method (i.e., the B\&P algorithm), we solve 71 out of 84 instances optimally within the two-hour time limit, while the B\&C algorithm of \cite{bongiovanni2019electric} can only solve 49 of them optimally. In addition, we obtain 26 new best solutions, 54 equal solutions, and 30 improved lower bounds. On larger-sized instances (i.e., type-r instances), we obtain 16 new best solutions, compared with the existing results of \cite{su2023deterministic}. 
Then, we analyze the impact of (1) incorporating the total excess user ride time into the objective, (2) performing the DA algorithm for column pool initialization, and (3) allowing unlimited visits to recharging stations. From our numerical results, we have the following conclusions: (1) By integrating total excess user ride time into the objective, we obtain solutions that achieve dramatic reductions in total excess user ride time, accompanied by only a marginal increase in total travel time. (2) Computing initial columns by means of an existing metaheuristic algorithm contributes towards a faster convergence to optimality. 
(3) The relaxation on charging visits per recharging station allows to obtain feasible solutions for instances that could not be solved in the more restrictive settings and for many of the other instances. Forty-nine improved solutions are obtained, compared to the best-obtained results of allowing at most one recharging visit per station. 



Finally, our CG and B\&P algorithms also offer new insights into designing an exact algorithm for solving a practical version of the electric DARP, considering multiple objectives. One important ``by-product" of our labeling algorithm is the first exact scheduling procedure that can efficiently determine the excess-user-ride-time optimal schedule for a given E-ADARP route. This scheduling procedure can also be applied to optimize excess user ride time in the classical DARP or the DARP with multiple objectives, in which total excess user ride time is minimized in a separate objective. The proposed efficient approaches can be adapted to tackle the dynamic DARP/E-ADARP, where new requests arrive in real-time. 
Another promising direction stemming from our study is to take a more comprehensive analysis of the non-linear nature of recharging and discharging, thereby offering theoretical insights on modeling the E-ADARP in a more energy-realistic context.


\bibliography{mybib.bib} 

\begin{thebibliography}{45}
\expandafter\ifx\csname natexlab\endcsname\relax\def\natexlab#1{#1}\fi
\providecommand{\url}[1]{\texttt{#1}}
\providecommand{\href}[2]{#2}
\providecommand{\path}[1]{#1}
\providecommand{\DOIprefix}{doi:}
\providecommand{\ArXivprefix}{arXiv:}
\providecommand{\URLprefix}{URL: }
\providecommand{\Pubmedprefix}{pmid:}
\providecommand{\doi}[1]{\href{http://dx.doi.org/#1}{\path{#1}}}
\providecommand{\Pubmed}[1]{\href{pmid:#1}{\path{#1}}}
\providecommand{\bibinfo}[2]{#2}
\ifx\xfnm\relax \def\xfnm[#1]{\unskip,\space#1}\fi
\bibitem[{Alyasiry et~al.(2019)Alyasiry, Forbes \& Bulmer}]{alyasiry2019exact}
\bibinfo{author}{Alyasiry, A.~M.}, \bibinfo{author}{Forbes, M.}, \&
  \bibinfo{author}{Bulmer, M.} (\bibinfo{year}{2019}).
\newblock \bibinfo{title}{An exact algorithm for the pickup and delivery
  problem with time windows and last-in-first-out loading}.
\newblock {\it \bibinfo{journal}{Transportation Science}\/},  {\it
  \bibinfo{volume}{53}\/}, \bibinfo{pages}{1695--1705}.
\bibitem[{Bongiovanni et~al.(2022{\natexlab{a}})Bongiovanni, Geroliminis \&
  Kaspi}]{bongiovanni2022ride}
\bibinfo{author}{Bongiovanni, C.}, \bibinfo{author}{Geroliminis, N.}, \&
  \bibinfo{author}{Kaspi, M.} (\bibinfo{year}{2022}{\natexlab{a}}).
\newblock \bibinfo{title}{A ride time-oriented scheduling algorithm for
  dial-a-ride problems}.
\newblock {\it \bibinfo{journal}{arXiv preprint arXiv:2211.07347}\/}, .
\bibitem[{Bongiovanni et~al.(2022{\natexlab{b}})Bongiovanni, Kaspi, Cordeau \&
  Geroliminis}]{bongiovanni2022machine}
\bibinfo{author}{Bongiovanni, C.}, \bibinfo{author}{Kaspi, M.},
  \bibinfo{author}{Cordeau, J.-F.}, \& \bibinfo{author}{Geroliminis, N.}
  (\bibinfo{year}{2022}{\natexlab{b}}).
\newblock \bibinfo{title}{A machine learning-driven two-phase metaheuristic for
  autonomous ridesharing operations}.
\newblock {\it \bibinfo{journal}{Transportation Research Part E: Logistics and
  Transportation Review}\/},  {\it \bibinfo{volume}{165}\/},
  \bibinfo{pages}{102835}.
\bibitem[{Bongiovanni et~al.(2019)Bongiovanni, Kaspi \&
  Geroliminis}]{bongiovanni2019electric}
\bibinfo{author}{Bongiovanni, C.}, \bibinfo{author}{Kaspi, M.}, \&
  \bibinfo{author}{Geroliminis, N.} (\bibinfo{year}{2019}).
\newblock \bibinfo{title}{The electric autonomous dial-a-ride problem}.
\newblock {\it \bibinfo{journal}{Transportation Research Part B:
  Methodological}\/},  {\it \bibinfo{volume}{122}\/},
  \bibinfo{pages}{436--456}.
\bibitem[{Braekers et~al.(2014)Braekers, Caris \& Janssens}]{braekers2014exact}
\bibinfo{author}{Braekers, K.}, \bibinfo{author}{Caris, A.}, \&
  \bibinfo{author}{Janssens, G.~K.} (\bibinfo{year}{2014}).
\newblock \bibinfo{title}{Exact and meta-heuristic approach for a general
  heterogeneous dial-a-ride problem with multiple depots}.
\newblock {\it \bibinfo{journal}{Transportation Research Part B:
  Methodological}\/},  {\it \bibinfo{volume}{67}\/}, \bibinfo{pages}{166--186}.
\bibitem[{Bruglieri et~al.(2015)Bruglieri, Pezzella, Pisacane \&
  Suraci}]{bruglieri2015variable}
\bibinfo{author}{Bruglieri, M.}, \bibinfo{author}{Pezzella, F.},
  \bibinfo{author}{Pisacane, O.}, \& \bibinfo{author}{Suraci, S.}
  (\bibinfo{year}{2015}).
\newblock \bibinfo{title}{A variable neighborhood search branching for the
  electric vehicle routing problem with time windows}.
\newblock {\it \bibinfo{journal}{Electronic Notes in Discrete Mathematics}\/},
  {\it \bibinfo{volume}{47}\/}, \bibinfo{pages}{221--228}.
\bibitem[{Ceselli et~al.(2021)Ceselli, Felipe, Ortu{\~n}o, Righini \&
  Tirado}]{ceselli2021branch}
\bibinfo{author}{Ceselli, A.}, \bibinfo{author}{Felipe, {\'A}.},
  \bibinfo{author}{Ortu{\~n}o, M.~T.}, \bibinfo{author}{Righini, G.}, \&
  \bibinfo{author}{Tirado, G.} (\bibinfo{year}{2021}).
\newblock \bibinfo{title}{A branch-and-cut-and-price algorithm for the electric
  vehicle routing problem with multiple technologies}.
\newblock In {\it \bibinfo{booktitle}{Operations Research Forum}\/} (pp.
  \bibinfo{pages}{1--33}).
\newblock \bibinfo{organization}{Springer} volume~\bibinfo{volume}{2}.
\bibitem[{Conrad \& Figliozzi(2011)}]{conrad2011recharging}
\bibinfo{author}{Conrad, R.~G.}, \& \bibinfo{author}{Figliozzi, M.~A.}
  (\bibinfo{year}{2011}).
\newblock \bibinfo{title}{The recharging vehicle routing problem}.
\newblock In {\it \bibinfo{booktitle}{Proceedings of the 2011 industrial
  engineering research conference}\/} (p.~\bibinfo{pages}{8}).
\newblock \bibinfo{organization}{IISE Norcross, GA}.
\bibitem[{Cordeau(2006)}]{cordeau2006branch}
\bibinfo{author}{Cordeau, J.-F.} (\bibinfo{year}{2006}).
\newblock \bibinfo{title}{A branch-and-cut algorithm for the dial-a-ride
  problem}.
\newblock {\it \bibinfo{journal}{Operations Research}\/},  {\it
  \bibinfo{volume}{54}\/}, \bibinfo{pages}{573--586}.
\bibitem[{Cordeau \& Laporte(2003)}]{cordeau2003tabu}
\bibinfo{author}{Cordeau, J.-F.}, \& \bibinfo{author}{Laporte, G.}
  (\bibinfo{year}{2003}).
\newblock \bibinfo{title}{A tabu search heuristic for the static multi-vehicle
  dial-a-ride problem}.
\newblock {\it \bibinfo{journal}{Transportation Research Part B:
  Methodological}\/},  {\it \bibinfo{volume}{37}\/}, \bibinfo{pages}{579--594}.
\bibitem[{Cordeau \& Laporte(2007)}]{cordeau2007dial}
\bibinfo{author}{Cordeau, J.-F.}, \& \bibinfo{author}{Laporte, G.}
  (\bibinfo{year}{2007}).
\newblock \bibinfo{title}{The dial-a-ride problem: models and algorithms}.
\newblock {\it \bibinfo{journal}{Annals of operations research}\/},  {\it
  \bibinfo{volume}{153}\/}, \bibinfo{pages}{29--46}.
\bibitem[{Desaulniers et~al.(2011)Desaulniers, Desrosiers \&
  Spoorendonk}]{desaulniers2011cutting}
\bibinfo{author}{Desaulniers, G.}, \bibinfo{author}{Desrosiers, J.}, \&
  \bibinfo{author}{Spoorendonk, S.} (\bibinfo{year}{2011}).
\newblock \bibinfo{title}{Cutting planes for branch-and-price algorithms}.
\newblock {\it \bibinfo{journal}{Networks}\/},  {\it \bibinfo{volume}{58}\/},
  \bibinfo{pages}{301--310}.
\bibitem[{Desaulniers et~al.(2016)Desaulniers, Errico, Irnich \&
  Schneider}]{desaulniers2016exact}
\bibinfo{author}{Desaulniers, G.}, \bibinfo{author}{Errico, F.},
  \bibinfo{author}{Irnich, S.}, \& \bibinfo{author}{Schneider, M.}
  (\bibinfo{year}{2016}).
\newblock \bibinfo{title}{Exact algorithms for electric vehicle-routing
  problems with time windows}.
\newblock {\it \bibinfo{journal}{Operations Research}\/},  {\it
  \bibinfo{volume}{64}\/}, \bibinfo{pages}{1388--1405}.
\bibitem[{Desaulniers et~al.(2020)Desaulniers, Gschwind \&
  Irnich}]{desaulniers2020variable}
\bibinfo{author}{Desaulniers, G.}, \bibinfo{author}{Gschwind, T.}, \&
  \bibinfo{author}{Irnich, S.} (\bibinfo{year}{2020}).
\newblock \bibinfo{title}{Variable fixing for two-arc sequences in
  branch-price-and-cut algorithms on path-based models}.
\newblock {\it \bibinfo{journal}{Transportation Science}\/},  {\it
  \bibinfo{volume}{54}\/}, \bibinfo{pages}{1170--1188}.
\bibitem[{Desaulniers et~al.(2008)Desaulniers, Lessard \&
  Hadjar}]{desaulniers2008tabu}
\bibinfo{author}{Desaulniers, G.}, \bibinfo{author}{Lessard, F.}, \&
  \bibinfo{author}{Hadjar, A.} (\bibinfo{year}{2008}).
\newblock \bibinfo{title}{Tabu search, partial elementarity, and generalized
  k-path inequalities for the vehicle routing problem with time windows}.
\newblock {\it \bibinfo{journal}{Transportation Science}\/},  {\it
  \bibinfo{volume}{42}\/}, \bibinfo{pages}{387--404}.
\bibitem[{Desrochers et~al.(1992)Desrochers, Desrosiers \&
  Solomon}]{desrochers1992new}
\bibinfo{author}{Desrochers, M.}, \bibinfo{author}{Desrosiers, J.}, \&
  \bibinfo{author}{Solomon, M.} (\bibinfo{year}{1992}).
\newblock \bibinfo{title}{A new optimization algorithm for the vehicle routing
  problem with time windows}.
\newblock {\it \bibinfo{journal}{Operations research}\/},  {\it
  \bibinfo{volume}{40}\/}, \bibinfo{pages}{342--354}.
\bibitem[{Duman et~al.(2021)Duman, Ta{\c{s}} \& {\c{C}}atay}]{duman2021branch}
\bibinfo{author}{Duman, E.~N.}, \bibinfo{author}{Ta{\c{s}}, D.}, \&
  \bibinfo{author}{{\c{C}}atay, B.} (\bibinfo{year}{2021}).
\newblock \bibinfo{title}{Branch-and-price-and-cut methods for the electric
  vehicle routing problem with time windows}.
\newblock {\it \bibinfo{journal}{International Journal of Production
  Research}\/},  (pp. \bibinfo{pages}{1--22}).
\bibitem[{Felipe et~al.(2014)Felipe, Ortu{\~n}o, Righini \&
  Tirado}]{felipe2014heuristic}
\bibinfo{author}{Felipe, {\'A}.}, \bibinfo{author}{Ortu{\~n}o, M.~T.},
  \bibinfo{author}{Righini, G.}, \& \bibinfo{author}{Tirado, G.}
  (\bibinfo{year}{2014}).
\newblock \bibinfo{title}{A heuristic approach for the green vehicle routing
  problem with multiple technologies and partial recharges}.
\newblock {\it \bibinfo{journal}{Transportation Research Part E: Logistics and
  Transportation Review}\/},  {\it \bibinfo{volume}{71}\/},
  \bibinfo{pages}{111--128}.
\bibitem[{Froger et~al.(2017)Froger, Mendoza, Jabali \&
  Laporte}]{froger2017matheuristic}
\bibinfo{author}{Froger, A.}, \bibinfo{author}{Mendoza, J.~E.},
  \bibinfo{author}{Jabali, O.}, \& \bibinfo{author}{Laporte, G.}
  (\bibinfo{year}{2017}).
\newblock {\it \bibinfo{title}{A matheuristic for the electric vehicle routing
  problem with capacitated charging stations}\/}.
\newblock Ph.D. thesis Centre interuniversitaire de recherche sur les reseaux
  d'entreprise, la~….
\bibitem[{Genikomsakis \& Mitrentsis(2017)}]{genikomsakis2017computationally}
\bibinfo{author}{Genikomsakis, K.~N.}, \& \bibinfo{author}{Mitrentsis, G.}
  (\bibinfo{year}{2017}).
\newblock \bibinfo{title}{A computationally efficient simulation model for
  estimating energy consumption of electric vehicles in the context of route
  planning applications}.
\newblock {\it \bibinfo{journal}{Transportation Research Part D: Transport and
  Environment}\/},  {\it \bibinfo{volume}{50}\/}, \bibinfo{pages}{98--118}.
\bibitem[{Gschwind \& Drexl(2019)}]{gschwind2019adaptive}
\bibinfo{author}{Gschwind, T.}, \& \bibinfo{author}{Drexl, M.}
  (\bibinfo{year}{2019}).
\newblock \bibinfo{title}{Adaptive large neighborhood search with a
  constant-time feasibility test for the dial-a-ride problem}.
\newblock {\it \bibinfo{journal}{Transportation Science}\/},  {\it
  \bibinfo{volume}{53}\/}, \bibinfo{pages}{480--491}.
\bibitem[{Gschwind \& Irnich(2015)}]{gschwind2015effective}
\bibinfo{author}{Gschwind, T.}, \& \bibinfo{author}{Irnich, S.}
  (\bibinfo{year}{2015}).
\newblock \bibinfo{title}{Effective handling of dynamic time windows and its
  application to solving the dial-a-ride problem}.
\newblock {\it \bibinfo{journal}{Transportation Science}\/},  {\it
  \bibinfo{volume}{49}\/}, \bibinfo{pages}{335--354}.
\bibitem[{Hiermann et~al.(2019)Hiermann, Hartl, Puchinger \&
  Vidal}]{hiermann2019routing}
\bibinfo{author}{Hiermann, G.}, \bibinfo{author}{Hartl, R.~F.},
  \bibinfo{author}{Puchinger, J.}, \& \bibinfo{author}{Vidal, T.}
  (\bibinfo{year}{2019}).
\newblock \bibinfo{title}{Routing a mix of conventional, plug-in hybrid, and
  electric vehicles}.
\newblock {\it \bibinfo{journal}{European Journal of Operational Research}\/},
  {\it \bibinfo{volume}{272}\/}, \bibinfo{pages}{235--248}.
\bibitem[{Ho et~al.(2018)Ho, Szeto, Kuo, Leung, Petering \& Tou}]{ho2018survey}
\bibinfo{author}{Ho, S.~C.}, \bibinfo{author}{Szeto, W.~Y.},
  \bibinfo{author}{Kuo, Y.-H.}, \bibinfo{author}{Leung, J.~M.},
  \bibinfo{author}{Petering, M.}, \& \bibinfo{author}{Tou, T.~W.}
  (\bibinfo{year}{2018}).
\newblock \bibinfo{title}{A survey of dial-a-ride problems: Literature review
  and recent developments}.
\newblock {\it \bibinfo{journal}{Transportation Research Part B:
  Methodological}\/},  {\it \bibinfo{volume}{111}\/},
  \bibinfo{pages}{395--421}.
\bibitem[{Jepsen et~al.(2008)Jepsen, Petersen, Spoorendonk \&
  Pisinger}]{jepsen2008subset}
\bibinfo{author}{Jepsen, M.}, \bibinfo{author}{Petersen, B.},
  \bibinfo{author}{Spoorendonk, S.}, \& \bibinfo{author}{Pisinger, D.}
  (\bibinfo{year}{2008}).
\newblock \bibinfo{title}{Subset-row inequalities applied to the
  vehicle-routing problem with time windows}.
\newblock {\it \bibinfo{journal}{Operations Research}\/},  {\it
  \bibinfo{volume}{56}\/}, \bibinfo{pages}{497--511}.
\bibitem[{Jin et~al.(2018)Jin, Kong, Wu \& Sui}]{jin2018ridesourcing}
\bibinfo{author}{Jin, S.~T.}, \bibinfo{author}{Kong, H.}, \bibinfo{author}{Wu,
  R.}, \& \bibinfo{author}{Sui, D.~Z.} (\bibinfo{year}{2018}).
\newblock \bibinfo{title}{Ridesourcing, the sharing economy, and the future of
  cities}.
\newblock {\it \bibinfo{journal}{Cities}\/},  {\it \bibinfo{volume}{76}\/},
  \bibinfo{pages}{96--104}.
\bibitem[{Keskin \& {\c{C}}atay(2016)}]{keskin2016partial}
\bibinfo{author}{Keskin, M.}, \& \bibinfo{author}{{\c{C}}atay, B.}
  (\bibinfo{year}{2016}).
\newblock \bibinfo{title}{Partial recharge strategies for the electric vehicle
  routing problem with time windows}.
\newblock {\it \bibinfo{journal}{Transportation research part C: emerging
  technologies}\/},  {\it \bibinfo{volume}{65}\/}, \bibinfo{pages}{111--127}.
\bibitem[{Kohl et~al.(1999)Kohl, Desrosiers, Madsen, Solomon \&
  Soumis}]{kohl19992}
\bibinfo{author}{Kohl, N.}, \bibinfo{author}{Desrosiers, J.},
  \bibinfo{author}{Madsen, O.~B.}, \bibinfo{author}{Solomon, M.~M.}, \&
  \bibinfo{author}{Soumis, F.} (\bibinfo{year}{1999}).
\newblock \bibinfo{title}{2-path cuts for the vehicle routing problem with time
  windows}.
\newblock {\it \bibinfo{journal}{Transportation Science}\/},  {\it
  \bibinfo{volume}{33}\/}, \bibinfo{pages}{101--116}.
\bibitem[{Kullman et~al.(2022)Kullman, Cousineau, Goodson \&
  Mendoza}]{kullman2022dynamic}
\bibinfo{author}{Kullman, N.~D.}, \bibinfo{author}{Cousineau, M.},
  \bibinfo{author}{Goodson, J.~C.}, \& \bibinfo{author}{Mendoza, J.~E.}
  (\bibinfo{year}{2022}).
\newblock \bibinfo{title}{Dynamic ride-hailing with electric vehicles}.
\newblock {\it \bibinfo{journal}{Transportation Science}\/},  {\it
  \bibinfo{volume}{56}\/}, \bibinfo{pages}{775--794}.
\bibitem[{Lam et~al.(2022)Lam, Desaulniers \& Stuckey}]{lam2022branch}
\bibinfo{author}{Lam, E.}, \bibinfo{author}{Desaulniers, G.}, \&
  \bibinfo{author}{Stuckey, P.~J.} (\bibinfo{year}{2022}).
\newblock \bibinfo{title}{Branch-and-cut-and-price for the electric vehicle
  routing problem with time windows, piecewise-linear recharging and
  capacitated recharging stations}.
\newblock {\it \bibinfo{journal}{Computers \& Operations Research}\/},  (p.
  \bibinfo{pages}{105870}).
\bibitem[{Madsen et~al.(1995)Madsen, Ravn \& Rygaard}]{madsen1995heuristic}
\bibinfo{author}{Madsen, O.~B.}, \bibinfo{author}{Ravn, H.~F.}, \&
  \bibinfo{author}{Rygaard, J.~M.} (\bibinfo{year}{1995}).
\newblock \bibinfo{title}{A heuristic algorithm for a dial-a-ride problem with
  time windows, multiple capacities, and multiple objectives}.
\newblock {\it \bibinfo{journal}{Annals of operations Research}\/},  {\it
  \bibinfo{volume}{60}\/}, \bibinfo{pages}{193--208}.
\bibitem[{Masmoudi et~al.(2018)Masmoudi, Hosny, Demir, Genikomsakis \&
  Cheikhrouhou}]{masmoudi2018dial}
\bibinfo{author}{Masmoudi, M.~A.}, \bibinfo{author}{Hosny, M.},
  \bibinfo{author}{Demir, E.}, \bibinfo{author}{Genikomsakis, K.~N.}, \&
  \bibinfo{author}{Cheikhrouhou, N.} (\bibinfo{year}{2018}).
\newblock \bibinfo{title}{The dial-a-ride problem with electric vehicles and
  battery swapping stations}.
\newblock {\it \bibinfo{journal}{Transportation research part E: logistics and
  transportation review}\/},  {\it \bibinfo{volume}{118}\/},
  \bibinfo{pages}{392--420}.
\bibitem[{Molenbruch et~al.(2017)Molenbruch, Braekers, Caris \&
  Berghe}]{molenbruch2017multi}
\bibinfo{author}{Molenbruch, Y.}, \bibinfo{author}{Braekers, K.},
  \bibinfo{author}{Caris, A.}, \& \bibinfo{author}{Berghe, G.~V.}
  (\bibinfo{year}{2017}).
\newblock \bibinfo{title}{Multi-directional local search for a bi-objective
  dial-a-ride problem in patient transportation}.
\newblock {\it \bibinfo{journal}{Computers \& Operations Research}\/},  {\it
  \bibinfo{volume}{77}\/}, \bibinfo{pages}{58--71}.
\bibitem[{Montoya et~al.(2017)Montoya, Gu{\'e}ret, Mendoza \&
  Villegas}]{montoya2017electric}
\bibinfo{author}{Montoya, A.}, \bibinfo{author}{Gu{\'e}ret, C.},
  \bibinfo{author}{Mendoza, J.~E.}, \& \bibinfo{author}{Villegas, J.~G.}
  (\bibinfo{year}{2017}).
\newblock \bibinfo{title}{The electric vehicle routing problem with nonlinear
  charging function}.
\newblock {\it \bibinfo{journal}{Transportation Research Part B:
  Methodological}\/},  {\it \bibinfo{volume}{103}\/}, \bibinfo{pages}{87--110}.
\bibitem[{Parragh(2011)}]{parragh2011introducing}
\bibinfo{author}{Parragh, S.~N.} (\bibinfo{year}{2011}).
\newblock \bibinfo{title}{Introducing heterogeneous users and vehicles into
  models and algorithms for the dial-a-ride problem}.
\newblock {\it \bibinfo{journal}{Transportation Research Part C: Emerging
  Technologies}\/},  {\it \bibinfo{volume}{19}\/}, \bibinfo{pages}{912--930}.
\bibitem[{Parragh et~al.(2012)Parragh, Cordeau, Doerner \&
  Hartl}]{parragh2012models}
\bibinfo{author}{Parragh, S.~N.}, \bibinfo{author}{Cordeau, J.-F.},
  \bibinfo{author}{Doerner, K.~F.}, \& \bibinfo{author}{Hartl, R.~F.}
  (\bibinfo{year}{2012}).
\newblock \bibinfo{title}{Models and algorithms for the heterogeneous
  dial-a-ride problem with driver-related constraints}.
\newblock {\it \bibinfo{journal}{OR spectrum}\/},  {\it
  \bibinfo{volume}{34}\/}, \bibinfo{pages}{593--633}.
\bibitem[{Parragh et~al.(2009)Parragh, Doerner, Hartl \&
  Gandibleux}]{parragh2009heuristic}
\bibinfo{author}{Parragh, S.~N.}, \bibinfo{author}{Doerner, K.~F.},
  \bibinfo{author}{Hartl, R.~F.}, \& \bibinfo{author}{Gandibleux, X.}
  (\bibinfo{year}{2009}).
\newblock \bibinfo{title}{A heuristic two-phase solution approach for the
  multi-objective dial-a-ride problem}.
\newblock {\it \bibinfo{journal}{Networks: An International Journal}\/},  {\it
  \bibinfo{volume}{54}\/}, \bibinfo{pages}{227--242}.
\bibitem[{Pelletier et~al.(2017)Pelletier, Jabali, Laporte \&
  Veneroni}]{pelletier2017battery}
\bibinfo{author}{Pelletier, S.}, \bibinfo{author}{Jabali, O.},
  \bibinfo{author}{Laporte, G.}, \& \bibinfo{author}{Veneroni, M.}
  (\bibinfo{year}{2017}).
\newblock \bibinfo{title}{Battery degradation and behaviour for electric
  vehicles: Review and numerical analyses of several models}.
\newblock {\it \bibinfo{journal}{Transportation Research Part B:
  Methodological}\/},  {\it \bibinfo{volume}{103}\/},
  \bibinfo{pages}{158--187}.
\bibitem[{Rist \& Forbes(2021)}]{rist2021new}
\bibinfo{author}{Rist, Y.}, \& \bibinfo{author}{Forbes, M.~A.}
  (\bibinfo{year}{2021}).
\newblock \bibinfo{title}{A new formulation for the dial-a-ride problem}.
\newblock {\it \bibinfo{journal}{Transportation Science}\/},  {\it
  \bibinfo{volume}{55}\/}, \bibinfo{pages}{1113--1135}.
\bibitem[{Ropke et~al.(2007)Ropke, Cordeau \& Laporte}]{ropke2007models}
\bibinfo{author}{Ropke, S.}, \bibinfo{author}{Cordeau, J.-F.}, \&
  \bibinfo{author}{Laporte, G.} (\bibinfo{year}{2007}).
\newblock \bibinfo{title}{Models and branch-and-cut algorithms for pickup and
  delivery problems with time windows}.
\newblock {\it \bibinfo{journal}{Networks: An International Journal}\/},  {\it
  \bibinfo{volume}{49}\/}, \bibinfo{pages}{258--272}.
\bibitem[{Ropke \& Pisinger(2006)}]{ropke2006adaptive}
\bibinfo{author}{Ropke, S.}, \& \bibinfo{author}{Pisinger, D.}
  (\bibinfo{year}{2006}).
\newblock \bibinfo{title}{An adaptive large neighborhood search heuristic for
  the pickup and delivery problem with time windows}.
\newblock {\it \bibinfo{journal}{Transportation science}\/},  {\it
  \bibinfo{volume}{40}\/}, \bibinfo{pages}{455--472}.
\bibitem[{Savelsbergh(1992)}]{savelsbergh1992vehicle}
\bibinfo{author}{Savelsbergh, M.~W.} (\bibinfo{year}{1992}).
\newblock \bibinfo{title}{The vehicle routing problem with time windows:
  Minimizing route duration}.
\newblock {\it \bibinfo{journal}{ORSA journal on computing}\/},  {\it
  \bibinfo{volume}{4}\/}, \bibinfo{pages}{146--154}.
\bibitem[{Schneider et~al.(2014)Schneider, Stenger \&
  Goeke}]{schneider2014electric}
\bibinfo{author}{Schneider, M.}, \bibinfo{author}{Stenger, A.}, \&
  \bibinfo{author}{Goeke, D.} (\bibinfo{year}{2014}).
\newblock \bibinfo{title}{The electric vehicle-routing problem with time
  windows and recharging stations}.
\newblock {\it \bibinfo{journal}{Transportation Science}\/},  {\it
  \bibinfo{volume}{48}\/}, \bibinfo{pages}{500--520}.
\bibitem[{Su et~al.(2023)Su, Dupin \& Puchinger}]{su2023deterministic}
\bibinfo{author}{Su, Y.}, \bibinfo{author}{Dupin, N.}, \&
  \bibinfo{author}{Puchinger, J.} (\bibinfo{year}{2023}).
\newblock \bibinfo{title}{A deterministic annealing local search for the
  electric autonomous dial-a-ride problem}.
\newblock {\it \bibinfo{journal}{European Journal of Operational Research}\/},
  {\it \bibinfo{volume}{309}\/}, \bibinfo{pages}{1091--1111}.
\bibitem[{Toth \& Vigo(1996)}]{toth1996fast}
\bibinfo{author}{Toth, P.}, \& \bibinfo{author}{Vigo, D.}
  (\bibinfo{year}{1996}).
\newblock \bibinfo{title}{Fast local search algorithms for the handicapped
  persons transportation problem}.
\newblock In {\it \bibinfo{booktitle}{Meta-Heuristics}\/} (pp.
  \bibinfo{pages}{677--690}).
\newblock \bibinfo{publisher}{Springer}.

\end{thebibliography}
\newpage

\appendix
\section{Linear Programming Model} \label{case 2 solving}

We present the Linear Programming (LP) model for calculating the minimum excess user ride time in case of two or more requests are served simultaneously on a fragment $\mathcal{F}$ and $\exists i \in \mathcal{F}$ such that $\Delta_i \neq 0$. 

Let $P_\mathcal{F}$ denote all pickup nodes on $\mathcal{F} = \{1,\cdots,m\}$:

\begin{equation} \label{objective R}
     \min \sum\limits_{i \in P_\mathcal{F}}R_i
\end{equation}

s.t.


\begin{equation} \label{TW2}
\begin{cases}
    T_i = \mathcal{B}_1^l, \quad & \text{if $i = 1$}\\
    T_i = \mathcal{B}_m^l, \quad & \text{if $i = m$}\\
    e_i \leqslant T_i \leqslant l_i, \quad & \text{Otherwise}\\
\end{cases}
\end{equation}

\begin{equation} \label{TW3}
    T_i + s_i + t_{i,j} \leqslant T_j, \quad \forall i \in \mathcal{F} \setminus \{m\}, \quad idx_j = idx_i + 1
\end{equation}
\begin{equation} \label{user ride 1}
    T_{n+i} - (T_i + s_i) \leqslant m_i, \quad \forall i \in P_\mathcal{F}
\end{equation}
\begin{equation} \label{user ride 2}
     T_{n+i} - T_i - s_i - t_{i,n+i} \leqslant R_i, \quad \forall i \in P_\mathcal{F}
\end{equation}
\begin{equation}
      R_i \geqslant 0, \quad \forall i \in P_\mathcal{F}
\end{equation}


where $idx_i$ is the index of node $i$ on the segment. The objective function is to minimize the total excess user ride time. Constraints (\ref{TW2}) to (\ref{TW3}) are time window constraints where we set the service start time at node 1 and node $m$ to $\mathcal{B}_1^l$ and $\mathcal{B}_1^m$ ($\mathcal{B}^l$ is the latest vehicle-waiting-time optimal schedule), respectively. Constraints (\ref{user ride 1}) and constraints (\ref{user ride 2}) are user ride time constraints.

\section{Computational experiments for fragment enumeration}
\label{preliminary for frag enumeration}
In this section, we present experimental details of fragment enumeration for each instance in Table \ref{details seg}, as in \cite{su2023deterministic}. For the sake of completeness, we also present detailed results in this section. 
These details include the number of fragments generated (in the column ``$N_{frag}$"), the average and maximum length of fragments (in columns ``$Leg_{avg}$" and ``$Leg_{max}$"), respectively, the number of time LP is solved (in the column ``$N_{LP}$"), and the total computational time for enumeration in seconds (in the column named ``CPU").

\begin{table}[!htp]
\renewcommand\arraystretch{0.95}
 \centering
 \footnotesize
 \begin{threeparttable}
 \caption{Details of fragments enumeration for all the instances}
 \label{details seg}
   \begin{tabular}{c c c c c c}
        \toprule
        & $N_{frag}$ & $Leg_{avg}$ & $Leg_{max}$ & $N_{LP}$ & CPU(s) \\
        \midrule
        a2-16 &32  & 3.06 & 6  & 0 & 0.94\\
        a2-20 &51  & 3.41 & 6  & 1  & 0.23\\
        a2-24 &64  & 3.72 & 8  & 1  & 0.09\\
        a3-18 &71  & 4.25 & 8  & 4  & 0.04\\
        a3-24 &110 & 4.71 & 12 & 0  & 0.06\\
        a3-30 &89  & 3.66 & 8  & 0  & 0.12\\
        a3-36 &114 & 4.12 & 12 & 1  & 0.27\\
        a4-16 &78  & 4.51 & 8  & 4  & 0.04\\
        a4-24 &91  & 4.07 & 8  & 2 & 0.07\\
        a4-32 &206 & 5.58 & 12 & 3  & 0.20\\
        a4-40 &242 & 5.45 & 12 & 6 & 0.37\\
        a4-48 &355 & 5.33 & 12 & 15 & 0.61\\
        a5-40 &337 & 5.65 & 12 & 3 & 0.38\\
        a5-50 &659 & 8.25 & 24 & 33& 0.99\\
        \hline
        Avg   &178.5 &4.70 &10.57 &5.21 &0.32 \\
        \hline
        u2-16 &61   & 3.80 & 6  & 0  & 1.05\\
        u2-20 &180  & 5.26 & 12 & 7  & 0.32\\
        u2-24 &66   & 3.27 & 4  & 0  & 0.06\\
        u3-18 &78   & 3.95 & 8 & 0  & 0.04\\
        u3-24 &129  & 4.25 & 8  & 0  & 0.08\\
        u3-30 &255  & 5.06 & 8  & 19 & 0.29 \\
        u3-36 &276  & 5.14 & 12 & 12 & 0.30\\
        u4-16 &75  & 4.03 & 8 & 1  & 0.04\\
        u4-24 &57   & 3.19 & 6  & 0  & 0.05\\
        u4-32 &177  & 4.14 & 10 & 3  & 0.21\\
        u4-40 &149  & 4.01 & 8  & 2  & 0.26\\
        u4-48 &1177 & 9.01 & 18 & 7  & 1.69\\
        u5-40 &335  & 5.28 & 14 & 1  & 0.49\\
        u5-50 &584  & 6.13 & 14 & 6  & 0.96\\
        \hline
        Avg   &257.07 &4.75 &9.71 & 4.14 &0.42 \\
        \hline
        r5-60  &632   & 6.44  & 16 & 44  & 2.61 \\
        r6-48   & 4082  & 14.20 & 36 & 414  & 6.89\\
        r6-60   &809  & 6.58  & 18 & 40   & 1.65\\
        r6-72   &1080  & 7.12  & 22 & 36   & 2.51\\
        r7-56   &1089  & 7.92  & 18 & 83   & 1.70\\
        r7-70   &2340  & 8.32  & 18 & 183  & 4.14\\
        r7-84   &2892  & 11.66 & 30 & 405  & 7.77\\
        r8-64   &11694 & 18.23 & 42 & 3517 & 40.52\\
        r8-80   &5822  & 14.89 & 30 & 260  & 14.07\\
        r8-96   &3155  & 9.30  & 26 & 312  & 9.65\\
        \hline
        Avg   &3359.50 &10.47 &25.6 &526.4 &9.15 \\
        \bottomrule
    \end{tabular}
    \end{threeparttable}
\end{table}

\newpage

\section{Single-objective E-ADARP v.s. weighted-sum objective E-ADARP} \label{effect of weighted sum objective}
In this section, we compare results from solving a single-objective function (i.e., minimizing total travel time) with the results of solving the weighted-sum objective function in Table \ref{BP results type-a with total travel time only} and Table \ref{BP results type-u with total travel time only}. The additional notations used in the tables are summarized as follows:
\begin{itemize}
    \item $Obj^{''}$: the objective values of optimal solutions from B\&P in solving the minimization of total travel time;
    \item $Obj_1^{''}$: the total travel times corresponding to optimal solutions of solely minimizing total travel time;
    \item $Obj_2^{''}$: the total excess user ride times corresponding to optimal solutions of solely minimizing total travel time;
    \item $N_{node}^{''}$: the number of nodes explored in the B\&P tree in solely minimizing total travel time;
    \item $CPU^{''}$: the CPU time in seconds of solely minimizing total travel time by applying B\&P algorithm;
    \item $Obj_1^{'}$: the total travel times corresponding to optimal solutions of minimizing weighted-sum objective;
    \item $Obj_2^{'}$: the total excess user times corresponding to optimal solutions of minimizing weighted-sum objective;
    \item Gap$_1$: the gaps between $Obj_1^{''}$ and $Obj_1^{'}$;
    \item Gap$_2$: the gaps between $Obj_2^{''}$ and $Obj_2^{'}$;
    \item -$\infty$: the gaps between $Obj_2^{''}$ and $Obj_2^{'}$ cannot be calculated as $Obj_2^{'} = 0$. 
\end{itemize}

And Gap$_1$, Gap$_2$ are calculated as follows:
$${Gap_k\% = \dfrac {Obj_k^{'}-Obj_k^{''}} {Obj_k^{'}}}\times 100\% \quad k \in \{1,2\}$$ 

\newpage

\begin{table}[!htp]
\renewcommand\arraystretch{0.75}
 \setlength\tabcolsep{3pt}
    \caption{B\&P results on type-a instances under $\gamma=0.1, 0.4, 0.7$: minimizing total travel time only v.s. minimizing the weighted sum of total travel time and total excess user ride time}
    \label{BP results type-a with total travel time only}
    \begin{center}
    \footnotesize
    \begin{tabular}{c|c c c c c| c c c c c|c c}
    \hline
        \textbf{$\gamma = 0.1$}&\multicolumn{5}{c}{Minimizing total travel time only} &\multicolumn{5}{c}{Minimizing the weighted-sum objective}&\multicolumn{2}{c}{Gaps on objectives} \\
        \hline
        Instance & $Obj^{''}$ & $Obj_1^{''}$ & $Obj_2^{''}$ &$N_{node}^{''}$ & CPU$^{''}$ &$Obj^{'}$ &$Obj_1^{'}$ &$Obj_2^{'}$ &$N_{node}$ &CPU$^{'}$(s) & Gap$_1$ & Gap$_2$\\
        \hline
        a2-16 &\textbf{294.25$^*$} &294.25 & 72.98  & 1  & 11.1   
        & \textbf{237.38$^*$} & 299.26 & 51.73  &1  &11.1 
        & 1.67\% & -41.08\%  \\
        a2-20 &\textbf{344.83$^*$} &344.83 & 112.56 & 1  & 56.0   
        & \textbf{279.08$^*$} & 347.89 & 72.63  &1  &70.9 
        & 0.88\% & -54.96\% \\
        a2-24 &\textbf{431.12$^*$} &431.12 & 157.41 & 1  & 204.70  
        & \textbf{346.21$^*$} & 433.26 & 85.08  &1  &243.3  
        & 0.49\% & -85.03\% \\
        a3-18 &\textbf{300.48$^*$} &300.48 & 119.06 & 1  & 7.16    
        & \textbf{236.82$^*$} & 304.34 & 34.24  &1  &5.0
        & 1.27\% & -247.68\%  \\
        a3-24 &\textbf{344.83$^*$} &344.83 & 120.31 & 1  & 76.92   
        & \textbf{274.80$^*$} & 355.78 & 31.87  &1  &81.2  
        & 3.08\% & -277.47\%\\
        a3-30 &\textbf{494.85$^*$} &494.85 & 197.18 & 1  & 176.70  
        & \textbf{413.27$^*$} & 500.69 & 150.99 &1  &221.7  
        & 1.17\% & -30.59\%\\
        a3-36 &\textbf{583.19$^*$} &583.19 & 244.37 & 5  & 998.09  
        & \textbf{481.17$^*$} & 600.85 & 122.12 &1  &730.8  
        & 3.37\% & -100.11\%\\
        a4-16 &\textbf{282.68$^*$} &282.68 & 44.49  & 1  & 2.45    
        & \textbf{222.49$^*$} & 286.45 & 30.60  &5  &5.5  
        & 1.32\% & -45.40\%\\
        a4-24 &\textbf{375.02$^*$} &375.02 & 159.64 & 1  & 19.02   
        & \textbf{310.84$^*$} & 388.83 & 76.88  &1  &25.0 
        & 3.55\% & -107.66\%\\
        a4-32 &\textbf{485.50$^*$} &485.50 & 165.20 & 1  & 98.14   
        & \textbf{393.95$^*$} & 502.34 & 68.80  &1  &143.2   
        & 3.35\% & -140.10\%\\
        a4-40 &\textbf{557.69$^*$} &557.69 & 167.65 & 1  & 240.21  
        & \textbf{453.84$^*$} & 562.44 & 128.04 &1  &653.0  
        & 0.84\% & -30.94\%\\
        a4-48 &\textbf{668.82$^*$} &668.82 & 333.24 & 3  & 1107.98 
        & \textbf{554.54$^*$} & 697.85 & 124.60 &1  &1334.4 
        & 4.36\% & -167.45\% \\
        a5-40 &\textbf{499.25$^*$} &499.25 & 242.78 & 1  & 194.22  
        & \textbf{414.51$^*$} & 524.75 & 83.78  &21  &448.3 
        & 4.86\% & -189.78\% \\
        a5-50 &\textbf{686.62$^*$} &686.62 & 395.64 & 67 & 7028.19 
        & \textbf{559.17$^*$} & 712.28 & 99.81  &1  &496.2 
        & 3.43\% & -296.38\%\\
        \hline
        Avg  & & & & & & & & &2.7 &319.3 &2.40\% &-129.50\%\\
        \hline
        \textbf{$\gamma = 0.4$}& $Obj^{''}$ & $Obj_1^{''}$ & $Obj_2^{''}$ &$N_{node}^{''}$ & CPU$^{''}$ &$Obj^{'}$ &$Obj_1^{'}$ &$Obj_2^{'}$ &$N_{node}$ &CPU$^{'}$(s) & Gap$_1$ & Gap$_2$\\
        \hline
        a2-16 &\textbf{294.25$^*$} &294.25 & 72.98  & 1  & 6.2     
        & \textbf{237.38$^*$} & 299.26 & 51.73  &1  & 12.7    
        & 1.67\% & -41.08\% \\
        a2-20 &\textbf{347.19$^*$} &347.19  & 87.81  & 3  & 99.6    
        & \textbf{280.70$^*$} & 350.06 & 72.63  &1  & 93.8   
        & 0.82\% & -20.90\%\\
        a2-24  &\textbf{432.13$^*$} &432.13 & 157.41 & 1  & 279.9   
        & \textbf{347.04$^*$} & 434.36 & 85.08  &1  & 267.8  
        & 0.51\% & -85.03\%\\
        a3-18 &\textbf{300.48$^*$} &300.48 & 119.06 & 1  & 2.2     
        & \textbf{236.82$^*$} & 304.34 & 34.24  &1  & 5.3   
        & 1.27\% & -247.68\% \\
        a3-24 &\textbf{344.83$^*$} &344.83 & 120.31 & 1  & 50.6    
        & \textbf{274.80$^*$} & 355.78 & 31.87  &1  & 69.7    
        & 3.08\% & -277.47\%\\
        a3-30 &\textbf{494.92$^*$} &494.92 & 198.53 & 1  & 308.3   
        & \textbf{413.34$^*$} & 500.80 & 150.99 &1  & 306.9   
        & 1.17\% & -31.49\%\\
        a3-36 &\textbf{583.17$^*$} &583.17 & 273.05 & 5  & 2993.6 
        & \textbf{482.75$^*$} & 598.16 & 136.51 &13  & 5154.1 
        & 2.51\% & -100.01\%\\
        a4-16 &\textbf{282.68$^*$} &282.68 & 44.49  & 1  & 1.7     
        & \textbf{222.49$^*$} & 286.45 & 30.60  &5  & 3.3  
        & 1.32\% & -45.40\%\\
        a4-24 &\textbf{375.02$^*$} &375.02 & 159.64 & 1  & 8.0     
        & \textbf{311.03$^*$} & 394.48 & 60.68  &1  & 18.4  
        & 4.93\% & -163.08\%\\
        a4-32 &\textbf{485.90$^*$} &485.90 & 165.20 & 1  & 81.8    
        & \textbf{394.26$^*$} & 502.74 & 68.80  &1  & 199.3   
        & 3.35\% & -140.10\%\\
        a4-40 &\textbf{557.69$^*$} &557.69 & 167.65 & 1  & 258.4   
        & \textbf{453.84$^*$} & 562.44 & 128.04 &1  & 792.0  
        & 0.84\% & -30.94\%\\
        a4-48 &\textbf{668.59$^*$} &668.59 & 334.93 & 3  & 1963.5 
        & \textbf{554.60$^*$} & 697.93 & 124.60 &1  & 2292.0
        & 4.20\% & -168.81\%\\
        a5-40 &\textbf{499.25$^*$} &499.25 & 242.78 & 1  & 220.0   
        & \textbf{414.50$^*$} & 524.75 & 83.78  &3  & 323.6   
        & 4.86\% & -189.78\%\\
        a5-50 &\textbf{693.84$^*$} &693.84 & 395.64 & 53 & 7162.0 
        & \textbf{559.51$^*$} & 712.74 & 99.81  &1  & 1957.1 
        & 2.65\% & -296.38\%\\
        \hline
        Avg  & & & & & & & & &2.3 &821.1 &2.37\% &-131.29\%\\
        \hline
        \textbf{$\gamma = 0.7$} &$Obj^{''}$ & $Obj_1^{''}$ & $Obj_2^{''}$ &$N_{node}^{''}$ & CPU$^{''}$ &$Obj^{'}$ &$Obj_1^{'}$ &$Obj_2^{'}$ &$N_{node}$ &CPU$^{'}$(s) & Gap$_1$ & Gap$_2$\\
        \hline
        a2-16 &\textbf{298.63$^*$} & 298.63 & 72.98  & 3  & 23.3   
        & \textbf{240.66$^*$} & 303.64 & 51.73  &1  &20.8  
        & 1.65\% & -41.08\%  \\
        a2-20 &\textbf{363.27$^*$} & 363.27 & 92.58  & 1  & 625.6  
        & \textbf{293.27$^*$} & 372.95 & 54.20  &5  &1758.6  
        & 2.03\% & -33.43\% \\
        a2-24 &\textbf{440.70$^*$} & 440.70 & 194.73 & 19 & 4508.6 
        & \textbf{353.18$^*$} & 442.55 & 85.08  &1  &1241.9  
        & 0.42\% & -128.90\%\\
        a3-18 &\textbf{303.71$^*$} & 303.71 & 115.54 & 13 & 44.5   
        & \textbf{240.58$^*$} & 309.36 & 34.24  &1  &14.1 
        & 1.83\% & -237.40\%\\
        a3-24 &\textbf{346.38$^*$} & 346.38 & 120.31 & 1  & 124.7  
        & \textbf{275.97$^*$} & 357.33 & 31.87  &1  &100.8  
        & 3.06\% & -277.47\%\\
        a3-30 &\textbf{506.43$^*$} & 506.43 & 211.76 & 3  & 1722.8 
        & \textbf{424.93$^*$} & 515.66 & 152.74 & 3 &1893.3  
        & 1.79\% & -38.64\%\\
        a3-36 &\textbf{599.75$^*$} & 599.75 & 261.41 & 15 & 2513.0 
        & 494.04 & 618.01 & 122.12 & 20 &7200.0 
        & 2.95\% & -114.06\%\\
        a4-16 &\textbf{282.68$^*$} & 282.68 & 44.49  & 1  & 3.4    
        & \textbf{223.13$^*$} & 282.68 & 44.49  &1  &8.4  
        & 0.00\% & 0.00\% \\
        a4-24 & \textbf{381.70$^*$} & 381.70 & 149.72 & 1  & 31.5   
        & \textbf{316.65$^*$} & 401.98 & 60.68  &7  &164.4  
        & 5.05\% & -146.72\%\\
        a4-32 &\textbf{491.23$^*$} & 491.23 & 165.20 & 7  & 398.6  
        & \textbf{397.87$^*$} & 504.80 & 77.08  &9  &1638.6
        & 2.69\% & -114.32\%\\
        a4-40 &\textbf{575.97$^*$} & 575.97 & 178.15 & 9  & 3610.8 
        & \textbf{467.72$^*$} & 580.46 & 129.49 &9  &3987.4   
        & 0.77\% & -37.58\% \\
        a4-48 &NA       &NA        &NA        &1    &7200.0         
        & NA & NA & NA &1  &7200.0  &NA      &NA\\
        a5-40 &\textbf{507.48$^*$}  &507.48  &255.15   &9   &1535.3 
        &\textbf{418.75$^*$} & 531.14 & 81.60  &15  &1572.0  &4.54\%   &-212.67\% \\
        a5-50 &NA      &NA        &NA        &1       &7200.0         
        & \textbf{690.79} & 815.45 & 316.82 &1  &7200.0  &NA      &NA \\
        \hline
        Avg  & & & & & & & & &5.4 &2428.6 &NA &NA\\
        \hline
    \end{tabular}
    \end{center}
\end{table}

\newpage

\begin{table}[!htp]
\renewcommand\arraystretch{0.7}
 \setlength\tabcolsep{3pt}
    \caption{B\&P results on type-u instances under $\gamma=0.1, 0.4, 0.7$: minimizing total travel time only v.s. minimizing the weighted sum of total travel time and total excess user ride time}
    \label{BP results type-u with total travel time only}
    \begin{center}
    \small
    \begin{tabular}{c|c c c c c| c c c c c|c c}
    \hline
        \textbf{$\gamma = 0.1$}&\multicolumn{5}{c}{Minimizing total travel time only} &\multicolumn{5}{c}{Minimizing the weighted-sum objective}&\multicolumn{2}{c}{Gaps on objectives} \\
        \hline
        Instance & $Obj^{''}$ & $Obj_1^{''}$ & $Obj_2^{''}$ &$N_{node}^{''}$ & CPU$^{''}$ &$Obj^{'}$ &$Obj_1^{'}$ &$Obj_2^{'}$ &$N_{node}$ &CPU$^{'}$(s) & Gap$_1$ & Gap$_2$\\
        \hline
        u2-16 & \textbf{73.77$^*$}  & 73.77  & 12.21 & 3  & 18.1    
        & \textbf{57.61$^*$}  & 76.81  & 0.00  &3  &22.2  
        & 3.96\% & -$\infty$   \\
        u2-20 & \textbf{71.56$^*$}  & 71.56  & 15.19 & 1  & 65.2    
        & \textbf{55.59$^*$}  &73.70  & 1.24  &1  &69.3  
        & 2.90\% & -1123.65\%  \\
        u2-24 & \textbf{115.98$^*$} & 115.98 & 16.91 & 3  & 1118.0  
        & \textbf{90.66$^*$}  & 118.49 & 7.43  &7  &1649.0 
        & 2.12\% & -127.53\% \\
        u3-18 &\textbf{66.98$^*$}  & 66.98  & 9.18  & 1  & 30.2    
        & \textbf{50.74$^*$}  & 67.65  & 0.00  &1  &14.0  
        & 0.99\% & -$\infty$  \\
        u3-24 &\textbf{83.95$^*$}  & 83.95  & 29.16 & 1  & 40.3    
        & \textbf{67.56$^*$}  & 86.08  & 12.01 &1  &107.0  
        & 2.47\% & -142.74\%  \\
        u3-30 &\textbf{99.22$^*$}  & 99.22  & 21.77 & 3  & 1002.0  
        & \textbf{76.75$^*$}  & 100.83 & 4.50  &1  & 570.8
        & 1.60\% & -384.03\%  \\
        u3-36 &\textbf{127.97$^*$} & 127.97 & 45.23 & 13 & 6927.9 
        & \textbf{104.04$^*$} & 133.79 & 14.80 &11  &4204.6  
        & 1.79\% & -88.55\%  \\
        u4-16 &\textbf{65.24$^*$}  & 65.24  & 28.07 & 1  & 10.7    
        & \textbf{53.58$^*$}  & 68.76  & 8.06  &1  &1.9  
        & 5.11\% & -248.34\%  \\
        u4-24 &\textbf{115.10$^*$} & 115.10 & 16.13 & 1  & 39.9    
        & \textbf{89.83$^*$}  & 118.21 & 4.67  &1  &25.5  
        & 2.63\% & -245.40\% \\
        u4-32 &\textbf{126.71$^*$} & 126.71 & 29.81 & 3  & 423.5   
        & \textbf{99.29$^*$}  & 129.19 & 9.59  &1  &366.2 
        & 1.92\% & -210.93\%   \\
        u4-40 &\textbf{165.56$^*$} & 165.56 & 62.84 & 1  & 1788.6  
        & \textbf{133.11$^*$} & 168.40 & 27.25 &1  &1376.6  
        & 1.68\% & -130.63\%  \\
        u4-48 &\textbf{181.01$^*$} & 181.01 & 66.09 & 1  & 7196.3 
        & 148.08 & 185.69 & 32.34 &1  &7200.0  
        & 2.52\% & -104.40\%  \\
        u5-40 &\textbf{150.51$^*$} & 150.51 & 43.60 & 1  & 774.7   
        & \textbf{121.86$^*$} & 153.38 & 27.28 &1  &496.9  
        & 1.87\% & -59.81\% \\
        u5-50 &177.40 & 177.40 & 51.88 & 26 & 7200.0 
        & 142.82 & 180.08 & 31.08 &7  &7200.0  
        & 1.48\% & -66.93\% \\
        \hline
        Avg  & & &  &  &  &  &  &   &9.5  &1659.8  &2.24\%  &-$\infty$  \\
        \hline
        \textbf{$\gamma = 0.4$} & $Obj^{''}$ & $Obj_1^{''}$ & $Obj_2^{''}$ &$N_{node}^{''}$ & CPU$^{''}$ &$Obj^{'}$ &$Obj_1^{'}$ &$Obj_2^{'}$ &$N_{node}$ &CPU$^{'}$(s) & Gap$_1$ & Gap$_2$\\
        \hline
        u2-16 &\textbf{73.77$^*$}  & 73.77  & 12.21 & 9  & 68.3    
        & \textbf{57.65$^*$}  & 76.86  & 0.00 &1  &53.2  
        & 4.02\% & -$\infty$  \\
        u2-20 &\textbf{71.56$^*$}  & 71.56  & 15.19 & 1  & 96.2   
        & \textbf{56.34$^*$}  & 73.76  & 5.22  &3  &281.3  
        & 2.99\% & -190.93\%  \\
        u2-24 &\textbf{116.70$^*$} & 116.70 & 14.14 & 5  & 1413.0  
        & \textbf{91.16$^*$}  & 116.70 & 14.14  &25  &6917.0  
        & 0 & 0  \\
        u3-18 &\textbf{66.98$^*$}  & 66.98  & 9.18  & 1  & 18.7    
        & \textbf{50.74$^*$}  & 67.65  & 0.00 &1  &15.0  
        & 0.99\% & -$\infty$  \\
        u3-24 &\textbf{83.95$^*$}  & 83.95  & 29.16 & 1  & 127.9   
        & \textbf{67.56$^*$}  & 86.08  & 12.01 &1  &58.5  
        & 2.47\% & -142.74\% \\
        u3-30 &\textbf{99.22$^*$}  & 99.22  & 20.77 & 3  & 977.0   
        & \textbf{76.75$^*$}  & 100.83 & 4.50  &1 &500.2  
        & 1.60\% & -361.79\%  \\
        u3-36 &128.77 & 128.67 & 51.09 & 12 & 7200.0 
        & \textbf{104.06$^*$} & 132.22 & 19.57 &1  &1769.0  
        & 2.61\% & -161.03\%  \\
        u4-16 &\textbf{65.24$^*$}  & 65.24  & 28.07 & 1  & 3.1     
        & \textbf{53.58$^*$}  & 68.76  & 8.06  &1  & 2.2 
        & 5.11\% & -248.34\%  \\
        u4-24 &\textbf{115.10$^*$} & 115.10 & 16.13 & 1  & 19.9    
        & \textbf{89.83$^*$}  & 118.21 & 4.67  &1  &30.8 
        & 2.63\% & -245.40\% \\
        u4-32 &\textbf{126.55$^*$} & 126.55 & 29.81 & 3  & 609.0   
        & \textbf{99.29$^*$}  & 129.19 & 9.59  &1  &393.4  
        & 2.05\% & -210.93\%\\
        u4-40 &\textbf{165.97$^*$} & 165.97 & 62.84 & 1  & 1907.3  
        & 133.78 & 169.34 & 27.25 &32  &7200.0  
        & 1.99\% & -130.63\% \\
        u4-48 &181.38 & 181.38 & 66.09 & 3  & 7200.0 
        & 147.63 &188.73 & 28.28 &1  &7200.0  
        & 3.90\% &-133.74\% \\
        u5-40 &\textbf{150.51$^*$} & 150.51 & 43.60 & 1  & 1380.5  
        & \textbf{121.96$^*$} & 153.52 & 27.28 &1  &723.7  
        & 1.96\% & -59.81\%  \\
        u5-50 &177.40 & 177.40 & 51.88 & 11 & 7200.0 
        & 142.84 & 180.08 & 31.08 &1  &7200.0  
        & 1.48\% & -66.93\% \\
        \hline
        Avg  & & &  &  &  &  &  &   &5.1  &2310.2  &2.38\%  &-$\infty$  \\
        \hline
        \textbf{$\gamma = 0.7$} & $Obj^{''}$ & $Obj_1^{''}$ & $Obj_2^{''}$ &$N_{node}^{''}$ & CPU$^{''}$ &$Obj^{'}$ &$Obj_1^{'}$ &$Obj_2^{'}$ &$N_{node}$ &CPU$^{'}$(s) & Gap$_1$ & Gap$_2$\\
        \hline
        u2-16 &\textbf{76.08$^*$}  & 76.08  & 13.56 & 53 & 1432.3  
        & \textbf{59.19$^*$}  &78.93  & 0.00  &39  &679.8  
        & 3.61\% & -$\infty$  \\
        u2-20 &\textbf{73.75$^*$}  & 73.75  & 16.24 & 7  & 890.1  
        & \textbf{56.86$^*$}  & 75.40  & 1.24  &1  &337.0  
        & 2.19\% & -1207.73\%  \\
        u2-24 &\textbf{117.54$^*$} & 117.54 & 16.91 & 5  & 6336.3 
        & 92.01  & 120.20 & 7.43   &30  &7200.0  
        &2.21\% & -127.53\%   \\
        u3-18 &\textbf{67.44$^*$}  & 67.44  & 4.56  & 9  & 94.5   
        & \textbf{50.99$^*$}  & 67.99  & 0.00 &3  &45.2  
        & 0.81\% & -$\infty$   \\
        u3-24 &\textbf{84.10$^*$}  & 84.10  & 33.58 & 1  & 119.5  
        & \textbf{68.34$^*$}  & 87.12  & 12.01  &27  & 541.9
        &3.67\% &-179.52\%  \\
        u3-30 &\textbf{99.97$^*$}  & 99.97  & 34.71 & 44 & 6928.0  
        & \textbf{77.41$^*$}  & 101.72 & 4.50  &5  &1725.5  
        & 1.72\% & -671.92\%   \\
        u3-36 &\textbf{131.80$^*$} & 131.80 & 47.84 & 8  & 7164.0 
        & 105.79 & 134.53 & 19.57 &15  &7200.0  
        & 2.02\% & -144.43\%  \\
        u4-16 &\textbf{65.86$^*$}  & 65.86  & 28.07 & 3  & 11.0    
        & \textbf{53.87$^*$}  & 69.14  & 8.06  &1  &4.0  
        & 4.74\% & -248.34\%  \\
        u4-24 &\textbf{115.10$^*$} & 115.10 & 16.13 & 1  & 38.5    
        & \textbf{89.96$^*$}  & 118.40 & 4.67  &1  &82.3  
        & 2.79\% & -245.40\% \\
        u4-32 &\textbf{126.99$^*$} & 126.99 & 29.81 & 23 & 4991.2  
        & \textbf{99.50$^*$}  & 128.58 & 12.26 &1  &1028.6  
        & 1.24\% & -143.23\%   \\
        u4-40 &169.09 &169.09 &62.84  &28  &7200.0  
        &135.76  &173.75  &21.79   & 43 & 7200.0 
        &2.68\%  &-188.43\%  \\
        u4-48 &NA &NA &NA  &1  &7200.0  
        &185.16  &230.49  &49.17   &1  &7200.0  
        &NA  &NA  \\
        u5-40 &\textbf{151.24$^*$} &151.24 &49.50  &7  &7130.7  
        &\textbf{122.93$^*$}  &153.78  &30.40   &7  &6513.9  
        &1.64\%  &-62.85\%  \\
        u5-50 &246.81 &246.81 &92.38  &1  &7200.0  
        &195.72  &254.17  &20.38   &1  &7200.0  
        &2.90\%  &-353.29\%  \\
        \hline
        Avg  & & &  &  &  
        &  &  &   &12.4  &3354.1  
        &NA  &NA  \\
        \hline
    \end{tabular}
    \end{center}
\end{table}

\section{Experimental results of initializing column pool with a single iteration of DA algorithm} \label{experiments 1 iter}
In this section, we present the B\&P results with only 1 iteration of the DA algorithm to initialize the column pool. These results are compared with our B\&P results in Section \ref{sec:: BP experimental results}, which initializes the column pool with 500 iterations of the DA algorithm.

\begin{table}[!hp]
\renewcommand\arraystretch{0.75}
 \setlength\tabcolsep{3pt}
    \caption{B\&P results on type-a instances: initializing column pool with 1 iteration DA v.s. 500 iterations DA}
    \label{BP results type-a with 1 iteration}
    \vspace{-4mm}
    \begin{center}
    \footnotesize
    \begin{tabular}{c|c c c c c| c c c c c|c c c c}
    \hline
        \textbf{$\gamma = 0.1$}&\multicolumn{5}{c}{initializing with 1 DA iteration} &\multicolumn{5}{c}{initializing with 500 DA iterations} &\multicolumn{4}{c}{Bongiovanni et al., (2019) results $^a$} \\
        \hline
        Instance &$Obj_{1iter}$ &$LB_{1iter}$ &$LB_{1iter}\%$  &$N_{node}^{1iter}$ &CPU$_{1iter}$(s) & $Obj^{'}$ & $LB^{'}$ & $LB^{'}\%$  &$N_{node}$ & CPU$^{'}$(s) &$Obj_2$ & $LB_2$ & $LB_2\%$ &CPU(s) \\
        \hline
        a2-16 &\textbf{237.38$^*$} & \textbf{237.38} & 0 & 1 & 42.0   
        & \textbf{237.38$^*$} & \textbf{237.38} & 0 &1  &11.1  
        & \textbf{237.38$^*$} & \textbf{237.38} &0   & 1.2\\
        a2-20 &\textbf{279.08$^*$} & \textbf{279.08} & 0 & 1 & 130.6  
        & \textbf{279.08$^*$} & \textbf{279.08} & 0 &1  & 70.9 
        &\textbf{279.08$^*$} & \textbf{279.08} & 0&4.2\\
        a2-24 &\textbf{346.21$^*$} & \textbf{346.21} & 0 & 1 & 328.4  
        & \textbf{346.21$^*$} & \textbf{346.21} & 0 &1  & 243.3 
        &\textbf{346.21$^*$} & \textbf{346.21} &0 &9.0\\
        a3-18 &\textbf{236.82$^*$} & \textbf{236.82} & 0 & 1 & 10.2   
        & \textbf{236.82$^*$} & \textbf{236.82} & 0 &1  & 5.0 
        &\textbf{236.82$^*$} & \textbf{236.82}  &0 &4.8\\
        a3-24 &\textbf{274.80$^*$} & \textbf{274.80} & 0 & 1 & 83.1   
        & \textbf{274.80$^*$} & \textbf{274.80} & 0 &1  & 81.2 
        &\textbf{274.80$^*$} & \textbf{274.80} &0  &13.8\\
        a3-30 &\textbf{413.27$^*$} & \textbf{413.27} & 0 & 1 & 310.7  
        & \textbf{413.27$^*$} & \textbf{413.27} & 0 &1  & 221.7
        & \textbf{413.27$^*$} &\textbf{413.27}  &0 &102.0 \\
        a3-36 &\textbf{481.17$^*$} & \textbf{481.17} & 0 & 1 & 691.7  
        & \textbf{481.17$^*$} & \textbf{481.17} & 0 &1  &730.8  
        & \textbf{481.17$^*$} &\textbf{481.17}  &0 &106.8\\
        a4-16 &\textbf{222.49$^*$} & \textbf{222.49} & 0 & 5 & 3.5    
        & \textbf{222.49$^*$} & \textbf{222.49} & 0 &5  & 5.5 
        &\textbf{222.49$^*$}  &\textbf{222.49}  &0 & 3.6\\
        a4-24 &\textbf{310.84$^*$} & \textbf{310.84} & 0 & 3 & 43.2   
        & \textbf{310.84$^*$} & \textbf{310.84} & 0 &1  & 25.0 
        &\textbf{310.84$^*$} &\textbf{310.84} &0  &31.2\\
        a4-32 &\textbf{393.95$^*$} & \textbf{393.95} & 0 & 1 & 225.3  
        & \textbf{393.95$^*$} & \textbf{393.95} & 0 &1  & 143.2 
        &\textbf{393.96$^*$} &\textbf{393.96} &0 &612.0\\
        a4-40 &\textbf{453.84$^*$} & \textbf{453.84} & 0 & 1 & 555.9  
        & \textbf{453.84$^*$} & \textbf{453.84} & 0 &1  & 653.0
        &\textbf{453.84$^*$} & \textbf{453.84} &0 &517.2 \\
        a4-48 &\textbf{554.54$^*$} & \textbf{554.54} & 0 & 1 & 2117.2 
        & \textbf{554.54$^*$} & \textbf{554.54} & 0 &1  & 1334.4 
        &\textbf{554.54}  &526.96 &5.04\% & 7200.0\\
        a5-40 &\textbf{414.50$^*$} & \textbf{414.50} & 0 & 9 & 1306.4 
        & \textbf{414.50$^*$} & \textbf{414.50} & 0 &21  & 448.3 
        &\textbf{414.51$^*$} & \textbf{414.51} &0 &1141.8\\
        a5-50 &\textbf{559.17$^*$} & \textbf{559.17} & 0 & 1 & 2502.0 
        & \textbf{559.17$^*$} & \textbf{559.17} & 0 &1  & 496.2 
        &\textbf{559.17} & 531.15  &5.01\% &7200.0\\
        \hline
        Avg  & & &0  &2.0  &596.4  &  &  &0   &2.7  & 319.3 & & &0.72\% &1210.5\\
        \hline
        \textbf{$\gamma = 0.4$}&$Obj_{1iter}$ &$LB_{1iter}$ &$LB_{1iter}\%$  &$N_{node}^{1iter}$ &CPU$_{1iter}$(s) & $Obj^{'}$ & $LB^{'}$ & $LB^{'}\%$  &$N_{node}$ & CPU$^{'}$(s) &$Obj_2$ & $LB_2$ & $LB_2\%$ &CPU(s)\\
        \hline
        a2-16 &\textbf{237.38$^*$} & \textbf{237.38} & 0 & 1  & 18.3   
        & \textbf{237.38$^*$} & \textbf{237.38} & 0 & 1  &12.7  
        &\textbf{237.38$^*$} &\textbf{237.38} &0 &1.8\\
        a2-20 &\textbf{280.70$^*$} & \textbf{280.70} & 0 & 1  & 174.1  
        & \textbf{280.70$^*$} & \textbf{280.70} & 0 & 1  &93.8  
        &\textbf{280.70$^*$}  &\textbf{280.70} &0 &49.8\\
        a2-24 &\textbf{347.04$^*$} & \textbf{347.04} & 0 & 1  & 437.8  
        & \textbf{347.04$^*$} & \textbf{347.04} & 0 & 1  &267.8  
        &\textit{348.04$^*$} &\textit{348.04} &-0.29\% &25.2\\
        a3-18 &\textbf{236.82$^*$} & \textbf{236.82} & 0 & 1  & 6.7    
        & \textbf{236.82$^*$} & \textbf{236.82} & 0 & 1  &5.3  
        &\textbf{236.82$^*$} & \textbf{236.82} &0 &4.2 \\
        a3-24 &\textbf{274.80$^*$} & \textbf{274.80} & 0 & 1  & 153.9  
        & \textbf{274.80$^*$} & \textbf{274.80} & 0 & 1  &69.7   
        &\textbf{274.80$^*$} &\textbf{274.80}  &0 &16.8\\
        a3-30 &\textbf{413.34$^*$} & \textbf{413.34} & 0 & 1  & 409.4  
        & \textbf{413.34$^*$} & \textbf{413.34} & 0 & 1  &306.9   
        &\textit{413.37$^*$} &\textit{413.37} &-0.01\% &99.0\\
        a3-36 &\textbf{482.75$^*$} & \textbf{482.75} & 0 & 11 & 4950.0 
        & \textbf{482.75$^*$} & \textbf{482.75} & 0 & 13 &5154.1  
         &\textit{484.14$^*$} &\textit{484.14} &-0.06\% &306.6\\
        a4-16 &\textbf{222.49$^*$} & \textbf{222.49} & 0 & 3  & 3.6    
        & \textbf{222.49$^*$} & \textbf{222.49} & 0 & 5  &3.3   
        &\textbf{222.49$^*$} &\textbf{222.49} &0 &5.4\\
        a4-24 &\textbf{311.03$^*$} & \textbf{311.03} & 0 & 3  & 54.1   
        & \textbf{311.03$^*$} & \textbf{311.03} & 0 & 1  &18.4 
        &\textbf{311.03$^*$} &\textbf{311.03} &0 &39.6\\
        a4-32 &\textbf{394.26$^*$} & \textbf{394.26} & 0 & 1  & 466.7 
        & \textbf{394.26$^*$} & \textbf{394.26} & 0 & 1  &199.3  
        &\textbf{394.26$^*$} &\textbf{394.26} &0 &681.6\\
        a4-40 &\textbf{453.84$^*$} & \textbf{453.84} & 0 & 1  & 640.8  
        & \textbf{453.84$^*$} & \textbf{453.84} & 0 & 1  &792.0 
        &\textbf{453.84$^*$} &\textbf{453.84} &0 &417.6\\
        a4-48 &\textbf{554.60$^*$} & \textbf{554.60} & 0 & 1  & 4991.9 
        & \textbf{554.60$^*$} & \textbf{554.60} & 0 & 1  &2292.0 
        &\textbf{554.60} &529.22 &4.58\% &7200.0\\
        a5-40 &\textbf{414.51$^*$} & \textbf{414.51} & 0 & 3  & 1116.2 
        & \textbf{414.51$^*$} & \textbf{414.51} & 0 & 3  &323.6  
        &\textbf{414.51$^*$} &\textbf{414.51} &0 &1221.0\\
        a5-50 &\textbf{559.51$^*$} & \textbf{559.51} & 0 & 1  & 2686.0 
        & \textbf{559.51$^*$} & \textbf{559.51} & 0 & 1  &1957.1 
        &560.50 &528.91 &5.47\% &7200.0 \\
        \hline
        Avg  & & & 0 &2.1  &1150.7  &  &  &0   &2.3  &821.1 & & &0.74\% &1233.5 \\
        \hline
        \textbf{$\gamma = 0.7$} &$Obj_{1iter}$ &$LB_{1iter}$ &$LB_{1iter}\%$  &$N_{node}^{1iter}$ &CPU$_{1iter}$(s) & $Obj^{'}$ & $LB^{'}$ & $LB^{'}\%$  &$N_{node}$ & CPU$^{'}$(s) &$Obj_2$ & $LB_2$ & $LB_2\%$ &CPU(s)\\
        \hline
        a2-16 &\textbf{240.66$^*$}  & \textbf{240.66} & 0 & 1  & 35.2    
        &\textbf{240.66$^*$} &\textbf{240.66} &0 &1 &20.8 
        &\textbf{240.66$^*$} &\textbf{240.66} &0 &5.4 \\
        a2-20 &\textbf{293.27$^*$}  & \textbf{293.27} & 0 & 7  & 2525.4  
        &\textbf{293.27$^*$} &\textbf{293.27} &0  &5 &1758.6  
        &NA &287.17 &2.08\% &7200.0\\
        a2-24 &\textbf{353.18$^*$}  & \textbf{353.18} & 0 & 1  & 767.6   
        &\textbf{353.18$^{*}$} &\textbf{353.18} &0 &1 &1241.9 
        &\textit{358.21$^*$} &\textit{358.21}  &-1.42\% &961.2\\
        a3-18 &\textbf{240.58$^*$}  & \textbf{240.58} & 0 & 1  & 23.0    
        &\textbf{240.58$^*$} &\textbf{240.58} &0 &1 & 14.1  
        &\textbf{240.58$^*$} &\textbf{240.58}  &0 &48.0\\
        a3-24 &\textbf{275.97$^*$}  & \textbf{275.97} & 0 & 1  & 197.8   
        &\textbf{275.97$^{*}$} &\textbf{275.97} &0 &1 & 100.8  
        &\textit{277.72$^*$} &\textit{277.72} &-0.63\%  &152.4\\
        a3-30 &\textbf{424.93$^*$}  & \textbf{424.93} & 0 & 3  & 2158.5  
        &\textbf{424.93$^*$} &\textbf{424.93} &0 &3 &1893.3 
        &NA &417.06 &1.85\%  &7200.0 \\
        a3-36 &\textbf{494.04}  & 493.79 & 0.10\% & 22 & 7200.0  
        &\textbf{494.04} &\textbf{494.01} &0.01\% &20 &7200.0  
        &\textbf{494.04} &485.91 &1.65\% &7200.0\\
        a4-16 &\textbf{223.13$^*$}  & \textbf{223.13} & 0 & 1  & 3.4   
        &\textbf{223.13$^{*}$} &\textbf{223.13} &0 &1 &8.4 
        &\textbf{223.13$^*$} &\textbf{223.13} &0 &67.2\\
        a4-24 &\textbf{316.65$^*$}  & \textbf{316.65} & 0 & 7  & 222.8   
        &\textbf{316.65$^{*}$} &\textbf{316.65} &0 &7 &164.4 
        &\textit{318.21$^*$} &\textit{318.19} &-0.49\% &1834.8\\
        a4-32 &\textbf{397.87$^*$}  & \textbf{397.87} & 0 & 7  & 905.9   
        &\textbf{397.87$^{*}$} &\textbf{397.87} &0 &9 &1638.6  
        &430.07 &387.99  &2.48\% &7200.0 \\
        a4-40 &\textbf{467.72$^*$}  & \textbf{467.72} & 0 & 9  & 3053.4  
        &\textbf{467.72$^{*}$} &\textbf{467.72} &0 &9 &3987.4  
        &NA &443.62  &5.15\% &7200.0 \\
        a4-48 &NA & 470.15 & NA  & 1  & 7200.0 
        &NA  &476.54 &NA &1 &7200.0  
        &NA &524.92 &NA &7200.0 \\
        a5-40 &\textbf{418.75$^*$}  & \textbf{418.75} & 0 & 7  & 2013.7  
        &\textbf{418.75$^{*}$} &\textbf{418.75} & 0  &15 &1572.0  
        &447.63 &405.99  &4.78\% &7200.0\\
        a5-50 &NA & 179.35 & NA  & 1  & 7200.0   
        &NA  &176.58 &NA  & 1 &7200.0 
        &NA &522.37 &NA &7200.0 \\
        \hline
        Avg  & & &0  &4.9  &2393.3  &  &  &0   &5.4  &2428.6 
        & & &1.70\% &4333.5\\
        \hline
        Summary &$\#opt$ & $\#best lb$ &$\overline{LB_{1iter}}\%$ &$\overline{N_{node}^{1iter}}$ & $\#best ub$ &$\#opt$ & $\#best lb$ &$\overline{LB'}\%$ &$\overline{N_{node}}$ &$\#best ub$  
        &$\#opt$ & $\#best lb$ & $\overline{LB_2}\%$  & $\#best ub$\\
        &39 &39 &0.003\% & &40 
        &39 &40 &0 &3.5 &40
        &24 &26 &1.06\% &28\\
        \hline
    \end{tabular}
    \end{center}
\end{table}

\begin{table}[!hp]
\renewcommand\arraystretch{0.8}
 \setlength\tabcolsep{3pt}
    \caption{B\&P results on type-u instances: initializing column pool with 1 iteration DA v.s. 500 iterations DA}
    \label{BP results type-u with 1 iteration}
    \begin{center}
    \footnotesize
    \begin{tabular}{c|c c c c c| c c c c c| c c c c}
    \hline
        \textbf{$\gamma = 0.1$}&\multicolumn{5}{c}{initializing with 1 DA iteration} &\multicolumn{5}{c}{initializing with 500 DA iteration}&\multicolumn{4}{c}{Bongiovanni et al., (2019) results $^a$} \\
        \hline
        Instance &$Obj_{1iter}$ &$LB_{1iter}$ &$LB_{1iter}\%$  &$N_{node}^{1iter}$ &CPU$_{1iter}$(s) & $Obj^{'}$ & $LB^{'}$ & $LB^{'}\%$  &$N_{node}$ & CPU$^{'}$(s) &$Obj_2$ & $LB_2$ & $LB_2\%$ &CPU(s)\\
        \hline
        u2-16 &\textbf{57.61$^*$}    & \textbf{57.61}    & 0 & 3  & 84.1    &\textbf{57.61$^*$} &\textbf{57.61} &0 &3 &22.2 
        &\textbf{57.61$^*$} &\textbf{57.61} &0 &21.0\\
        u2-20 &\textbf{55.59$^*$}    & \textbf{55.59}    & 0 & 1  & 765.2   &\textbf{55.59$^*$} &\textbf{55.59} &0 &1 &69.3 
        &\textbf{55.59$^*$} &\textbf{55.59} &0 &9.6\\
        u2-24 &\textbf{90.66$^*$}    & \textbf{90.66}    & 0 & 13 & 4524.2  &\textbf{90.66$^{*}$} &\textbf{90.66} &0 &7 &1649.0 
         &\textit{91.27$^*$} & \textit{91.27} & -0.67\% &432.0 \\
        u3-18 &\textbf{50.74$^*$}    & \textbf{50.74}    & 0 & 1  & 94.6    &\textbf{50.74$^*$} &\textbf{50.74} &0 &1 &14.0
        &\textbf{50.74$^*$} &\textbf{50.74} &0 &10.8\\
        u3-24 &\textbf{67.56$^*$}    & \textbf{67.56}    & 0 & 1  & 309.3   &\textbf{67.56$^*$} &\textbf{67.56} &0 &1 &40.3  
        &\textbf{67.56$^*$} &\textbf{67.56} &0 &130.2 \\
        u3-30 &\textbf{76.75$^*$}    & \textbf{76.75}    & 0 & 1  & 1753.4  &\textbf{76.75$^*$} &\textbf{76.75} &0 &1 &570.8
        &\textbf{76.75$^*$} &\textbf{76.75} &0 &438.0\\
        u3-36 &\textbf{104.04$^*$}   & \textbf{104.04}   & 0 & 5  & 5375.1  &\textbf{104.04$^*$} &\textbf{104.04} &0 &11 &4204.6  
        &\textbf{104.04$^*$} &\textbf{104.04} &0 &1084.8 \\
        u4-16 &\textbf{53.58$^*$}    &\textbf{53.58}    & 0 & 1  & 62.4   &\textbf{53.58$^*$} &\textbf{53.58} &0 &1 &1.9
        &\textbf{53.58$^*$} &\textbf{53.58} &0 &48.0\\
        u4-24 &\textbf{89.83$^*$}    &\textbf{89.83}    & 0 & 1  & 60.4   &\textbf{89.83$^*$} &\textbf{89.83} &0 &1 &25.5 
        &\textbf{89.83$^*$} &\textbf{89.83} &0 &13.2\\
        u4-32 &\textbf{99.29$^*$}    &\textbf{99.29}    & 0 & 1  & 717.1   &\textbf{99.29$^*$} &\textbf{99.29} &0 &1 &366.2   
        &\textbf{99.29$^*$} &\textbf{99.29} &0 &1158.6\\
        u4-40 &\textbf{133.11$^*$}   &\textbf{133.11}   & 0 & 1  & 3273.5 &\textbf{133.11$^*$} &\textbf{133.11} &0 &1 &1376.6
        &\textbf{133.11$^*$} &\textbf{133.11} &0 &185.4\\
        u4-48 &NA &NA &NA  &1  &7200.0  
        &\textbf{148.08} &147.02 &0.72\% &1 &7200.0 
        &148.30 &134.48 &9.18\% &7200.0 \\
        u5-40 &\textbf{121.86$^*$}   & \textbf{121.86}   & 0 & 1  & 1542.1  &\textbf{121.86$^*$} &\textbf{121.86} &0 &1 &496.9 
        &\textbf{121.86} &114.12 &6.35\% &7200.0\\
        u5-50 &\textbf{142.82}   & \textbf{142.75}   & 0.05\% & 3  & 7200.0 
        &\textbf{142.82} &\textbf{142.75} &0.05\% &7 &7200.0  
        &143.10 &132.69 &7.09\% &7200.0\\
        \hline
        Avg  & & &  & 2.4 &2354.4  &  &  &0.06\%   &9.5  &1659.8  
        & & &1.66\% &1795.1\\
        \hline
        \textbf{$\gamma = 0.4$}&$Obj_{1iter}$ &$LB_{1iter}$ &$LB_{1iter}\%$  &$N_{node}^{1iter}$ &CPU$_{1iter}$(s) & $Obj^{'}$ & $LB^{'}$ & $LB^{'}\%$  &$N_{node}$ & CPU$^{'}$(s) &$Obj_2$ & $LB_2$ & $LB_2\%$ &CPU(s) \\
        \hline
        u2-16 &\textbf{57.65$^*$}  & \textbf{57.65}  & 0  & 3  & 72.6  
        &\textbf{57.65$^*$} &\textbf{57.65} &0 &1 &53.2  
        &\textbf{57.65$^*$} &\textbf{57.65} &0 &25.8\\
        u2-20 &\textbf{56.34$^*$}  &\textbf{56.34}  & 0  & 5  & 793.9 
        &\textbf{56.34$^*$} &\textbf{56.34} &0 &3 &281.3  
        &\textbf{56.34$^*$} &\textbf{56.34} &0 &12.0\\
        u2-24 &\textbf{91.16$^*$}  & \textbf{91.16}  & 0  & 9  & 3304.4 
        &\textbf{91.16$^{*}$} &\textbf{91.16} &0 &25 &6917.0  
        &\textit{91.63$^*$} &\textit{91.63} &-0.52\% &757.2\\
        u3-18 &\textbf{50.74$^*$}  & \textbf{50.74}  & 0  & 1  & 47.2  
        &\textbf{50.74$^*$} &\textbf{50.74} &0 &1 &15.0  
        &\textbf{50.74$^*$} &\textbf{50.74} &0 &13.8\\
        u3-24 &\textbf{67.56$^*$}  &\textbf{67.56}  & 0  & 1  & 257.7   
        &\textbf{67.56$^*$} &\textbf{67.56} &0 &1 &58.5  
        &\textbf{67.56$^*$} &\textbf{67.56} &0 &220.8\\
        u3-30 &\textbf{76.75$^*$}  & \textbf{76.75}  & 0  & 1  & 1027.7  
        &\textbf{76.75$^*$} &\textbf{76.75} &0 &1 &500.2  
        &\textbf{76.75$^*$} &\textbf{76.75} &0 &336.6\\
        u3-36 &\textbf{104.06$^*$} &\textbf{104.06} & 0  & 1  & 4816.7  
        &\textbf{104.06$^*$} &\textbf{104.06} &0 &1 &1769.0 
        &\textbf{104.06$^*$} &\textbf{104.06} &0 &2010.0\\
        u4-16 &\textbf{53.58$^*$}  &\textbf{53.58}  & 0  & 1  & 14.8  
        &\textbf{53.58$^*$} &\textbf{53.58} &0 &1 &2.2  
        &\textbf{53.58$^*$} &\textbf{53.58} &0 &44.4 \\
        u4-24 &\textbf{89.83$^*$}  &\textbf{89.83}  & 0  & 1  & 67.9   
        &\textbf{89.83$^*$} &\textbf{89.83} &0 &1 &30.8 
        &\textbf{89.83$^*$} &\textbf{89.83} &0 &28.2\\
        u4-32 &\textbf{99.29$^*$}  &\textbf{99.29}  & 0  & 1  & 1256.7 
        &\textbf{99.29$^*$} &\textbf{99.29} &0 &1 &393.4 
        &\textbf{99.29$^*$} &\textbf{99.29} &0 &2667.6\\
        u4-40 &\textbf{133.78} & 133.62 & 0.12\%  & 32 & 7200.0 
        &\textbf{133.78$^{*}$} &\textbf{133.70} &0.06\% &32 &7200.0  
        &\textit{133.91$^*$} &\textit{133.91} &-0.10\% &2653.2\\
        u4-48 &194.97 & \textit{67.69}  & NA & 1  &7200.0  
        &\textbf{147.63} &\textbf{146.37} &0.85\% & 1 &7200.0  
        &NA &133.86 &9.33\% &7200.0 \\
        u5-40 &\textbf{121.96$^*$} & \textbf{121.96} & 0  & 1  & 2621.6  
        &\textbf{121.96$^*$} &\textbf{121.96} &0 &1 &723.7  
        &122.23 &112.58 &7.69\% &7200.0\\
        u5-50 &157.55 & \textbf{142.76} & 9.39\%  & 6  & 7200.0 
        &\textbf{142.84} &142.75 &0.06\% &1 &7200.0  
        &143.14 &134.09 &6.13\% &7200.0\\
        \hline
        Avg  & & &  &4.6  &2562.9  &  &  &0.07\%   &5.1  &2310.3 
        & & &1.70\% &2169.3\\
        \hline
        \textbf{$\gamma = 0.7$} &$Obj_{1iter}$ &$LB_{1iter}$ &$LB_{1iter}\%$  &$N_{node}^{1iter}$ &CPU$_{1iter}$(s) & $Obj^{'}$ & $LB^{'}$ & $LB^{'}\%$  &$N_{node}$ & CPU$^{'}$(s) &$Obj_2$ & $LB_2$ & $LB_2\%$ &CPU(s) \\
        \hline
        u2-16 &\textbf{59.19$^*$}   & \textbf{59.19}  & 0   & 53 & 1149.8  & \textbf{59.19$^*$}&\textbf{59.19} &0 &39 &679.8  
        &\textbf{59.19$^*$} &\textbf{59.19} &0 &338.4 \\
        u2-20 &\textbf{56.86$^*$}   &\textbf{56.86}  & 0    & 1  & 1627.6
        &\textbf{56.86$^*$} &\textbf{56.86} &0 &1 &337.0  
        &\textbf{56.86$^*$} &\textbf{56.86} &0 &72.0\\
        u2-24 &\textbf{92.01}   & \textbf{92.00} & 0.01\%    & 33 & 7200.0  
        &92.08 &91.60 &0.52\% &30 &7200.0  
        &NA &90.83 &1.36\% &7200.0 \\
        u3-18 &\textbf{50.99$^*$}   & \textbf{50.99}  & 0   & 5  & 293.5  
        &\textbf{50.99$^*$} &\textbf{50.99} &0  &3 &45.2  
        &\textbf{50.99$^*$} &\textbf{50.99} &0 &24.0\\
        u3-24 &\textbf{68.34$^*$}   &\textbf{68.34}  & 0   & 49 & 1412.9  &\textbf{68.34$^{*}$} &\textbf{68.34} &0 &27 &541.9  
        &\textit{68.39$^*$} &\textit{68.39} &-0.07\% &400.2 \\
        u3-30 &\textbf{77.41$^*$}   & \textbf{77.41}  & 0.02    & 19 & 4378.1   &\textbf{77.41$^{*}$} &\textbf{77.41} &0  &5 &1725.5 
        &\textit{78.14$^*$} &\textit{78.14} &-0.94\% &3401.4 \\
        u3-36 &\textbf{105.79}  & 105.69  & 0.09\%    & 14 & 7200.0  
        &\textbf{105.79} &\textbf{105.78} &0.01\% & 15 &7200.0  
        &\textbf{105.79} &104.37 &1.34\% & 7200.0\\
        u4-16 &\textbf{53.87$^*$}   & \textbf{53.87}  & 0    & 1  & 22.2  
        &\textbf{53.87$^*$} &\textbf{53.87} &0 &1 &4.0  
        &\textbf{53.87$^*$} &\textbf{53.87} &0 &88.8\\
        u4-24 &\textbf{89.96$^*$}   & \textbf{89.96} & 0    & 3  & 166.5  
        &\textbf{89.96$^*$} &\textbf{89.96} &0 &1 &82.3  
        &\textbf{89.96$^*$} &\textbf{89.96} &0 &22.8\\
        u4-32 &\textbf{99.50$^*$}   & \textbf{99.50}  & 0    & 1  & 1428.9   &\textbf{99.50$^*$} &\textbf{99.50} &0 &1 &1028.6 
        &\textbf{99.50$^*$} &\textbf{99.50} &0 &2827.2\\
        u4-40 &136.59  & \textbf{135.00}  & 1.16\%    & 42 & 7200.0  
        &\textbf{135.80}  &134.99 &0.22\% &43 &7200.0  
        &NA &133.01 &2.46\% &7200.0\\
        u4-48 &NA &NA &NA  &1  &7200.0  
        &\textbf{185.16} &\textit{-216.07} &NA &1 &7200.0  
        &NA &\textbf{132.49} &NA &7200.0\\
        u5-40 &\textbf{122.93}  & 122.91  & 0.02\%    & 5  & 7200.0  
        &\textbf{122.93$^{*}$} &\textbf{122.93} &0 &7 &6513.9  
        &NA &109.28 &11.88\% &7200.0\\
        u5-50 &NA &NA &NA  &1  &7200.0  &195.72 
        &\textit{132.79} &NA &1 &7200.0  
        &\textbf{144.36} &\textbf{133.33} &7.64\% &7200.0\\
        \hline
        Avg  & & &  &16.3  &3834.2  &  &  &0.09\%   &12.4  &3354.1  
        & & &1.97\% &3598.2\\
        \hline
        Summary&$\#opt$ & $\#best lb$ &$\overline{LB_{1iter}}\%$ &$\overline{N_{node}^{1iter}}$ & $\#best ub$ &$\#opt$ & $\#best lb$ &$\overline{LB'}\%$ &$\overline{N_{node}}$ &$\#best ub$  
        &$\#opt$ & $\#best lb$ & $\overline{LB_2}\%$  & $\#best ub$\\ 
        &31 & &0.10\% &7.8 &
        &32 &39 &0.07\% &9 &39
        &25 &27 &1.78\% &28\\
        \hline
    \end{tabular}
    \end{center}
\end{table}

\newpage

\section{Experimental results of allowing unlimited visits to each recharging station} \label{experiments multiple visits}
In this section, we conduct experiments on type-a, -u, and -r instances by allowing unlimited visits to each recharging station ($N_{max}^s = \infty$) under different values of $\gamma$. The corresponding results are shown in Table \ref{multiple type-a}, \ref{multiple type-u}, and \ref{multiple type-r}.

\begin{table}[!htp]
\centering
 \footnotesize
 \begin{threeparttable}
    \caption{Allowing unlimited visits to each recharging station: B\&P algorithm results on Type-a instances under $\gamma=0.4,0.7$}
    \label{multiple type-a}
    \setlength{\belowcaptionskip}{0.5cm}
    \begin{tabular}{c c c c c c c |c c}
    \hline
        \textbf{$\gamma = 0.4$}&\multicolumn{6}{c}{Unlimited charging visits} &\multicolumn{2}{c}{At-most-one visit}\\
        \hline
        Instance & $Obj$ & $LB$ &$LB\%$ &$N_{node}$ &$N_{max}^s$ &CPU(s) & BKS$'$ & BKS$'$\% \\
        \hline
        a2-16 &\textbf{237.38$^*$}  & 237.38  &0 &1   &1   & 34.6  
        &\textbf{237.38$^*$} &0  \\
        a2-20 &\textbf{280.70$^*$}  & 280.70  &0 &1   &1   & 255.4  
        &\textbf{280.70$^*$} &0 \\
        a2-24 &\textbf{346.28$^*$}  & 346.28  &0 &1   &2   & 964.6 
        &347.04$^*$ &0.21\%  \\
        a3-18 &\textbf{236.82$^*$}  & 236.82  &0 &1   &1   & 30.1  
        &\textbf{236.82$^*$} &0  \\
        a3-24 &\textbf{274.80$^*$}  & 274.80  &0 &1   &2   & 86.7
        &\textbf{274.80$^*$} &0  \\
        a3-30 &\textbf{413.28$^*$}  & 413.28  &0 &1  &3   & 742.9
        &413.34$^*$ &0.02\%  \\
        a3-36 &\textbf{481.17$^*$}  & 481.17  &0 &1   &3   & 1522.8 
        &482.75 &0.33\%  \\
        a4-16 &\textbf{222.49$^*$}  & 222.49  &0 &1   &0   & 9.6 
        &\textbf{222.49$^*$} &0  \\
        a4-24 &\textbf{311.03$^*$}  & 311.03  &0 &1   &0   & 34.5  
        &\textbf{311.03$^*$} &0  \\
        a4-32 &\textbf{394.26$^*$}  & 394.26  &0 &1   &2   & 205.0  
        &\textbf{394.26$^*$} &0\\
        a4-40 &\textbf{453.84$^*$}  & 453.84  &0 &1   &1   & 1027.1 
        &\textbf{453.84$^*$} &0 \\
        a4-48 &\textbf{554.54$^*$}  & 554.54  &0 &1   &3   & 4630.9
        &554.60$^*$ &0.01\% \\
        a5-40 &\textbf{414.51$^*$}  & 414.51 &0 &1    &2   & 539.0  
        &\textbf{414.51$^*$} &0 \\
        a5-50 &\textbf{559.48$^*$}  & 559.48 &0 &1   &2   &2661.4 
        &559.51$^*$ &0.004\% \\
        \hline
        Avg &  &  &0 &1.0 &1.6  &909.6 
        & &0.04\% \\
        \hline
        \textbf{$\gamma = 0.7$} & $Obj$ & $LB$ &$LB\%$ &$N_{node}$ &$N_{max}^s$ &CPU(s) & BKS$'$ & BKS$'$\%\\
        \hline
        a2-16 &\textbf{240.66$^*$}  &240.66 &0  &1  &1   &225.4
        &\textbf{240.66$^*$} &0 \\
        a2-20 &\textbf{285.86$^*$}  &285.86  &0 &1 &2   &469.6 
        &293.27$^*$ &2.53\% \\
        a2-24 &\textbf{350.32$^*$}  &350.32  &0 &1 &3   & 3513.6 
        &353.18$^*$ &0.81\%  \\
        a3-18  &\textbf{238.82$^*$} & 238.82 &0 &1 &3   & 62.1
        &240.58$^*$ &0.73\% \\
        a3-24 &\textbf{275.20$^*$}  & 275.20 &0 &1 &2   & 244.7
        &275.97$^*$ & 0.28\% \\
        a3-30 &\textbf{413.35$^*$}  & 413.35 &0 &1 &4   & 1556.4
        &424.93$^*$ & 2.70\%  \\
        a3-36 &\textbf{483.08$^*$}  & 483.08 &0 &1 &3   & 1254.4 
        &494.04 &2.22\% \\
        a4-16 &\textbf{222.49$^*$}  & 222.49 &0 &1 &2   & 11.0
        &223.13$^*$ &0.29\%  \\
        a4-24 &\textbf{315.40$^*$}  & 315.40 &0 &7 &2   &193.63
        &316.65$^*$ &0.39\% \\
        a4-32 &\textbf{394.94$^*$}  & 394.94 &0 &1  &4   & 427.7  
        &397.87 &0.74\%  \\
        a4-40 &\textbf{457.76$^*$} & 457.76  &0 &1 &4   & 1604.2
        &467.72$^*$ &2.13\%\\
        a4-48 &570.31  & 556.99  &1.95\%  &1  &4   & 7200.0 
        &NA &NA  \\
        a5-40 &\textbf{415.88$^*$}  & 415.88  &0 &1 &4   & 1171.2 
        &418.75$^*$ &0.69\% \\
        a5-50 &580.00  & 565.89  &2.30\% &1 &4   & 7200.0  
        &NA &NA  \\
        \hline
        Avg &  & &0.30\% &1.4 &3   &1795.3  & &1.13\% \\
        \hline
        Summary & $\#opt$ &$\#bestub$ &$\overline{LB}\%$ &$\overline{N_{node}}$ &$\overline{N_{max}^s}$  & $\overline{CPU}$(s) &$\#bestub$ &$\overline{BKS'}\%$ \\
        &26 &26 &0.15\% &1.2 &2.3  &1352.4 &10 &0.59\% \\
        \hline
    \end{tabular}
    \end{threeparttable}
\end{table}

\begin{table}[!htp]
\renewcommand\arraystretch{0.85}
\centering
 \footnotesize
 \begin{threeparttable}
    \caption{Allowing unlimited visits to each recharging station: B\&P algorithm results on Type-u instances under $\gamma=0.1,0.4,0.7$}
    \label{multiple type-u}
    \setlength{\belowcaptionskip}{0.5cm}
    \begin{tabular}{c c c c c c c|c c}
    \hline
        \textbf{$\gamma = 0.1$}&\multicolumn{6}{c}{Unlimited charging visits} &\multicolumn{2}{c}{At-most-one visit}\\
        \hline
        Instance & $Obj$ & $LB$ &$LB\%$  &$N_{node}$ &$N_{max}^s$  &CPU(s) & BKS$'$  & BKS$'$\%\\
        \hline
        u2-16 &\textbf{57.61$^*$}  &57.61   &0 &1  &1  &69.5   
        &\textbf{57.61$^*$} &0 \\
        u2-20 &\textbf{55.59$^*$}  &55.59   &0 &1 &1  &215.3
        &\textbf{55.59$^*$} &0 \\
        u2-24 &\textbf{90.66$^*$}  &90.66   &0 &1 &1  &1649.0
        &90.66$^*$ &0 \\
        u3-18 &\textbf{50.74$^*$}  &50.74   &0 &1 &0  &46.0 
        &\textbf{50.74$^*$} &0\\
        u3-24 &\textbf{67.56$^*$}  &67.56   &0 &1 &1  &109.6
        &\textbf{67.56$^*$} &0 \\
        u3-30 &\textbf{76.75$^*$}  & 76.75  &0 &1 &0  &750.8
        &\textbf{76.75$^*$} &0 \\
        u3-36 &\textbf{103.93$^*$}  &103.93   &0 &1 &2  &4326.5 
        &104.04$^*$ &0.11\% \\
        u4-16 &\textbf{53.58$^*$}  &53.58   &0 &1 &0  &6.5  
        &\textbf{53.58$^*$} &0 \\
        u4-24 &\textbf{89.83$^*$}  &89.83   &0 &1 &1  &63.2 
        &\textbf{89.83$^*$} &0\\
        u4-32 &\textbf{99.29$^*$}  &99.29   &0 &1 &1  &416.4
        &\textbf{99.29$^*$} &0\\
        u4-40 &\textbf{133.11$^*$}  &133.11   & 0  &1  &1    &2586.1
        &\textbf{133.11$^*$} &0\\
        u4-48 &\textbf{147.33}  &146.74   &0.40\% &1 &2  &7200.0
        &148.08 &0.51\% \\
        u5-40 &\textbf{121.86$^*$}  &121.86   &0 &1 &1  &1591.2  
        &\textbf{121.86$^*$}  &0\\
        u5-50 &\textbf{142.82}  &142.75   &0.05\% &1 &1  &7200.0  
        &142.82 &0\\
        \hline
        Avg &  &  &0.03\% &1.0 &0.93  &2270.1 & &0.04\%\\
        \hline
        \textbf{$\gamma = 0.4$} & $Obj$ & $LB$ &$LB\%$ &$N_{node}$ &$N_{max}^s$ &CPU(s) & BKS$'$  & BKS$'$\%\\
        \hline
        u2-16 &\textbf{57.65$^*$}  &57.65   &0 &1 &1  &35.7     &\textbf{57.65$^*$} &0 \\
        u2-20 &\textbf{56.34$^*$}  &56.34   &0 &1 &1  &329.9     &\textbf{56.34$^*$} &0\\
        u2-24 &\textbf{90.84$^*$}  &90.84   &0 &1 &2  &2408.5   &91.27 &0.47\%\\
        u3-18 &\textbf{50.74$^*$}  &50.74  &0  &1 &1  &62.6      &\textbf{50.74$^*$} &0 \\
        u3-24 &\textbf{67.56$^*$}  &67.56   &0 &1 &1   &141.0    &\textbf{67.56$^*$} & 0\\
        u3-30 &\textbf{76.75$^*$}  &76.75   &0 &1 &1  &1457.8   &\textbf{76.75$^*$} & 0\\
        u3-36 &\textbf{104.06$^*$}  &104.06  &0 &1 &1  &3729.6 &\textbf{104.06$^*$} &0\\
        u4-16 &\textbf{53.58$^*$}  &53.58   &0  &1 &0  &7.1      &\textbf{53.58$^*$} &0 \\
        u4-24 &\textbf{89.83$^*$}  &89.83   &0 &1  &1   &39.0     &\textbf{89.83$^*$} &0\\
        u4-32 &\textbf{99.29$^*$}  &99.29   &0 &1  &1   &438.3    &\textbf{99.29$^*$} &0\\
        u4-40 &\textbf{133.36$^*$}  &133.36   &0 &1 &2    &1581.5 &133.78 &0.31\% \\
        u4-48 &\textbf{147.56}  &146.97   &0.40\% &1 &2   &7200.0   &147.63 &0.05\%\\
        u5-40 &\textbf{121.96$^*$} &121.96   &0 &1  &1   &1073.8  &\textbf{121.96$^*$} &0 \\
        u5-50 &\textbf{142.83$^*$}  &142.83   &0 &1  &1    &6587.9 &142.84 &0.007\%\\
        \hline
        Avg &  & &0.03\% &1.0 &1.14   &1927.5 & &0.06\%\\
        \hline
        \textbf{$\gamma = 0.7$} & $Obj$ & $LB$ &$LB\%$ &$N_{node}$ &$N_{max}^s$ &CPU(s) & BKS$'$  & BKS$'$\% \\
        \hline
        u2-16 &\textbf{58.17$^*$}  &58.17   &0 &1 &2  &172.2    &59.19$^*$  &1.72\%\\
        u2-20 &\textbf{56.86$^*$}  &56.86   &0 &1 &1  &397.0    &\textbf{56.86$^*$}  &0\\
        u2-24 &\textbf{91.33$^*$}  &91.33   &0 &1 &2  &5250.1   &92.17  &0.91\%\\
        u3-18  &\textbf{50.99$^*$}  &50.99  &0 &1 &1  &47.4     &\textbf{50.99$^*$}  &0\\
        u3-24 &\textbf{68.06$^*$}  &68.06   &0 &1 &2  &558.9    &68.44  &0.48\%\\
        u3-30 &\textbf{77.29$^*$}  &77.29   &0 &1 &2  &807.9    &77.41$^*$  &0.16\%\\
        u3-36 &106.72  &104.85   &1.07\% &1 &2  &7200.0    &\textbf{106.50}  &-0.88\%\\
        u4-16 &\textbf{53.87$^*$}  &53.87   &0 &1 &1  &16.7     &\textbf{53.87$^*$}  &0\\
        u4-24 &\textbf{89.83$^*$}  &89.83   &0 &1 &2  &74.6     &\textbf{89.96$^*$}  &0.14\%\\
        u4-32 &\textbf{99.50$^*$}  &99.50   &0 &1 &1  &1600.3   &\textbf{99.50$^*$}  &0\\
        u4-40 &\textbf{134.38}  &134.16   &0.16\% &45  &3  &7200.0  &135.29  &0.67\%\\
        u4-48 &152.72  &145.99   &2.35\% &1 &2  &7200.0   &185.16  &17.52\%\\
        u5-40 &\textbf{123.00}  &122.72   &0.23\% &1 &1  &7200.0   &123.82  &0.66\%\\
        u5-50 &\textbf{142.89$^*$}  &142.89   &0 &1 &2  &5628.9 &144.36  &1.02\%\\
        \hline
        Avg &  & &0.27\% &4.1 &1.71  &3096.7 & &1.60\%\\
        \hline
        Summary & $\#opt$ &$\#best ub$ &$\overline{LB}\%$ &$\overline{N_{node}}$ &$\overline{N_{max}^s}$ & $\overline{CPU}$(s) &$\#bestub$ &$\overline{BKS}\%$\\
        &35 &40 &0.11\% &2.0 &1.3   &2431.4  &26 &0.57\%\\
        \hline
    \end{tabular}
    \end{threeparttable}
\end{table}

\begin{table}[!htp]
\renewcommand\arraystretch{0.85}
\centering
 \footnotesize
 \begin{threeparttable}
    \caption{Allowing unlimited visits to each recharging station: B\&P algorithm results on Type-r instances under $\gamma=0.1,0.4,0.7$}
    \label{multiple type-r}
    \setlength{\belowcaptionskip}{0.5cm}
    \begin{tabular}{c c c c c c c| c c}
    \hline
        \textbf{$\gamma = 0.1$}&\multicolumn{6}{c}{Unlimited charging visits} &\multicolumn{2}{c}{At-most-one visit}\\
        \hline
        Instance & $Obj$ & $LB$ &$LB\%$ &$N_{node}$ &$N_{max}^s$ &CPU(s) & BKS$'$  & BKS$'$\%\\
        \hline
        r5-60 &\textbf{683.39$^*$} &683.39 &0 &1 &0  &13941.5  &\textbf{683.39$^*$} &0\\
        r6-48 &\textbf{506.45$^*$} &506.45 &0 &1 &0  &1675.1   &\textbf{506.45$^*$} &0\\
        r6-60 &\textbf{689.45$^*$} &689.45 &0 &9 &3  &3305.3  &\textbf{689.45$^*$} &0 \\
        r6-72 &761.91 &740.87 &2.76\%  &1 &1  &18000.0    &\textbf{761.34} &-0.07\%\\
        r7-56 &\textbf{612.02$^*$} &612.02 &0 &1  &1  &2452.9   &\textbf{612.02$^*$} &0\\
        r7-70 &\textbf{754.28$^*$} &754.28 &0 &21 &1  &10204.4  &\textbf{754.28} &0\\
        r7-84 &\textbf{870.21} &861.39 &1.01\% &1  &1  &18000.0    &889.38 &2.16\%\\
        r8-64 &\textbf{632.22$^*$} &632.22 &0 &1 &0  &4005.4   &\textbf{632.22$^*$} &0\\
        r8-80 &\textbf{788.99$^*$} &788.99 &0 &1 &1  &4883.0  &\textbf{788.99} &0\\
        r8-96 &1098.12 &\textit{807.04} &NA &1 &1  &18000.0  &\textbf{1053.11} &-4.27\%\\
        \hline
        Avg & & &0.42\% &3.8 &0.9  &9446.8 & &-0.22\%\\
        \hline
        \textbf{$\gamma = 0.4$} & $Obj$ & $LB$ &$LB\%$ &$N_{node}$ & $N_{max}^s$ &CPU(s) & BKS$'$  & BKS$'$\%\\
        \hline
        r5-60 &\textbf{683.58$^*$} &683.58 &0 &1 &3  &14205.2  &\textbf{684.49$^*$} &0.13\%\\
        r6-48 &\textbf{506.45$^*$} &506.45 &0 &1 &0  &1225.8   &\textbf{506.45$^*$} &0\\
        r6-60 &\textbf{689.46$^*$} &689.46 &0 &1 &2  &5133.0   &\textbf{689.46$^*$} &0 \\
        r6-72 &\textbf{762.65} &756.44 &0.81\% &1 &3  &18000.0    &NA &NA\\
        r7-56 &\textbf{612.02$^*$} &612.02 &0 &3 &1  &846.1  &\textbf{612.02$^*$} &0\\
        r7-70 &\textbf{754.28$^*$} &754.28 &0 &19 &1  &10630.8 &\textbf{754.28$^*$} &0\\
        r7-84 &\textbf{891.44} &864.14 &3.06\% &1 &3  &18000.0 &1081.76 &17.59\%\\
        r8-64 &\textbf{632.22$^*$} &632.22 &0 &1 &1  &5757.1 &\textbf{632.22$^*$} &0\\
        r8-80 &\textbf{788.99$^*$} &788.99 &0 &1 &4  &11757.7 &\textbf{788.99} &0\\
        r8-96 &\textbf{1144.15} &\textit{881.91} &NA &1  &3  &18000.0  &NA &NA\\
        \hline
        Avg & & &0.39\% &3.0 &2.1  &10355.6 & &1.77\%\\
        \hline
        \textbf{$\gamma = 0.7$} & $Obj$ & $LB$ &$LB\%$ &$N_{node}$ &$N_{max}^s$ &CPU(s) & BKS$'$  & BKS$'$\%\\
        \hline
        r5-60 &\textbf{688.84} &684.03 &0.70\% &1 &5  &18000.0  &NA &NA\\
        r6-48 &\textbf{508.10$^*$} &508.10 &0 &1 &4  &2027.7  &NA &NA\\
        r6-60 &\textbf{689.55$^*$} &689.55 &0 &1 &7  &6604.7 &NA &NA \\
        r6-72 &\textbf{788.83} &749.70 &4.96\% &1 &4  &18000.0  &NA &NA\\
        r7-56 &\textbf{616.16$^*$} &616.16 &0 &59 &7  &14247.4 &NA &NA\\
        r7-70 &\textbf{765.49} &754.14 &1.48\% &1 &6  &18000.0  &NA &NA\\
        r7-84 &\textbf{1004.81} & \textit{787.22}&NA &1 &4  &18000.0  &NA &NA\\
        r8-64 &\textbf{632.61$^*$} &632.61 &0 &1 &6  &4962.2  &NA &NA\\
        r8-80 &\textbf{794.41} &794.41 &0 &1 &8  &15051.3 &NA &NA\\
        r8-96 &\textbf{1174.49} &\textit{755.46} &NA &1 &4  &18000.0 &NA &NA\\
        \hline
        Avg & & &0.81\%  &6.8 &5.5  &13289.3 & &NA\\
        \hline
        Summary & $\#opt$ &$\#best ub$ &$\overline{LB}\%$ &$\overline{N_{node}}$  &$\overline{N_{max}^s}$ & $\overline{CPU}$(s) & $\#best ub$ &$\overline{BKS}\%$\\
        &18 &28 &0.54\% &4.5 &2.8 &11030.6 &14 &0.79\%\\
        \hline
    \end{tabular}
    \end{threeparttable}
\end{table}

\newpage

\section{Supplementary Material} \label{supplementary document}
Several solutions are found by our B\&P algorithm that are strictly better than the optimal solutions reported in \cite{bongiovanni2019electric}. In this section, we analyze why the model proposed in the original paper leads to incorrect results. To facilitate illustration, we take our obtained solution of instance a2-24-0.4 as an example, where we get an optimal solution value of \textbf{347.04} while the reportedly optimal solution is \textbf{348.03}. The solution gap is higher than the 0.01\% default tolerance gap of Gurobi. To find the conflicting constraints in the MIP model, we set the binary variables $x_{i,j}^k$ to the solution we found as constraints and use ``compute conflict" in Julia and print the conflicting constraints.

The optimal solution we obtained is:

(1) route 1: $[51,7,31,11,35,10,34,56,5,29,4,21,20,28,45,44,1,25,12,8,36,32,54]$

(2) route 2: 
$[52,17,41,19,43,22,46,18,2,42,26,15,39,6,30,16,13,40,37,23,47,14,38,57,3,27,24,9,48,33,55,53]$

\subsection{Error: Incorrect Value for Big ``M"}
The problematic constraints in \cite{bongiovanni2019electric} are:
\begin{equation} \label{RT1}
    E_s^k \leqslant T_s^k - t_{i,s} - T_i^k + M_{i,s}^k\left(1-x_{i,s}^{k}\right), \quad \forall s \in S, i \in D\cup{S}\cup{O_k}, k\in K, i \ne s
\end{equation}

\begin{equation} \label{RT2}
    E_s^k \geqslant T_s^k - t_{i,s} - T_i^k - M_{i,s}^k\left(1-x_{i,s}^{k}\right), \quad \forall s \in S, i \in D\cup{S}\cup{O_k}, k\in K, i \ne s
\end{equation}
where $E_s^k$ are the decision variables indicating the recharging time at recharging station $s$ for vehicle $k$. $T_i^k$ and $T_s^k$ are the decision variables indicating the time at which vehicle $k$ starts its service at location $i$ and $s$, respectively.

In their real implementation, constraints (\ref{RT1}) and (\ref{RT2}) are implemented as constraints (\ref{RT1_1}) and (\ref{RT2_1}). Indeed, as the time windows on the recharging visits are not binding, we need to consider the path from $s$ to $i$ to restrict the recharging time at station $s$. The recharging time constraints are formulated as follows:
\begin{equation} \label{RT1_1}
    E_s^k \leqslant T_i^k - t_{s,i} - T_s^k + M_{s,i}^k\left(1-x_{s,i}^{k}\right), \quad \forall s \in S, i \in P\cup{S}\cup{F}, k\in K, i \ne s
\end{equation}

\begin{equation} \label{RT2_1}
    E_s^k \geqslant T_i^k - t_{s,i} - T_s^k - M_{s,i}^k\left(1-x_{s,i}^{k}\right), \quad \forall s \in S, i \in P\cup{S}\cup{F}, k\in K, i \ne s
\end{equation}

That is, we have the following constraints always hold:
\begin{eqnarray} \label{RT*}
    &T_i^k - t_{s,i} - M_{s,i}^k(1-x_{s,i}^{k})\leqslant T_s^k+ E_s^k \leqslant  T_j^k - t_{s,j} + M_{s,j}^k(1-x_{s,j}^{k}), \\ &\forall s \in S, i \in P\cup{S}\cup{F}, k\in K, i \ne s \nonumber
\end{eqnarray}

However, the value of big ``M" used in \cite{bongiovanni2019electric} is: $M_{i,j}^k=max\{0,l_i+s_i+t_{i,j}-e_j\}$. With this value of big ``M", it will not hold for some cases. We take an example with our obtained solution of a2-24-0.4, where $x_{56,24}^1 = 0$, $x_{56,5}^1 = 1$, $M_{56,24}^1 = 127.20$, $M_{56,5}^1 = 445.32$, $t_{56,24} = 15.86$, $t_{56,5} = 4.98$. As the earliest time window at node 24 is 603.0, the latest time window at node 5 is 302.85, the time window constraints on node 24 and node 5 are:


\begin{eqnarray} \label{tw at 24 and 5}
    &T_{24}^1 \geqslant 603.0, \quad T_{5}^1 \leqslant 302.85  
\end{eqnarray}

Recharging time constraints are: 
\begin{equation} \label{RT1 at 56}
       T_{24}^1 - t_{56,24} - M_{56,24}^1(1-x_{56,24}^1) \leqslant T_{56}^1+ E_{56}^1
\end{equation}

\begin{equation} \label{RT2 at 56}
      T_{56}^1+ E_{56}^1 \leqslant T_{5}^1 - t_{56,5}+M_{56,5}^1(1-x_{56,5}^1)
\end{equation}

Introducing constraints (\ref{tw at 24 and 5}) and $x_{56,24}^1 = 0$, $x_{56,5}^1 = 1$, $M_{56,24}^1 = 127.20$, $M_{56,5}^1 = 445.32$,  $t_{56,24} = 15.86$, $t_{56,5} = 4.98$ to constraint (\ref{RT1 at 56}) and (\ref{RT2 at 56}), we have:

\begin{equation}
    459.94 \leqslant T_{56}^1+ E_{56}^1 \leqslant 297.87
\end{equation}

which is a contradiction!

To obtain the correct big ``M" value, we calculate a lower bound for big ``M" parameters from constraints (\ref{RT1_1}) as follows:

\begin{equation}
    E_s^k - T_i^k + t_{s,i} + T_s^k \leqslant  M_{s,i}^k
\end{equation}

The maximum value of the left-hand side is obtained when $E_s^k = \frac{Q}{\alpha}$, $T_i^k = e_i$, $T_s^k = T_p$, where $T_p$ is the planning horizon. We obtain:

\begin{equation}
    \frac{Q}{\alpha} - e_i + t_{s,i} + T_p \leqslant  M_{s,i}^k \longrightarrow M_{s,i}^k = \min\{\frac{Q}{\alpha} - e_i + t_{s,i} + T_p, T_p\}
\end{equation}

We take $T_p$ directly as the value of big ``M" and the contradiction is solved.

\subsection{Typo: Incorrect Value of Big ``G"}
Furthermore, there is also a typo related to the value of big ``G" in \cite{bongiovanni2019electric} with the following constraints:

\begin{equation} \label{big G problem constraint}
    L_i^k + l_j + G_{i,j}^k\left(1-x_{i,j}^{k}\right) \geqslant L_j^k, \quad \forall i \in V \setminus F, j \in V \setminus O_k,i \neq j, k \in K
\end{equation}

The big ``G" values used in \cite{bongiovanni2019electric} is $G_{i,j}^k = \min\{C^k,C^k+l_i\}$. With this value, the solutions that have continuously visited $C^k$ times drop-off nodes are cut off. For example, our obtained solution has three continuous visits of drop-off nodes: 28-$>$45-$>$44, and we have $x_{28,45}^1 = x_{45,44}^1 = 1.0$ and $x_{44,28,1} = 0.0$. As these nodes are drop-off nodes, the loads on these nodes are negative:  $l_{45} = l_{44} = l_{28} = -1.0$. The maximal vehicle capacity $C^k (k=1)$ equals three. Based on constraints  (\ref{big G problem constraint}), we have:

\begin{equation} \label{conflict1}
    L_{28}^1 + l_{45} \geqslant L_{45}^{1}
\end{equation}

\begin{equation} \label{conflict2}
    L_{45}^1 + l_{44} \geqslant L_{44}^{1}
\end{equation}

\begin{equation} \label{conflict3}
    L_{44}^1 + l_{28} + G_{44,28}^1 \geqslant L_{28}^{1}
\end{equation}

However, with the defined value in \cite{bongiovanni2019electric}, $G_{44,28}^1$ must satisfy:

\begin{equation} \label{conflict4}
    G_{44,28}^1 \leqslant C^1 + l_{44}
\end{equation}

Introducing constraint (\ref{conflict2}), (\ref{conflict3}), and (\ref{conflict4}) into constraints (\ref{conflict1}) leads to contradiction:
\begin{equation}
    -1.0 = C^1 + l_{44} + l_{45} + l_{28} \geqslant 0
\end{equation}


The correct value is $G_{i,j}^k = \max\{C^k,C^k+l_i\}$. 

\end{document}